\documentclass[oneside,english]{amsart}
\usepackage[T1]{fontenc}
\usepackage[latin9]{inputenc}
\usepackage{geometry}
\usepackage{amsbsy,verbatim}
\usepackage{amstext,txfonts}
\usepackage{amsthm}
\usepackage{amssymb}
\usepackage{scalerel,stackengine}
\stackMath
\newcommand\reallywidehat[1]{%
\savestack{\tmpbox}{\stretchto{%
  \scaleto{%
    \scalerel*[\widthof{\ensuremath{#1}}]{\kern-.6pt\bigwedge\kern-.6pt}%
    {\rule[-\textheight/2]{1ex}{\textheight}}
  }{\textheight}%
}{0.5ex}}%
\stackon[1pt]{#1}{\tmpbox}%
}

\makeatletter
\numberwithin{equation}{section}
\numberwithin{figure}{section}
\theoremstyle{plain}
\newtheorem{thm}{\protect\theoremname}[section]
\theoremstyle{plain}
\newtheorem{cor}[thm]{\protect\corollaryname}
\theoremstyle{plain}
\newtheorem{lem}[thm]{\protect\lemmaname}
\ifx\proof\undefined
\newenvironment{proof}[1][\protect\proofname]{\par
	\normalfont\topsep6\p@\@plus6\p@\relax
	\trivlist
	\itemindent\parindent
	\item[\hskip\labelsep\scshape #1]\ignorespaces
}{%
	\endtrivlist\@endpefalse
}
\providecommand{\proofname}{Proof}
\fi
\theoremstyle{plain}
\newtheorem{prop}[thm]{\protect\propositionname}
\theoremstyle{plain}
\newtheorem{hyp}[thm]{\protect\hypothesisname}
\theoremstyle{plain}
\newtheorem*{prop*}{\protect\propositionname}
\theoremstyle{definition}
\newtheorem{defn}[thm]{\protect\definitionname}
\theoremstyle{plain}
\newtheorem*{lem*}{\protect\lemmaname}
\theoremstyle{remark}
\newtheorem{claim}[thm]{\protect\claimname}

\usepackage{babel}

\makeatother

\usepackage{babel}
\providecommand{\claimname}{Claim}
\providecommand{\corollaryname}{Corollary}
\providecommand{\definitionname}{Definition}
\providecommand{\lemmaname}{Lemma}
\providecommand{\propositionname}{Proposition}
\providecommand{\theoremname}{Theorem}
\providecommand{\hypothesisname}{Hypothesis}
\newcommand{\remove}[1]{}

\begin{document}
\global\long\def\connected{\text{highly connected}}%
 
\global\long\def\f{\mathcal{F}}%
 
\global\long\def\a{\mathcal{A}}%
 
\global\long\def\pn{\mathcal{P}\left(\left[n\right]\right)}%
 
\global\long\def\g{\mathcal{G}}%

\global\long\def\Hom{\mathrm{Hom}}%
 
\global\long\def\l{\mathcal{L}}%
 
\global\long\def\s{\mathcal{S}}%
 
\global\long\def\j{\mathcal{J}}%
 
\global\long\def\d{\mathcal{D}}%
 
\global\long\def\Cay{\mathrm{Cay}}%

\global\long\def\OPT{\mathrm{OPT}}
 
\global\long\def\Image{\mathrm{Im}}%

\global\long\def\supp{\mathrm{supp}}
 
\global\long\def\GL{\mathrm{GL}}%
 
\global\long\def\SL{\mathrm{SL}}%
 
\global\long\def\Inf{}%
 
\global\long\def\Id{\textrm{Id}}%
 
\global\long\def\Tr{\mathrm{Tr}}%
 
\global\long\def\sgn{\textrm{sgn}}%
 
\global\long\def\p{\mathcal{P}}%
 
\global\long\def\h{\mathcal{H}}%
 
\global\long\def\n{\mathbb{N}}%
 
\global\long\def\a{\mathcal{A}}%
 
\global\long\def\b{\mathcal{B}}%
 
\global\long\def\c{\mathcal{C}}%
 
\global\long\def\e{\mathbb{E}}%
 
\global\long\def\x{\mathbf{x}}%
 
\global\long\def\y{\mathbf{y}}%
 
\global\long\def\z{\mathbf{z}}%
 
\global\long\def\c{\mathcal{C}}%
 
\global\long\def\av{\mathsf{A}}%
 
\global\long\def\chop{\mathrm{Chop}}%
 
\global\long\def\stab{\mathrm{Stab}}%
 
\global\long\def\Span{\mathrm{Span}}%
 
\global\long\def\Domain{\mathrm{Domain}}%
 
\global\long\def\codim{\mathrm{codim}}%
 
\global\long\def\Var{\mathrm{Var}}%
 
\global\long\def\rank{\mathrm{rank}}%
 
\global\long\def\t{\mathsf{T}}%
 
\global\long\def\sqbinom#1#2{\left[\begin{array}{c}
#1\\
#2
\end{array}\right]}%

\title[An analogue of Bonami's Lemma for functions on spaces of linear maps, and 2-2 Games]{An analogue of Bonami's Lemma for functions on spaces of linear maps, and 2-2 Games}
\author{David Ellis}
\address{School of Mathematics, University of Bristol, United Kingdom.}
\email{david.ellis@bristol.ac.uk}
\author{Guy Kindler}
\address{Rachel and Selim Benin School of Computer Science and Engineering,
Hebrew University of Jerusalem, Edmond J. Safra Campus, Givat Ram,
Jerusalem 91904, Israel.}
\email{gkindler@cs.huji.ac.il}
\author{Noam Lifshitz}
\address{Einstein Institute of Mathematics, Hebrew University of Jerusalem,
Edmond J. Safra Campus, Givat Ram, Jerusalem 91904, Israel.}
\email{noamlifshitz@gmail.com}
\begin{abstract}
We prove an analogue of Bonami's (hypercontractive) lemma for complex-valued functions on
$\mathcal{L}(V,W)$, where $V$ and $W$ are vector spaces over a
finite field. This inequality is useful for functions on $\mathcal{L}(V,W)$
whose `generalised influences' are small, in an appropriate sense.
It leads to a significant shortening of the proof of a recent seminal
result by Khot, Minzer and Safra \cite{kms} that pseudorandom sets
in Grassmann graphs have near-perfect expansion, which (in combination
with the work of Dinur, Khot, Kindler, Minzer and Safra \cite{dkkms})
implies the 2-2 Games conjecture (the variant, that is, with imperfect completeness).
\end{abstract}

\maketitle

\section{Introduction}
Hypercontractive inequalities  are of great importance and use in mathematical physics, analysis, geometry, probability theory, combinatorics and theoretical computer science (having first been introduced by Nelson \cite{Nelson}, motivated by mathematical physics). In general, for $1 \leq p < q \leq \infty$, a $(p,q)$-hypercontractive inequality for a measure space $X$ and an operator $T:L^p(X) \to L^q(X)$ says that $\|T(f)\|_q \leq \|f\|_p$ for all $f \in L^p(X)$.  One of the most classical, fundamental and useful hypercontractive inequalities is the hypercontractive inequality of Bonami, Beckner and Gross regarding the noise operator on the discrete cube, with the uniform measure. Let us give the statement in full. For $0 \leq \rho \leq 1$, the noise operator $T_{\rho}:L^p(\{0,1\}^n) \to L^q(\{0,1\}^n)$ is defined by
$$(T_{\rho}f)(x) = \mathbb{E}_{y \sim N_{\rho}(x)}[f(y)]\quad \forall x \in \{0,1\}^n,\ \forall f:\{0,1\}^n \to \mathbb{R},$$
where the distribution $y \sim N_\rho(x)$ is defined as follows: independently for each coordinate $i \in [n]$, we set $y_i = x_i$ with probability $\rho$, and with probability $1-\rho$ we take $y_i \in \{0,1\}$ uniformly at random (independently of $x_i$). In other words, we obtain $y$ from $x$ by resampling each coordinate of $x$ independently with probability $1-\rho$, so $y$ is a `noisy' version of $x$. Note that $T_1(f)=f$, i.e.\ $T_1$ is simply the identity operator; on the other hand, $T_0$ maps a function $f$ to the constant function with value $\mathbb{E}[f]$. For $0 < \rho < 1$, $T_{\rho}$ interpolates between these two extremes: the smaller the value of $\rho$, the greater the degree of `smoothing'.

The hypercontractive inequality of Bonami \cite{bonami}, Beckner \cite{beckner} and Gross \cite{gross-ineq}\footnote{It was discovered independently by these three authors, though Bonami considered only the case $p=2$, which suffices for most applications.} states that 
$$\|T_{\rho}(f)\|_q \leq \|f\|_p\quad \forall \rho \leq \sqrt{(p-1)/(q-1)},\ \forall f:\{0,1\}^n \to \mathbb{R}.$$
As the spectral norm of $T_{\rho}$ is $1$, this inequality means that it acts as a smoothing operator, smoothing out sharp peaks.

Often, the special case with $q=4$ and $p=2$ suffices for applications; this says that
\begin{equation}
    \label{eq:42hyp}
\|T_{\rho}(f)\|_4 \leq \|f\|_2 \quad \forall \rho \leq 1/\sqrt{3}.
\end{equation}
$T_{\rho}$ can also be written in terms of the Fourier transform, writing $f:\{0,1\}^n \to \mathbb{R}$ as $f = \sum_{S \subseteq [n]} \hat{f}(S)\chi_S$, where $\chi_S(x) = (-1)^{\sum_{i\in S}x_i}$ for all $x \in \{0,1\}^n$ and $S \subseteq [n]$ (here, $\hat{f}(S) = \langle f,\chi_S\rangle$ for all $S \subseteq [n]$), the noise operator $T_{\rho}$ is given by
$$T_{\rho}(f) = \sum_{S \subset [n]} \rho^{|S|} \hat{f}(S)\chi_S.$$
This yields the following corollary of (\ref{eq:42hyp}), known as Bonami's lemma, which is extremely useful.
\begin{lem}[Bonami's Lemma]
\label{lem:bonami}
Let $f:\{0,1\}^n \to \mathbb{R}$ be a function of degree at most $d$; then
$$\|f\|_4 \leq 3^{d/2}\|f\|_2.$$
\end{lem}
(Recall that the {\em degree} of a function $f:\{0,1\}^n \to \mathbb{R}$ is the maximal size of a set $S$ such that $\hat{f}(S) \neq 0$.) Bonami's lemma bounds the 4-norm of a low-degree function in terms of its 2-norm; roughly speaking, it says that low-degree functions on $\{0,1\}^n$ do not have very large `peaks' in their modulus (such peaks would lead to their having large 4-norm). 

The Bonami-Beckner-Gross hypercontractive inequality was a crucial ingredient in the proof of the seminal Kahn-Kalai-Linial theorem \cite{kkl} on the influences of Boolean functions, and of Friedgut's junta theorem \cite{friedgut-junta}; both have been of huge importance in combinatorics and theoretical computer science over the last three decades. (In fact, Bonami's lemma suffices for these two applications.) 

The notion of a `noise' operator (which we defined above for the discrete cube) readily generalises to $L^p(X,\mu)$ for many other measure spaces $(X,\mu)$: one just needs to find a (natural) way to resample the input of a function (resampling in a way that is more or less `extreme', depending on the noise parameter). For example, the noise operator $U_{\rho}:L^p(\mathbb{R}^n,\gamma^n) \to L^q(\mathbb{R}^n,\gamma^n)$ on $n$-dimensional Gaussian space (with the standard $n$-dimensional Gaussian measure $\gamma^n$) is defined by
$$U_{\rho}(f)(x) = \mathbb{E}_{Y \sim \gamma^n}\left[f\left(\rho x+ \sqrt{1-\rho^2}Y\right)\right]\quad \forall f \in L^p(\mathbb{R}^n, \gamma^n);$$
this is natural because if $X$ and $Y$ are independent $n$-dimensional standard Gaussian random variables, then $X$ and $\rho X + \sqrt{1-\rho^2}Y$ are  $\rho$-correlated standard $n$-dimensional Gaussians. Here, the `noisy' version of $x$ is the random variable $\mathbb{N}_\rho(x): = \rho x+\sqrt{1-\rho^2}Y$, where $Y \sim \gamma^n$.

Hypercontractive inequalities for natural `noise' operators on many other spaces have been obtained over the last five decades. A very useful example is the hypercontractive inequality for the noise operator in Gaussian space \cite{beckner,gross-ineq,Nelson2} (an earlier suboptimal version appeared in~\cite{Nelson}); this is intimately related to the heat equation. Rothaus proved a sharp hypercontractive inequality \cite{rothaus} for the $n$-dimensional sphere $S^n$. Gross \cite{gross-ineq} proved that a hypercontractive inequality for a space is equivalent to a log-Sobolev inequality for that space, linking two important bodies of work, and proved hypercontractive inequalities over some noncommutative algebras related to quantum field theory. The hypercontractive inequality for the noise operator in Gaussian space was a crucial ingredient in the proof of the seminal Invariance Theorem of Mossel, O'Donnell and Oleszkiewicz \cite{moo}.

The hypercontractive inequalities we have discussed above hold for {\em all} functions on the corresponding space. For some important examples of spaces, however, a (strong) hypercontractive inequality does not hold for all functions --- even an analogue of Bonami's lemma does not hold, since there are `badly-behaved' low-degree functions whose 4-norm is large compared to their 2-norm. This is the case for functions on the $p$-biased cube $(\{0,1\}^n,\mu_p)$, where $p=o(1)$: the `dictatorship' functions defined by $f(x) = x_i$ for some $i$ have 4-norm $p^{1/4}$, which is much greater than their 2-norm $p^{1/2}$, when $p=o(1)$. (Recall that the $p$-biased measure on $\{0,1\}^n$ is defined by $\mu_p(x) =  p^{\sum_{i=1}^{n}x_i}(1-p)^{n-\sum_{i=1}^{n}x_i}$. A `weak' analogue of Bonami's lemma holds for the $p$-biased measure, but with $\sqrt{3}$ replaced by function of $p$ which tends to infinity as $p$ tends to zero.  This weak analogue is insufficient for many important applications.) 

Recently, Keevash, Lifshitz, Long and Minzer \cite{kllm} proved that `dictatorships' and similar `junta-type' constructions are in a sense the only barrier to hypercontractivity: for functions whose norm is not too much affected by restricting the values of a small set of coordinates, a hypercontractive inequality does hold. (Keevash, Lifshitz, Long and Minzer called such functions {\em global functions}.) A hypercontractive inequality for global functions may be termed a {\em conditional hypercontractive inequality} (the precise quantitative notion of `global' may differ according to the context or the application in mind) --- the classical Bonami-Beckner-Gross inequality and its Gaussian analogue, on the other hand, are unconditional (the hypercontractive inequality there holds for all functions, not just global ones). In \cite{kllm}, Keevash, Lifshitz, Long and Minzer obtained (conditional) hypercontractive inequalities (for the natural noise operator) for global functions on both the $p$-biased cube $(\{0,1\}^n,\mu_p)$ and for a general product space $(X^n,\mu^n)$; these had several important applications in extremal combinatorics and theoretical computer science (see e.g.\ \cite{kllm,global-app}). In \cite{fklm}, Filmus, Kindler, Lifshitz and Minzer obtained a (conditional) hypercontractive inequality for global functions on the symmetric group $S_n$, a non-product space (in the case of the symmetric group, again, a `strong' hypercontractivity does not hold for all functions, as one can see by considering the indicator function of a point-stabilizer); this in turn was a crucial ingredient in the resolution by Keevash, Lifshitz and Minzer \cite{an} of a well-known open problem of Crane, concerning the largest product-free sets in the alternating groups $A_n$.

One important family of applications of hypercontractive inequalities (both unconditional and conditional hypercontractive inequalities) is to obtain {\em small-set expansion} theorems. A small-set expansion theorem for a finite, regular undirected graph $G$ says, roughly speaking, that small sets\footnote{Possibly, provided they satisfy an additional `globalness' or `psuedorandomness' condition, such as having no large density increment on a `nice' subset.} have very large vertex-boundary in $G$, much larger than the bound guaranteed by the Cheeger constant\footnote{Or, which is roughly equivalent, the second eigenvalue of the graph.}, the latter bound being sharp only for larger (or non-pseudorandom) sets. More precisely, a small-set expansion theorem for $G = (V,E)$ says that if $S \subset V(G)$ with $|S|$ small (and, possibly, satisfies an additional globalness or psuedorandomness condition), then choosing a uniform random element $u$ of $S$ and a random edge $uv$ of $G$ incident with $u$, the vertex $v$ (at the other end of the random edge) will lie outside $S$ with probability close to 1. There is a similar notion for weighted graphs, where the edges are weighted with non-negative weights and the weighting is regular (meaning that the sum of the weights of edges incident to each vertex is the same): in this case, the random edge $uv$ is chosen with probability proportional to the weight of the edge $uv$.

A hypercontractive inequality can often be used to prove a small-set expansion theorem, as we shall now roughly outline. First, given a graph $G$ on a probability space $(X,\mu)$, one finds a noise operator $T_\rho$ defined by $T_{\rho}f(x) = \mathbb{E}_{y \sim N_\rho(x)}[f(y)]$, such that $T_{\rho}$ satisfies a hypercontractive inequality, and such that the `noised' version $N_\rho(x)$ of $x$ is concentrated on close neighbours of $x$ in $G$ (i.e., on vertices of $G$ with small graph-distance from $x$). This means, roughly, that $T_{\rho}f(x)$ is an average value of $f(y)$ over vertices $y$ that are `close neighbors' of $x$. This in turn means that if $f$ is the indicator function of a set $S$, then the inner product $\langle T_{\rho}f,f\rangle$ is roughly (or sometimes, exactly) proportional to the probability that if we choose a uniform random vertex $u$ in $S$ and a uniform random edge $uv$ incident with $u$, traversing the edge from $u$ to $v$ does not take us outside the set $S$.   
Partitioning $T_\rho f$ to its low-degree and high-degree part, the high-degree part contributes little to the inner product because $T_{\rho}$ shrinks its 2-norm to something very small (this follows from the Fourier transform representation of $T_\rho$). As for the contribution of the low-degree part, this can be bounded by an expression involving the $4$-norm of $T_\rho f$, by using H\"older's inequality. Applying the hypercontractive inequality for $T_{\rho}$ and rearranging, we obtain an upper bound on the probability of staying inside $S$, and thus a lower bound on the probability of moving outside it. 

\remove{First, one proves (using a hypercontractive inequality) that a (possibly pseudorandom) set has 4-norm bounded from above in terms of its 2-norm. One then applies H\"older's inequality to conclude that a small (and possibly pseudorandom) set has Fourier transform concentrated on the high degrees. Finally, one uses estimates on the eigenvalues of the adjacency operator $A$ of $G$ (specifically, estimates showing that the large eigenvalues correspond to low degrees), together with an expansion of $\langle A1_S,1_S\rangle$ analogous to the proof of Hoffman's eigenvalue bound, to conclude that small (pseudorandom) sets expand.}

\bigskip In this paper, we obtain an analogue of Bonami's lemma for `global' functions on the space $\l(V,W)$ of linear maps from $V$ to $W$, where $V$ and $W$ are finite-dimensional vector spaces over a finite field. This leads to a significant conceptual simplification and streamlining/shortening of the proof of the seminal result of Khot, Minzer and Safra \cite{kms} obtaining small-set expansion for pseudorandom sets in the Grassmann graph; the latter was one of the crucial ingredients in the celebrated proof of the 2-2 Games conjecture (with imperfect completeness), along with the earlier results of Khot, Minzer and Safra in \cite{kms-first}, of Dinur, Kindler, Khot, Minzer and Safra in \cite{dkms2,dkkms}, and of Barak, Kothari and Steurer in \cite{bks}.

The Unique Games conjecture of Khot is considered by many to be the second-most important open problem in complexity theory, after the P versus NP problem; yet it is not considered to be out of reach in the same way as the P versus NP problem. The proof of the 2-2 Games conjecture (with imperfect completeness) is one of the greatest breakthroughs in the area, in recent times. We proceed to give a full statement of the problem.

One can think of an instance of the Unique Games problem as a system of linear equations over  $\mathbb{F}_p$ for some prime $p$, where every equation (/constraint) is of the form
$$t_{ij}x_i + t_{ij}'x_j = c_{ij},$$
for $i,j \in [n]$, where $x_1,x_2,\dots x_n$ are variables taking values in  $\mathbb{F}_p$, and $c_{ij},t_{ij},t_{ij}' \in \mathbb{F}_p$ are constants\footnote{The original version of the Unique Games Conjecture allowed for more general types of constraints, but it was shown in~\cite{maxcut} that one can assume without loss of generality that the constraints are as we describe here.}. The goal is to find an assignment of the variables that satisfies a large fraction of the equations (/constraints). The Unique Games conjecture states that for any $\epsilon>0$, there exists $p_0(\epsilon) \in \mathbb{N}$ such that for all primes $p \geq p_0(\epsilon)$, given an instance of the Unique Games conjecture for $\mathbb{F}_{p}$ where we are promised there is an assignment satisfying at least a $(1-\epsilon)$-fraction of the equations, it is an NP-hard problem to find an assignment satisfying (even) at least an $\epsilon$-fraction of the equations.

The `uniqueness' in the Unique Games problem refers to the fact each equation (/constraint) $\mathcal{E}$ of the form $t_{ij}x_i + t_{ij}'x_j = c_{ij}$ inside an instance actually fixes a one-to-one correspondence between assignments of the variable $x_i$ and assignments of the variable $x_j$, since if the coefficients $t_{ij}$ and $t_{ij}'$ are non-zero (which indeed we may assume, without loss of generality), then for each assignment of $x_i$ there is a unique assignment of $x_j$ for which $\mathcal{E}$ is satisfied, and vice versa. The 2-2 Games conjecture (the variant, that is, with imperfect completeness) refers to an analogous problem, where each constraint sets a relation between a pair of distinct variables $x_i$ and $x_j$ which, rather than being `unique' (or `one-to-one'), is instead `two-to-two'. (We explain precisely what this means, shortly.) This is a more general set of allowed constraints, and so intuitively one would guess that it would be more difficult to find an assignment that satisfies at least an $\epsilon$-fraction of the constraints, even when one is promised that there exists an assignment satisfying at least a $(1-\epsilon)$-fraction of them. This guess turns out to be correct: it is easy to prove that the 2-2 Games conjecture with imperfect completeness, follows from the Unique Games conjecture, and (as mentioned above) the former has now been proven, whereas the latter remains open.

 Now let us explain what a $2$-to-$2$ constraint is. A very simple example is the constraint
 $$t_{ij}x_i + t_{ij}'x_j \in \{c_{ij},c_{ij}'\}$$
 on the pair of variables $x_i$ and $x_j$, where $t_{ij},t_{ij}' \in \mathbb{F}_p^{\times}$ and $c_{ij}\neq c_{ij}' \in \mathbb{F}_p$. Now each assignment of $x_i$ that satisfies the constraint has {\em two} corresponding assignments of $x_j$ that satisfy the constraint, and vice versa. Formally, a constraint on two variables $x$ and $y$ is said to be a {\em $2$-to-$2$ relation on their assignments} if there is a partition of the set of possible assignments of $x$ into a collection of pairs $\mathcal{P}$, and a partition of the possible assignments of $y$ into a collection of pairs $\mathcal{Q}$, along with a perfect matching from $\mathcal{P}$ to $\mathcal{Q}$, such that once two matched pairs are chosen (one pair, $p$ say, in $\mathcal{P}$ and the other pair, $q$ say, in $\mathcal{Q}$), any assignment of $x$ from $p$ and any assignment of $y$ from $q$ will satisfy the constraint; and furthermore, any assignments of $x$ and $y$ that do not come from matched pairs do not satisfy the constraint.
 
 Let us now give a more complicated example of a 2-2 constraint, an example that was crucial in the aforementioned works on the 2-2 Games conjecture. We now index the variables by $\ell$-dimensional subspaces of $\mathbb{F}_2^k$, and we impose constraints $C_{L,L'}$ on pairs of variables $x_{L},x_{L'}$, where $L$ and $L'$ are $\ell$-dimensional subspaces with $\dim(L \cap L') = \ell-1$. For each $\ell$-dimensional subspace $L$, we seek to assign values (to the variable $x_L$) which are $\mathbb{F}_2$-linear functionals on $L$, i.e.\ the assignments to $x_L$ are elements $f_{L}$ of the dual space $L^*$. The constraint $C_{L,L'}$ is defined as follows: an assignment $f_L$ to $x_L$ and an assignment $f_{L'}$ to $x_{L'}$ together satisfy $C_{L,L'}$ if $f_L(x)=f_{L'}(x)$ for all $x\in L\cap L'$, i.e.\ if the linear functionals $f_L$ and $f_{L'}$ agree on $L\cap L'$. We note that since $L\cap L'$ is of codimension one in $L$ (and also of codimension one in $L'$), and since we are working over $\mathbb{F}_2$, for any given linear functional $g$ on $L \cap L'$ there are exactly two possible extensions of $g$ to a linear functional on $L$ and exactly two possible extensions of $g$ to a linear functional on $L'$. It follows that the constraint $C_{L,L'}$ is indeed 2-2 in the above sense. 

In \cite{dkms2}, Dinur, Khot, Kindler, Minzer and Safra reduced the 2-2 Games conjecture (with imperfect completeness) to a statement called the `Grassmann Soundness Hypothesis', which concerns constraints of the form $C_{L,L'}$ defined above. To explain further, we need some additional terminology. The {\em Grassmann graph} $G_{k,\ell}$ denotes the graph whose vertex-set consists of all $\ell$-dimensional subspaces of $\mathbb{F}_2^{k}$, and where two $\ell$-dimensional subspaces $L$ and $L'$ are joined by an edge if $\dim(L \cap L')=\ell-1$. An {\em $(\ell,k)$-Grassmann Test} is a system of constraints where we have a variable $x_L$ for every vertex of the Grassmann graph (i.e.\ for every $\ell$-dimensional subspace $L$ of $\mathbb{F}_2^k$), and a constraint $C_{L,L'}$ as defined above for every edge of the Grassmann graph. The Grassmann Soundness Hypothesis states (roughly) that if an assignment $(f_L)_{L \in V(G_{k,\ell})}$ satisfies at least an $\epsilon$-fraction of the constraints $(C_{L,L'})_{\{L,L'\}\in E(G_{k,\ell})}$, then there must be a linear functional $f:\mathbb{F}_2^{k}\to\mathbb{F}_2$ that agrees on $L$ with the assignment $f_L:L\to \mathbb{F}_2$, for many $\ell$-dimensional subspaces $L$. More precisely, there must be a linear functional $f:\mathbb{F}_2^k \to \mathbb{F}_2$, and two subspaces $A \leq B \leq \mathbb{F}_2^k$ (with $A$ of low dimension and $B$ of low codimension) such that $f$ agrees with a constant fraction of those assignments $f_L$ for which $L$ is sandwiched between $A$ and $B$. The formal statement is as follows.
\begin{hyp}[Grassmann Soundness Hypothesis]
For every $\epsilon>0$, there exist $\ell_0 \in \mathbb{N}$, $\eta>0$, $d\in \mathbb N$ and a function $k_0:\mathbb{N} \to \mathbb{N}$ such that the following holds. If $\ell \geq \ell_0$ and $k \geq k_0(\ell)$, and an assignment is given for the $(\ell,k)$-Grassmann Test that satisfies at least an $\epsilon$-fraction of the constraints, then there exists a linear functional $f:\mathbb{F}_2^k\to\mathbb{F}_2$ and subspaces $A\subseteq B\subseteq \mathbb{F}_2^k$ with $\dim(A)+\codim(B)\leq d$, such that for at least an $\eta$-fraction of the $\ell$-dimensional spaces $A\subseteq L \subseteq B$, it holds that $f_L$ (the assignment of $x_L$) is equal to the restriction of $f$ to $L$.
\end{hyp}
 The work of Barak, Kothari and Steurer \cite{bks} further reduced 
the Grassmann Soundness Hypothesis to the `Grassmann Expansion Hypothesis', a statement about the expansion properties of the Grassmann graph, which we now describe. Given a finite, $d$-regular graph $G=(V,E)$ and a set of vertices $S \subset V(G)$, we define the {\em expansion ratio}
$$\Phi_G(S) := \frac{|E_G(S,\overline{S})|}{d|S|},$$
where $E_G(S,\overline{S})$ denotes the set of edges of $G$ with one endpoint in $S$ and the other endpoint in $\overline{S}: = V(G)\setminus S$. (Note that $\Phi_G(S)$ is precisely the probability that, if we pick uniformly at random a vertex $u$ of $S$ and then uniformly at random an edge of $G$ that is incident with $u$, then the other endpoint of this edge lies outside $S$.) The Grassmann Expansion Hypothesis states that pseudorandom sets in the Grassmann graph have high expansion ratio, where by `psuedorandom' we mean that the density of the set on lower-order copies of the Grassmann graph is not too high:
\begin{hyp}[Grassmann Expansion Hypothesis]
For any $0 < \epsilon <1$, there exists $\ell_0=\ell_0(\epsilon) \in \mathbb{N}$, $d \in \mathbb{N}$ and $\eta >0$ such that the following holds. Let $\ell \geq \ell_0$ and let $k$ be sufficiently large depending on $\ell$. Let $S \subset V(G_{k,\ell})$ such that for any subspaces $A$ and $B$ of $\mathbb{F}_2^{k}$ with $A \subseteq B$ and $\dim(A)+\codim(B) \leq d$, we have
$$\frac{|\{L \in S:\ A \subseteq L \subseteq B\}|}{|\{L \in V(G_{k,\ell}):\ A \subseteq L \subseteq B\}|} \leq \eta.$$
Then $\Phi_{G_{k,\ell}}(S) \geq 1-\epsilon$.
\end{hyp}
The proof of the 2-2 Games conjecture (with imperfect completeness) was completed when Khot, Minzer and Safra proved the Grassmann Expansion Hypothesis in the seminal work \cite{kms}. The proof in \cite{kms}, however, is extremely long and technical. In this paper, we find a streamlined proof by first obtaining an (essentially optimal) analogue of Bonami's lemma for complex-valued functions on $\l(V,W)$, where $V$ and $W$ are vector spaces over $\mathbb{F}_q$, and then using the $q=2$ case of this to obtain a small-set expansion theorem for pseudorandom sets in the Shortcode Graph (the graph with vertex-set $\l(V,W)$, where two linear maps $A_1$ and $A_2$ are joined by an edge if $A_1-A_2$ is of rank one); such a small-set expansion theorem was already known to imply the Grassmann Expansion Hypothesis, by the work of Barak, Kothari and Steurer in \cite{bks}.

We now describe our results in more detail. Our conceptual starting-point is the following (conditional) analogue of Bonami's lemma for global functions on product spaces, obtained by Keevash, Lifshitz, Long and Minzer in \cite{kllm}. To state it we need some more notation and definitions. If $\Omega = X^n$ is a finite product-space, and $S\subset [n]$, we write $\Omega_S = X^S$. For $x \in \Omega_S$ and a function $f:\Omega\to \mathbb{C}$, we write $f_{S \to x}$ for the `restricted' function on $\Omega_{[n] \setminus S}$ defined by $f_{S \to x}(y) = f(x,y)$, where (abusing notation slightly) we write $(x,y)$ for the element $z \in \Omega$ with $z_i = x_i$ for all $i \in S$ and $z_i=y_i$ for all $i \in [n] \setminus S$. We equip the product-space $\Omega$ with the uniform (product) measure $\mu$ on $\Omega$, and similarly we equip the product-space $\Omega_S$ with the uniform (product) measure on $\Omega_S$, for any $S \subset [n]$. The {\em Efron-Stein decomposition} is an orthogonal decomposition of
$L^{2}(\Omega,\mu)$ into spaces $V_{S}$ (for $S \subset [n]$), where
$V_{S}$ consists of the functions in $L^2(\Omega,\mu)$ that depend
only upon the coordinates in $S$ and are orthogonal to any function that
depends only upon the coordinates in $T$, for a proper subsets $T$ of $S$. For a complex-valued function $f:\Omega \to \mathbb{C}$ and for each $S \subset [n]$, we define $f^{=S}$ to be the orthogonal projection
of $f$ onto $V_{S}$. We define the {\em Efron-Stein degree} of $f$ to be $\max\{|S|: f^{=S} \neq 0\}$, and we define the {\em degree-$d$ truncation} of $f$ to be the function $f^{\leq d}$ obtained by orthogonally projecting $f$ onto the linear space of functions of (Efron-Stein) degree at most $d$ (in other words, $f^{\leq d}$ is simply the degree-$d$ part of $f$).
\begin{thm}[Keevash, Lifshitz, Long, Minzer, 2019+]
\label{thmcor:KLLM} Let $\Omega$ be finite product space. Let $f\colon\Omega\to\mathbb{C}$ and let $\delta>0$. Suppose that 
$\|f_{S\to x}\|_{2}^{2}\leq \delta$ for sets $S\subseteq [n]$ with $|S| \leq d$ and all $x\in\Omega_{S}$. Then $\|f^{\leq d}\|_{4}^{4}\le 1000^{d}\delta\|f^{\leq d}\|_{2}^{2}$.
\end{thm}

We call the functions $f_{S \to x}$ (for $|S| \leq d$) the {\em $d$-restrictions} of $f$. The above theorem says that if $f$ is a function whose $d$-restrictions have small 2-norms, then the 4-norm of the degree-$d$ part of $f$ can be bounded from above in terms of its $2$-norm. Theorem \ref{thmcor:KLLM} was used
in \cite{kllm} to obtain a small-set expansion theorem for
noise operators on product spaces; this small-set expansion theorem then played a crucial role in obtaining sharp forbidden intersection
theorems for subsets $[m]^{n}$.

Our first aim in this paper is to obtain an analogue of Theorem \ref{thmcor:KLLM} for complex-valued functions on $\l(V,W)$, but with Efron-Stein degree replaced by a different notion of degree, namely, the maximum rank of a linear map appearing in the Fourier expansion of $f$ (this turns out to be the same as the `junta degree', defined below). We note that $\l(V,W)$ could be viewed as a product space by fixing bases of $V$ and $W$, and it could be equipped with the corresponding Efron-Stein degree, but this notion of degree would not be invariant under changes of basis and would not therefore be useful for applications.

To state our (conditional) Bonami-type lemma for functions on $\l(V,W)$, we need some more definitions. Let $q$ be a prime power, and let $V$ and $W$ be finite-dimensional vector spaces over $\mathbb{F}_q$. We must first define our notion of a {\em $d$-restriction} of a function $f:\l(V,W) \to \mathbb{C}$. This is a little notationally cumbersome, if intuitively clear.

Let $V_{1}$ be a subspace of $V$, let $W_{1}$ be a subspace
of $W$, let $T\in \l(V,W)$ be a linear map, and let $f:\l(V,W) \to \mathbb{C}$. The restriction $f_{(V_{1},W_{1})\to T}$ is the function from $\l(V/V_{1},W_1)$ to $\mathbb{C}$ defined by
\[f_{(V_{1},W_{1})\to T}(A)=f(A'+T) \quad \forall A \in \l(V/V_1,W_1),
\]
where $A'\in \l(V,W)$ is the unique linear map with kernel containing $V_1$ and satisfying $A' = A \circ \mathcal{Q}_{V_1}$, with $\mathcal{Q}_{V_1}:V \to V/V_1$ denoting the natural quotient map. If $\dim(V_1)+\codim(W_1) \leq d$ then we call such a restriction a \emph{$d$-restriction}. We note that the linear maps $B$ of the form $A'+T$ in the definition $f_{(V_1,W_1)\to T}$ are precisely the linear maps $B$ such that $B$ agrees with $T$ on $V_1$ and $B^*$ agrees with $T^*$ on the annihilator of $W_1$.

Adopting the matrix perspective, the $d$-restriction of a function $f$ on $n$ by $m$ matrices over $\mathbb{F}_q$ corresponds to restricting $f$ to those matrices where $r$ specific rows and $c$ specific columns take fixed values, where $r+c \leq d$ (and possibly translating the domain by a fixed matrix, if the matrix of $T$ has non-zero entries outside the $r$ fixed rows and the $c$ fixed columns).

A function $f:\l(V,W)\to \mathbb{C}$ is said to be a \emph{$d$-junta} if there exist $v_{1},\ldots,v_{i}\in V,u_{i+1},\ldots,u_{d}\in W^{*}$,
such that the value of $f(A)$ is determined once we know the values
of $A(v_{i})$ and the values of $A^{*}(u_{i})$. The {\em junta-degree} of a
function $f$ is the minimal integer $d$ such that $f$ can be written as
a sum of $d$-juntas. (As mentioned above, we will show that the junta-degree of $f$ is equal to the maximum rank of a linear map that appears in the Fourier expansion of $f$.) For a function $f:\l(W,V) \to \mathbb{C}$, we let $f^{\leq d}$ denote its orthogonal projection onto the (linear) space of all functions with junta-degree at most $d$ (in other words, as before, $f^{\leq d}$ is simply the degree-$d$ part of $f$).

For $d \in \mathbb{N}$ and $\delta >0$, we say a function $f:\l(V,W)\to \mathbb{C}$ is {\em $(d,\delta)$-restriction global} if $\|f_{(V_1,W_1) \to T}\|^2_2 \leq \delta$ for all $V_1 \leq V$ and $W_1 \leq W$ with $\dim(V_1)+\codim(W_1) \leq d$ and all $T \in \l(V,W)$; in other words, if all the $d$-restrictions of $f$ have 2-norm at most $\sqrt{\delta}$. This is our notion of `globalness' for functions on $\l(V,W)$.

We can now state our Bonami-type lemma for global functions on $\l(V,W)$.

\begin{thm}
\label{cor:KLLM for BS} Let $d \in \mathbb{N}$, let $\delta>0$, let $q$ be a prime power, let $V$ and $W$ be finite-dimensional vector spaces over $\mathbb{F}_q$, and suppose that $f:\l(V,W) \to \mathbb{C}$ is a $(d,\delta)$-restriction global function. Then 
\[
\|f^{\leq d}\|_{4}^{4}\le q^{Cd^{2}}\delta\|f^{\leq d}\|_{2}^{2},
\]
where $C>0$ is an absolute constant.
\end{thm}
The $d^{2}$ in the exponent is sharp, as can be verified by inspecting
the function $\sum_{X\in\mathcal{L}(W,V):\,\mathrm{rank}(X)=d}u_{X}$,
where $u_{X}(A)=\omega^{\tau(\Tr(XA))}$, $\omega = \exp(2\pi i /p)$ and $\tau:\mathbb{F}_q \to \mathbb{F}_p$ is defined by $\tau(x) = x+x^p+\ldots+x^{p^{s-1}}$ for $q = p^s$. (This example shows that one must take $C\geq 1$, for any $d$ and $q$.)

To motivate our proof of Theorem \ref{cor:KLLM for BS}, and to illustrate some of the key ideas in a simpler setting, we will first give a proof of a (slightly weaker) version of Theorem \ref{thmcor:KLLM} for the product space $\mathbb{F}_p^n$ for $p$ a prime, and with $C^d$ replaced by $(Cd)^d$.

Using Theorem \ref{cor:KLLM for BS}, we obtain the following quantitatively sharp small-set expansion theorem for the shortcode graph, which (as mentioned above) implies the Grassmann Expansion Hypothesis.

\begin{thm}[Small-set expansion theorem for the shortcode graph]
\label{thm:ssesc}
 There exist absolute constants $C_1,C_2>0$ such that the following holds. Let $r \in \mathbb{N}$, and let $S\subseteq\mathcal{L}\left(V,W\right)$
be a family of linear maps with $1_{S}$ being $\left(C_1r,q^{-C_2r^2}\right)$-restriction
global. Then 
\[
\Pr_{A\sim S,\ B\text{ of rank }1}\left[A+B\in S\right] < q^{-r}.
\]
\end{thm}
Theorem \ref{thm:ssesc} is sharp up to the values of $C_1$ and $C_2$, as can be seen by considering the family $S = \{A \in \l(V,W):\ \rank(A) \leq n-r\}$, where $\dim(V)=\dim(W)=n$.

The Bonami-type lemma established in this paper for the bilinear scheme --- specifically, the fact that the constant scales as $q^{O(d^2)}$, independent of the dimension $n$ --- has proven to be a fundamental tool for analyzing functions on non-Abelian groups. In a follow-up paper, Evra, Kindler and Lifshitz \cite{EKL24} extend these methods to the special linear group $\operatorname{SL}_n(q)$ by embedding it into $\mathcal{L}(\mathbb{F}_q^n,\mathbb{F}_q^n)$ applying the results established here. Furthermore, a subsequent work by Evra, Kindler, Lifshitz and Lindzey \cite{EKLL26+} generalizes this approach to all families of classical finite simple groups (symplectic, unitary and orthogonal). By lifting functions from the group to the bilinear scheme in a natural way, they derive a Bonami-type inequality for global functions on these groups. This framework yields two striking consequences. Firstly, it recovers the breakthrough character bounds of Guralnick, Larsen and Tiep (Theorem 1.3 in \cite{GLT24}) via a purely analytic argument, avoiding any reliance on Deligne-Lusztig theory. And secondly, it implies a polynomial Bogolyubov theorem for general (non-normal) subsets of simple groups of unbounded rank, a result that was previously out of reach.

\section{Warm up: Hypercontractivity on $\mathbb{F}_{p}^{n}$}

In this section we give a new proof of a (slightly weaker) version of
Theorem \ref{thmcor:KLLM} for the product space $\mathbb{F}_p^n$ (for a prime $p$), with $C^{d}$ replaced by $(Cd)^{d}$. This will
help motivate and elucidate its (more complicated) adaptation for functions on $\l(V,W)$. We start by
introducing some notation.

\subsection{The Fourier expansion of functions on $\mathbb{F}_{p}^{n}$}

We equip $\mathbb{F}_p^n$ with the natural product measure $\mu = \mu^n$ (which is simply the uniform measure on $\mathbb{F}_p^n$), and we let $L^2(\mathbb{F}_p^n)$ be the Hilbert space of complex-valued functions on $\mathbb{F}_p^n$ with the inner product
$$\langle f,g \rangle = \mathbb{E}[f\overline{g}].$$
Let $\omega=e^{\frac{2\pi i}{p}}$. The characters of the Abelian group $(\mathbb{F}_p^n,+)$ are the functions $\{\chi_\gamma:\ \gamma \in \mathbb{F}_p^n\}$, where $\chi_{\gamma}\left(x\right)=\omega^{\left\langle \gamma,x\right\rangle}$ for $x\in\mathbb{F}_{p}^{n}$. (Here, for $\delta,\gamma \in \mathbb{F}_p^n$, we write $\langle \delta,\gamma \rangle = \sum_{i=1}^{n}\delta_i \gamma_i$; note that this inner product is not normalised, unlike the one for functions). These characters form an orthonormal basis for $L^2(\mathbb{F}_p^n)$ and therefore every
function $f\colon\mathbb{F}_{p}^{n}\to\mathbb{C}$ has a \emph{Fourier
expansion} 
\[
f=\sum_{\gamma\in\mathbb{F}_{p}^{n}}\hat{f}\left(\gamma\right)\chi_{\gamma},
\]
 where 
\[
\hat{f}\left(\gamma\right)=\left\langle f,\chi_{\gamma}\right\rangle .
\]

\noindent We write $\mathrm{supp}\left(\gamma\right)$ for the set of coordinates
$i\in\left[n\right]$ for which $\gamma_{i}\ne0.$ We write $\left|\gamma\right|=|\mathrm{supp}\left(\gamma\right)|$.

\subsection{Restrictions}

For $S \subset [n]$, we write $\overline{S}: = [n] \setminus S$ for the complement of $S$. Abusing notation slightly, for $x\in\mathbb{F}_{p}^{S}$ and $y\in\mathbb{F}_{p}^{\overline{S}}$
we write $\left(x,y\right)$ for the vector $z \in \mathbb{F}_p^n$ defined by $z_i = x_i$ for all $i \in S$ and $z_i = y_i$ for all $i \in \overline{S}$, i.e.\ we put $x$
in the $S$-coordinates and $y$ in the $\overline{S}$-coordinates.
For $f:\mathbb{F}_p^n \to \mathbb{C}$ and $x\in\mathbb{F}_{p}^{S}$, we define the {\em restriction} $f_{S\to x}\colon\mathbb{F}_{p}^{\overline{S}}\to\mathbb{C}$
of $f$ by setting $f_{S\to x}\left(y\right)=f\left(x,y\right)$ for all $y \in \mathbb{F}_p^{\overline{S}}$. If $|S| \leq d$ then we call $f_{S \to x}$ a {\em $d$-restriction}. Similarly, for a set $A \subset \mathbb{F}_p^n$, we define the {\em restriction} $A_{S \to x}$ by
$$A_{S \to x} = \{y \in \mathbb{F}_p^{\overline{S}}:\ (x,y) \in A\} \subset \mathbb{F}_p^{\overline{S}}.$$
If $|S| \leq d$ then we call $A_{S \to x}$ a {\em $d$-restriction}.

\subsection{Global functions}
We now introduce the (crucial) definition of a {\em global function}: a function is global if its restrictions (more precisely, those restrictions corresponding to small sets) have small 2-norms.

\begin{defn}
For $d \in \mathbb{N} \cup \{0\}$ and $\epsilon >0$, we say a function $f:\mathbb{F}_p^n \to \mathbb{C}$ is {\em $(d,\epsilon)$-global} if $\|f_{S \to x}\|_2^2 \leq \epsilon$ for all $S \subset [n]$ with $|S| \leq d$ and all $x \in \mathbb{F}_p^S$.
\end{defn}
Note that for a set $A \subset \mathbb{F}_p^n$, the indicator function $1_{A}$ is $(d,\epsilon)$-global if each of its $d$-restrictions $A_{S \to x}$ satisfy
$$\mu^{\overline{S}}(A_{S \to x}) \leq \epsilon.$$
(Here, of course, $\mu^{\overline{S}} = \mu^{[n] \setminus S}$ denotes the product measure (= uniform measure) on $\mathbb{F}_p^{\overline{S}}$. As the relevant measure will always be clear from the context, sometimes we will omit the superscript.)

\subsection{The Efron-Stein decomposition}

For a product space $X^n$ (equipped with a product measure $\mu^n$), the {\em Efron-Stein decomposition} is an orthogonal decomposition of
$L^{2}(X^n,\mu^n)$ into spaces $V_{S}$ (for $S \subset [n]$), where
$V_{S}$ consists of the functions in $L^2(X^n,\mu^n)$ that depend
only upon the coordinates in $S$ and are orthogonal to all functions that
depend only upon the coordinates in $T$, for all proper subsets $T$ of $S$. For each $S \subset [n]$, we define $f^{=S}$ to be the orthogonal projection
of $f$ onto $V_{S}$. For $d \in \mathbb{N} \cup \{0\}$, we define $f^{=d} = \sum_{S \subset [n]:\ |S|=d}f^{=S}$ and we define $f^{\leq d} = \sum_{S \subset [n]:\ |S| \leq d} f^{=S}$. The functions $f^{=S}$ are sometimes called the {\em Efron-Stein components} of $f$, though `Efron-Stein projections' would be more accurate.

In our case, $X = \mathbb{F}_p$ and $\mu^n$ is the uniform measure on $\mathbb{F}_p^n$; the function $f^{=S}$ can also be defined Fourier-analytically, by
setting 
\[
f^{=S}=\sum_{\gamma:\mathrm{supp}\left(\gamma\right)=S}\hat{f}\left(\gamma\right)\chi_{\gamma}.
\]
 The \emph{degree} of $f\in L^{2}(\mathbb{F}_{p}^{n})$ is defined to be the maximal size of a set $S$ such that $f^{=S} \ne 0$.

\subsection{Laplacians and their Efron-Stein formulae}

Let $S\subseteq\left[n\right]$. We
define the \emph{expectation-when-resampling-$S$} operator
\[
\mathrm{E}_{S}\colon L^{2}\left(\mathbb{F}_{p}^{n}\right)\to L^{2}\left(\mathbb{F}_{p}^{n}\right)
\]
 by 
\[
\mathrm{E}_{S}f\left(x,y\right)=\mathbb{E}_{y'\sim\mathbb{F}_{p}^{S}}\left[f\left(x,y'\right)\right]
\]
for all $x \in \mathbb{F}_p^{\overline{S}}$ and $y \in \mathbb{F}_p^S$. For each $i \in [n]$, we write $\mathrm{E}_{i}$ for $\mathrm{E}_{\left\{ i\right\} }$, for  brevity. Note that $E_S$ is a linear operator. Since $E_S[\chi_{\gamma}] = 0$ whenever $\mathrm{supp}(\gamma) \cap S \neq \emptyset$, the operator $E_S$ has the following formula in terms of the Efton-Stein components:
$$E_S[f] = \sum_{T \subset [n]:\atop S \cap T = \emptyset} f^{=T}.$$
The \emph{Laplacians} $L_i:L^2(\mathbb{F}_p^n) \to L^2(\mathbb{F}_p^n)$ (for $i \in [n]$) are linear operators defined by
\[
L_{i}[f]=f-\mathrm{E}_{i}\left[f\right]
\]
 for all $f \in L^2(\mathbb{F}_p^n)$. (Note that $L_i$ is the Laplacian of the Markov chain on $\mathbb{F}_p^n$ where the $i$th coordinate is resampled uniformly at random at each step, independently of all previous steps, justifying the name.) It is easy to see that
\[
L_{i}\left[f\right]=\sum_{S\ni i }f^{=S}
\]
for all $f \in L^2(\mathbb{F}_p^n)$. Following \cite{kllm}, for a set $T=\{i_{1},\ldots,i_{t}\}$ of coordinates,
the \emph{iterated Laplacian} is defined by
\[
L_{T}[f]=L_{i_{1}}\circ\cdots\circ L_{i_{t}}[f]
\]
for $f \in L^2(\mathbb{F}_p^n)$ (it is easy to check that this definition is independent of the ordering of $i_1,\ldots,i_t$).
The iterated Laplacian can be interpreted in two ways. The `analytic'
way is via the formula 
\[
L_{T}[f]=\sum_{S\supseteq T}f^{=S}.
\]
 The `probabilistic' way (analogous to inclusion-exclusion) is via the formula 
\[
L_{T}[f]=\sum_{S\subseteq T}(-1)^{|S|}\mathrm{E}_{S}[f].
\]
Henceforth, by a slight abuse of terminology, we will refer to the iterated Laplacians simply as Laplacians.

\subsection{Derivatives}

For the Boolean cube $\left\{ 0,1\right\} ^{n}$, there is a natural
notion of the (discrete) derivative of a function $f:\{0,1\}^n \to \mathbb{R}$ in the $i$-direction, namely $f_{i\to1}-f_{i\to0}$. In the general
product space setting, such a notion is not so readily available, and instead
it is common to work with the Laplacians in place of the derivatives.

For our purposes, however, we need a derivative-like operator that (strictly) reduces the degree (much as the partial derivative operator $\partial/\partial x_i$, applied to multivariate polynomials in $x_1,\ldots,x_n$, reduces the total degree); this will enable us to carry out induction on the degree. The Laplacians themselves do not necessarily reduce the degree of a function on $\mathbb{F}_p^n$, but one can easily fix this problem by defining the derivatives to be restrictions
of the Laplacians.

For $S\subseteq\left[n\right]$ and $x\in\mathbb{F}_{p}^{S}$,
we define the derivative operator
\[
D_{S,x}\colon L^{2}\left(\mathbb{F}_{p}^{n}\right)\to L^{2}\left(\mathbb{F}_{p}^{\left[n\right]\setminus S}\right)
\]
by $D_{S,x}[f]:=(L_{S}\left[f\right])_{S\to x}$ for $f \in L^2(\mathbb{F}_p^n)$.
\begin{lem}
Let $f:\mathbb{F}_p^n \to \mathbb{C}$, let $S \subseteq [n]$, let $x \in \mathbb{F}_p^S$ and let $g=D_{S,x}\left[f\right]$. Then 
\begin{equation}
\hat{g}\left(\gamma\right)=\sum_{\beta\in\mathbb{F}_{p}^{S}: \atop \mathrm{supp}\left(\beta\right)=S}\hat{f}\left(\beta,\gamma\right)\chi_{\beta}\left(x\right)\label{eq:Fourier formulas of derivatves}
\end{equation}
for all $\gamma \in \mathbb{F}_p^{\overline{S}}$.
\end{lem}
\begin{proof}
By linearity of both sides of (\ref{eq:Fourier formulas of derivatves}) in $f$, it suffices to prove the lemma in the case where $f$ is a character $\chi$ of $\mathbb{F}_p^n$; such a character can of course be written in the form
$\chi_{\left(\beta',\gamma'\right)}$, for $\beta'\in\mathbb{F}_{p}^{S}$
and $\gamma'\in\mathbb{F}_{p}^{\overline{S}}$. Clearly, if $f = \chi_{(\beta',\gamma')}$, then the right-hand side of (\ref{eq:Fourier formulas of derivatves})
is equal to
\[
\begin{cases}
\chi_{\beta'}\left(x\right) & \text{ if } \gamma'=\gamma\ \text{and}\ \mathrm{supp}\left(\beta'\right)=S,\\
0 & \mathrm{otherwise.}
\end{cases}
\]
To prove the lemma we must therefore show that these are indeed
the Fourier coefficients of $D_{S,x}\left[\chi_{\left(\beta',\gamma'\right)}\right]$, i.e. that 
\[
D_{S,x}\left[\chi_{\left(\beta',\gamma'\right)}\right]=\begin{cases}
\chi_{\beta'}\left(x\right)\chi_{\gamma'} & \text{ if }\mathrm{supp}\left(\beta\right)=S,\\
0 & \text{otherwise.}
\end{cases}
\]
 We have $L_{S}\left[\chi_{\left(\beta',\gamma'\right)}\right]=0$ whenever $\mathrm{supp}\left(\beta'\right)\ne S$,
which settles the case $\mathrm{supp}\left(\beta'\right)\ne S$. When
$\mathrm{supp}\left(\beta'\right)=S$ we have $L_{S}\left[f\right]=\chi_{\left(\beta',\gamma'\right)}$
and therefore 
\[
D_{S,x}\left[\chi_{\left(\beta',\gamma'\right)}\right]=\chi_{\beta'}\left(x\right)\chi_{\gamma'},
\]
settling the case $\mathrm{supp}(\beta')=S$.
\end{proof}

\subsection{Influences}

We define the \emph{influence} of a set $T \subset [n]$ by setting $I_{T}[f]=\|L_{T}[f]\|_{2}^{2}$.
Roughly speaking, the influence of a set $T$ measures the impact
of the coordinates in $T$ on the value of $f$. For $x\in\mathbb{F}_{p}^{T}$, we define the influence of the pair
$\left(T,x\right)$ by 
\[
I_{T,x}[f]:=\|D_{T,x}\left[f\right]\|_{2}^{2}.
\]
It is easy to see that 
\[
I_{T}\left[f\right]=\mathbb{E}_{x\sim\mathbb{F}_{p}^{T}}\left[I_{T,x}\left[f\right]\right].
\]

The derivatives $D_{T,x}$ serve as a good analogues of the discrete
derivatives of a Boolean function. They satisfy the following properties.
\begin{enumerate}
\item \emph{Linearity:} The operator $D_{T,x}$ is a linear operator.
\item \emph{Degree reduction:} The operator $D_{T,x}$ sends a function
of degree $d$ to a function of degree at most $d-|T|$. (See Lemma
\ref{lem:degree reduction}, below.)
\item \emph{A measure of globalness:} As we are going to show in Section
\ref{sec:glob}, the smallness of the influences
$I_{T,x}$ of a function is equivalent to the smallness of the 2-norms
of its restrictions.
\end{enumerate}
\begin{lem}
\label{lem:degree reduction} Let $f\colon\mathbb{F}_{p}^{n}\to\mathbb{C}.$
Let $S,T\subseteq[n]$ with $S \cap T=\emptyset$, let $x\in\mathbb{F}_{p}^{T}$,
and let $g=D_{T,x}\left[f\right]$.
Then $g^{=S}=(f^{=T\cup S})_{T\to x}$. In particular, if $f$ is
a function of degree $d$, then $g$ is a function of degree at most
$d-|T|$.
\end{lem}
\begin{proof}
By linearity of the Efron-Stein components and the restriction operators
it is enough to prove the lemma for $f=\chi_{\left(\beta,\gamma\right)}$
where $\beta\in\mathbb{F}_{p}^{T}$ and $\gamma\in\mathbb{F}_{p}^{\overline{T}}.$
In this case, both the functions $g^{=S}$ and $\left(f^{=T\cup S}\right)_{T\to x}$
are zero unless $\supp\left(\beta\right)=T$. So suppose that
$\supp\left(\beta\right)=T$. Then $g=D_{T,x}\left[f\right]=\chi_{\beta}\left(x\right)\chi_{\gamma}$
and therefore 
\[
g^{=S}=\begin{cases}
g & \text{ if }\supp\left(\gamma\right)= S,\\
0 & \text{ if } \supp\left(\gamma\right)\neq S,
\end{cases}.
\]
 Similarly,
$$f^{=S\cup T}=\begin{cases}
f & \text{ if }\supp\left(\gamma\right)= S,\\
0 & \text{ if }\supp\left(\gamma\right)\neq S,
\end{cases}$$
and therefore 
\[
\left(f^{=S\cup T}\right)_{T\to x}=\chi_{\beta}\left(x\right)\chi_{\gamma}.
\]
\end{proof}

We would like to note here a consequence of the above: truncating a function at degree $d$ cannot increase the influences, as the following lemma implies.
\begin{lem}
\label{lem:decomposition of generalised influences} $I_{S,x}[f]=\sum_{T\supseteq S}I_{S,x}[f^{=T}]$.
In particular, $I_{S,x}[f^{\le d}]\le I_{S,x}[f]$, so if $f$ is $\left(d,\epsilon\right)$-global, then so is its degree-$d$ truncation
$f^{\le d}$.
\end{lem}
\begin{proof}
This follows from Lemma \ref{lem:degree reduction} and the orthogonality
of the Efron-Stein components.
\end{proof}

We can now state the `conditional hypercontractive inequality' we will prove; it bounds the 4-norm of a function of degree at most $d$, in terms of the 2-norms of its derivatives.

\begin{thm}
\label{thm:hyp-prod}
Let $f:\mathbb{F}_p^n \to \mathbb{C}$ be a function of degree at most $d$. Then 
\[
\|f\|_{4}^{4}\le\left(100d\right)^{d}\sum_{S\subseteq\left[n\right]}\mathbb{E}_{x\sim\mathbb{F}_{p}^{S}}\left[(I_{S,x}[f])^{2}\right].
\]
Equivalently,
\[
\|f\|_{4}^{4}\le\left(100d\right)^{d}\sum_{S\subseteq\left[n\right]}\mathbb{E}_{x\sim\mathbb{F}_{p}^{S}}\|D_{S,x}[f]\|_2^4.
\]
\end{thm}

We call this a `conditional hypercontractive inequality' (with only slight abuse of terminology) because it quickly implies a hypercontractive inequality for global functions, as we will shortly see, in Section \ref{sec:glob}.

\subsection{Overview of the proof of Theorem \ref{thm:hyp-prod}}

The key ingredient of the proof of Theorem \ref{thm:hyp-prod} is the following `degree-reduction' lemma.
\begin{prop}
\label{prop:Inductive lemma } If $f:\mathbb{F}_p^n \to \mathbb{C}$ is a function of degree at most $d$, then 
\begin{equation}
\|f\|_{4}^{4}\le 2\cdot9^{d}\|f\|_{2}^{4}+2\sum_{S\ne\varnothing}\left(4d\right)^{\left|S\right|}\|L_{S}\left[f\right]\|_{4}^{4}.\label{eq:inductive lemma}
\end{equation}
 
\end{prop}
Theorem \ref{thm:hyp-prod} follows fairly easily from this by induction on the degree. Indeed, we
may then restrict $S$ to some $x\in\mathbb{F}_{p}^{\left|S\right|}$
and apply the inductive hypothesis to $D_{S,x}[f]$,
which has lower degree than $f$.

To prove our degree reduction lemma we first use Parseval (applied to the Efron-Stein decomposition of $f^2$) to write 
\[
\mathbb{E}\left[\left|f\right|^{4}\right]=\sum_{S\subseteq\left[n\right]}\left\Vert \left(f^{2}\right)^{=S}\right\Vert _{2}^{2}.
\]
We then use Efron-Stein decomposition of $f$ to write 
\[
f^{2}=\sum_{T_{1},T_{2}\subseteq [n]}f^{=T_{1}}f^{=T_{2}},
\]
 so that
\[
\left(f^{2}\right)^{=S}=\sum_{T_{1},T_{2} \subseteq[n]}\left(f^{=T_{1}}f^{=T_{2}}\right)^{=S}.
\]
 This will allow us to divide the nonzero terms $\left(f^{=T_{1}}f^{=T_{2}}\right)^{=S}$
into two groups:
\begin{enumerate}
\item The `Boolean-type' terms: these are terms $\left(f^{=T_{1}}f^{=T_{2}}\right)^{=S}$
with $T_{1}\Delta T_{2}=S$. Such terms correspond to the situation
in the Boolean cube, where $\chi_{T_1}\chi_{T_2}=\chi_{T_1\Delta T_2}$ for all $T_1,T_2 \subset [n]$. We
upper-bound these terms by $9^{d}\|f\|_{2}^{4}$ via a
reduction to the Bonami-Beckner-Gross hypercontractivity theorem (for the Boolean cube).
\item The terms `explained' by the Laplacians: these are the terms with $T_{1}\cap T_{2}\cap S\ne\varnothing$.
They also appear as terms in $\left(L_{T}\left[f\right]^{2}\right)^{=S}$
for every $T\subseteq T_{1}\cap T_{2}\cap S$. We will therefore be
able to upper-bound the terms $\left(f^{=T_{1}}f^{=T_{2}}\right)^{=S}$
by the terms $\|L_{T}\left[f\right]\|_{4}^{4}$ that appear at the
right-hand side of (\ref{eq:inductive lemma}).
\end{enumerate}

\subsection{Proof of Theorem \ref{thm:hyp-prod}}

We first follow the strategy we outlined above to upper-bound each of the terms
$\left\Vert \left(f^{2}\right)^{=S}\right\Vert _{2}^{2}.$
\begin{lem}
Let $f:\mathbb{F}_p^n \to \mathbb{C}$, let $g=f^{2}$ and let $g_{T}:=(L_{T}[f])^{2}$ for each $T \subset [n]$. Then 
\[
\|g^{=S}\|_{2}^{2}\le 2\sum_{\emptyset \neq T\subseteq S}\left(2\left|S\right|\right)^{-\left|T\right|}\|(g_{T})^{=S}\|_{2}^{2}+2\left(\sum_{T_{1}\Delta T_{2}=S}\|f^{=T_{1}}\|_{2}\|f^{=T_{2}}\|_{2}\right)^{2}
\]
for all $S \subseteq [n]$.
\end{lem}
\begin{proof}
We have 
\[
\|f\|_{4}^{4}=\sum_{S\subseteq\left[n\right]}\|g^{=S}\|_{2}^{2}
\]
and 
\[
g^{=S}=\sum_{T_{1},T_{2}\subseteq\left[n\right]}\left(f^{=T_{1}}f^{=T_{2}}\right)^{=S}.
\]
We would like to write the right-hand side in terms of the Laplacians.
We divide the pairs $\left(T_{1},T_{2}\right) \in (\mathcal{P}([n]))^2$ into three categories:
\begin{enumerate}
\item $\mathcal{F}_{1}=\mathcal{F}_1(S) = \left\{ \left(T_{1},T_{2}\right):\:T_{1}\Delta T_{2}=S\right\}$.
\item $\mathcal{F}_{2} = \mathcal{F}_2(S)=\left\{ \left(T_{1},T_{2}\right):\,T_{1}\cap T_{2}\cap S\ne\varnothing\right\}$.
\item $\mathcal{F}_{3} = \mathcal{F}_3(S) = (\mathcal{P}([n]))^2\setminus (\mathcal{F}_1(S)\cup \mathcal{F}_2(S))$ consists of the rest of the pairs.
\end{enumerate}

\subsection*{Upper-bounding the contribution from the pairs in $\mathcal{F}_{2}$}

For each $T\subseteq S$, let $g_{T}=L_{T}\left[f\right]^{2}$ and
let $\mathcal{F}_{T}$ be the set of pairs $(T_{1},T_{2}) \in (\mathcal{P}([n]))^2$ with $T_{1}\cap T_{2}\supseteq T$.
Then we have 
\[
(g_{T})^{=S}=\sum_{\left(T_{1},T_{2}\right)\in\mathcal{F}_{T}}\left(f^{=T_{1}}f^{=T_{2}}\right)^{=S}.
\]
 Inclusion-exclusion yields
\[
1_{\left(T_{1},T_{2}\right)\in\mathcal{F}_{2}}=\sum_{\emptyset \neq T\subseteq S}\left(-1\right)^{\left|T\right|-1}1_{\left(T_{1},T_{2}\right)\in\mathcal{F}_{T}}.
\]
 This implies that 
\[
\sum_{\emptyset \neq T\subseteq S}\left(-1\right)^{\left|T\right|-1}(g_{T})^{=S}=\sum_{\left(T_{1},T_{2}\right)\in\mathcal{F}_{2}}\left(f^{=T_{1}}f^{=T_{2}}\right)^{=S}.
\]
 Therefore, by Cauchy--Schwarz and the triangle inequality, we have 
\begin{align*}
\left\Vert \sum_{\left(T_{1},T_{2}\right)\in\mathcal{F}_{2}}\left(f^{=T_{1}}f^{=T_{2}}\right)^{=S}\right\Vert _{2}^{2} & \le\left(\sum_{\emptyset \neq T\subseteq S}\left(2\left|S\right|\right)^{-\left|T\right|}\right)\left(\sum_{\emptyset \neq T\subseteq S}\left(2\left|S\right|\right)^{\left|T\right|}\|(g_{T})^{=S}\|_{2}^{2}\right)\\
 & \le \sum_{\emptyset \neq T\subseteq S}\left(2\left|S\right|\right)^{\left|T\right|}\|(g_{T})^{=S}\|_{2}^{2}.
\end{align*}

\subsection*{Upper-bounding the contribution from the pairs in $\mathcal{F}_{3}$}

In fact, there is no contribution from these pairs. Indeed, if $\chi'$ and $\chi''$ are characters with supports $T_{1}$ and $T_{2}$
respectively, then $\chi'\chi''$ is a character
whose support contains $T_{1}\Delta T_{2}$ and is contained in $T_{1}\cup T_{2}$, so $(f^{=T_1}f^{=T_2})^{=S}$ is zero unless
\[
T_{1}\Delta T_{2}\subseteq S\subseteq T_{1}\cup T_{2}.
\]
Hence, when $\left(T_{1},T_{2}\right)\in\mathcal{F}_{3}$
we have $(f^{=T_1}f^{=T_2})^{=S}=0$.

\subsection*{Upper-bounding the contribution from the pairs in $\mathcal{F}_{1}$}

Suppose now that $T_{1}\Delta T_{2}=S$. Then we have 
\[
\left(f^{=T_{1}}f^{=T_{2}}\right)^{=S}=\mathrm{E}_{\overline{S}}\left(f^{=T_{1}}f^{=T_{2}}\right).
\]
 Indeed, for each $S' \subsetneq S$ we have 
\[
\left(f^{=T_{1}}f^{=T_{2}}\right)^{=S'}=0,
\]
since $S'$ does not contain $T_{1}\Delta T_{2}$. 

We now upper-bound $\mathrm{E}_{\overline{S}}\left(f^{=T_{1}}f^{=T_{2}}\right)\left(x\right)$, for $x \in \mathbb{F}_p^S$. By Cauchy--Schwarz, we have
\[
\left|\mathrm{E}_{\overline{S}}\left(f^{=T_{1}}f^{=T_{2}}\right)\left(x\right)\right| = \left|\left\langle \left(f^{=T_{1}}\right)_{S\to x},\overline{\left(f^{=T_{2}}\right)_{S\to x}}\right\rangle\right| \le\|(f^{=T_1})_{S\to x}\|_{2} \cdot \|(f^{=T_2})_{S\to x}\|_{2}
\]
for each $x \in \mathbb{F}_p^S$. Now, the function $x\mapsto\|(f^{=T_1})_{S\to x}\|_{2}$ depends only
upon the coordinates in $T_{1}\cap S$, whereas the function $x\mapsto\|(f^{=T_2})_{S\to x}\|_{2}$
depends only upon the coordinates in $T_{2}\cap S$. As the set $T_1 \cap S$ is disjoint from the set $T_2 \cap S$, it follows that when $x \in \mathbb{F}_p^S$ is uniformly random, $\|(f^{=T_1})_{S\to x}\|_{2}$
and $\|(f^{=T_2})_{S\to x}\|_{2}$ are independent random variables,
and therefore 
\begin{align*}
\|\mathrm{E}_{\overline{S}}\left(f^{=T_{1}}f^{=T_{2}}\right)\|_{L^{2}\left(\mathbb{F}_{p}^{S}\right)} & = \sqrt{\mathbb{E}_{x \in \mathbb{F}_p^S}\left|\mathrm{E}_{\overline{S}}\left(f^{=T_1}f^{=T_2}\right)(x)\right|^2}\\
& \leq \sqrt{\mathbb{E}_{x \in \mathbb{F}_p^S}\left[\|(f^{=T_{1}})_{S \to x}\|^2_{2}\|(f^{=T_2})_{S\to x}\|^2_{2}\right]}\\
& = \sqrt{\mathbb{E}_{x \in \mathbb{F}_p^S}\left[\|(f^{=T_1})_{S\to x}\|^2_{2}\right] \mathbb{E}_{x \in \mathbb{F}_p^S}\left[\|(f^{=T_2})_{S\to x}\|^2_{2}\right]}\\
 & =\|f^{=T_{1}}\|_{2} \cdot \|f^{=T_{2}}\|_{2}.
\end{align*}
 Hence by Cauchy-Schwarz yet again, we obtain 
\begin{align*}
\left\|\sum_{\left(T_{1},T_{2}\right)\in\mathcal{F}_{1}}\left(f^{=T_{1}}f^{=T_{2}}\right)^{=S}\right\|_2^2 & =\sum_{\left(T_{1},T_{2}\right),\left(T_{3},T_{4}\right)\in\mathcal{F}_{1}}\left\langle \left(f^{=T_{1}}f^{=T_{2}}\right)^{=S},\left(f^{=T_{3}}f^{=T_{4}}\right)^{=S}\right\rangle \\
 & \le\sum_{\left(T_{1},T_{2}\right),\left(T_{3},T_{4}\right)\in\mathcal{F}_{1}}\|\left(f^{=T_{1}}f^{=T_{2}}\right)^{=S}\|_{2}\|\left(f^{=T_{3}}f^{=T_{4}}\right)^{=S}\|_{2}\\
 & \le\sum_{\left(T_{1},T_{2}\right),\left(T_{3},T_{4}\right)\in\mathcal{F}_{1}}\|f^{=T_{1}}\|_{2}\|f^{=T_{2}}\|_{2}\|f^{=T_{3}}\|_{2}\|f^{=T_{4}}\|_{2}\\
 & = \left(\sum_{\left(T_{1},T_{2}\right) \in\mathcal{F}_{1}}\|f^{=T_{1}}\|_{2}\|f^{=T_{2}}\|_{2}\right)^2\\
 & = \left(\sum_{\left(T_{1},T_{2}\right) :T_1 \Delta T_2=S}\|f^{=T_{1}}\|_{2}\|f^{=T_{2}}\|_{2}\right)^2.
\end{align*}
 Combining our upper bounds on the terms from $\mathcal{F}_{1}$ and on those from
$\mathcal{F}_{2}$, and applying Cauchy-Schwarz one more time, we obtain
\begin{align*}
\|g^{=S}\|_{2}^{2} &\le 2\left\|\sum_{\left(T_{1},T_{2}\right)\in\mathcal{F}_{2}}\left(f^{=T_{1}}f^{=T_{2}}\right)^{=S}\right\|_2^2 + 2\left\|\sum_{\left(T_{1},T_{2}\right)\in\mathcal{F}_{1}}\left(f^{=T_{1}}f^{=T_{2}}\right)^{=S}\right\|_2^2\\
&\leq 2\sum_{\varnothing\ne T\subseteq S}\left(2\left|S\right|\right)^{\left|T\right|}\|g_{T}^{=S}\|_{2}^{2}+2\left(\sum_{T_{1}\Delta T_{2}=S}\|f^{=T_{1}}\|_{2}\|f^{=T_{2}}\|_{2}\right)^{2},
\end{align*}
 as required.
\end{proof}
We are now ready to prove Proposition \ref{prop:Inductive lemma },
which we restate for the convenience of the reader.
\begin{prop*}
Let $f:\mathbb{F}_p^n \to \mathbb{C}$ be a function of degree at most $d$. Then 
\[
\|f\|_{4}^{4}\le 2\cdot9^{d}\|f\|_{2}^{4}+ 2\sum_{\emptyset \neq T\subseteq\left[n\right]}\left(4d\right)^{\left|T\right|}\|L_{T}\left[f\right]\|_{4}^{4}
\]
\end{prop*}
\begin{proof}[Proof of Proposition \ref{prop:Inductive lemma }]
 Let $g_{T}=L_{T}\left[f\right]^{2}$. Then the degree of $g_{T}$
is at most $2d$. We therefore have 
\begin{align*}
\sum_{S\subseteq\left[n\right]}\sum_{\varnothing\ne T\subseteq S}\left(2\left|S\right|\right)^{\left|T\right|}\|(g_{T})^{=S}\|_{2}^{2} & \le\sum_{T\ne\varnothing}\sum_{S\supseteq T}\left(4d\right)^{\left|T\right|}\|(g_{T})^{=S}\|_{2}^{2}\\
 & \leq \sum_{T\ne\varnothing}\left(4d\right)^{\left|T\right|}\|g_{T}\|_{2}^{2}\\
 & =\sum_{T\ne\varnothing}\left(4d\right)^{\left|T\right|}\|L_{T}\left[f\right]\|_{4}^{4}.
\end{align*}
Now define a function $f:\{0,1\}^n \to \mathbb{R}$ by
\[
\tilde{f}=\sum_{T\subseteq\left[n\right]}\|f^{=T}\|_{2}\, \chi_{T},
\]
 where $\chi_T(x) = (-1)^{\sum_{i \in T}x_i}$ for each $x \in \{0,1\}^n$ and $T \subset [n]$, i.e.\ the $\chi_T$ are the Fourier characters of the discrete cube (with the uniform measure). We then obtain, by Bonami's lemma (Lemma \ref{lem:bonami}, above), that
\[
\left(\sum_{T_{1}\Delta T_{2}=S}\|f^{=T_{1}}\|_{2}\|f^{=T_{2}}\|_{2}\right)^{2}=\|\tilde{f}\|_{4}^{4}\le9^{d}\|\tilde{f}\|_{2}^{4}=9^{d}\|f\|_{2}^{4}.
\]
 
\end{proof}
We now complete the proof of Theorem \ref{thm:hyp-prod}.
As mentioned above, the proof proceeds by induction on the degree, using the preceding proposition to obtain an inequality in terms of the $L_T[f]$'s, restricting $T$ to a random $x$
to go from $L_{T}\left[f\right]$ to $D_{T,x}\left[f\right]$ and
then applying the inductive hypothesis using the fact that the degree
of $D_{T,x}\left[f\right]$ is at most $d-\left|T\right| < d$, for $T \neq \emptyset$. 

\begin{proof}[Proof of Theorem \ref{thm:hyp-prod}.]
By induction on the degree of $f$. Trivially, the theorem holds for functions of degree zero. Now let $f$ be of degree $d$ and suppose that the theorem holds for all functions of degree less than $d$. By the preceding proposition, we have 
\[
\|f\|_{4}^{4}\le2\cdot9^{d}\|f\|_{2}^{4}+2\sum_{\varnothing\ne T\subseteq\left[n\right]}\left(4d\right)^{\left|T\right|}\|L_{T}\left[f\right]\|_{4}^{4}.
\]
 By the inductive hypothesis, for any $T \subset [n]$ with $T\neq \emptyset$, we have 
\begin{align*}
\|L_{T}\left[f\right]\|_{4}^{4} & =\mathbb{E}_{x\sim\mathbb{F}_{p}^{T}}\left[\|D_{T,x}\|_{4}^{4}\right]\\
 & \leq \mathbb{E}_{x\sim\mathbb{F}_{p}^{T}} \sum_{S' \subset [n] \setminus T} (100d)^{d-|T|} \mathbb{E}_{x' \sim \mathbb{F}_p^{S'}} \left[\|D_{S',x'}[D_{T,x}[f]]\|_2^4\right]\\
& = \sum_{S\supseteq T}\left(100d\right)^{d-\left|T\right|}\mathbb{E}_{x\sim\mathbb{F}_{p}^{S}}\left[\|D_{S,x}[f]\|_2^4\right]\\
& = \sum_{S\supseteq T}\left(100d\right)^{d-\left|T\right|}\mathbb{E}_{x\sim\mathbb{F}_{p}^{S}}\left[(I_{S,x}[f])^2\right],
\end{align*}
using the fact that $D_{T,x}[f]$ has degree at most $d-|T|$. Hence,
\begin{align*}
\|f\|_{4}^{4} & \leq 2\cdot9^{d}\|f\|_{2}^{2}+2\sum_{\varnothing\ne T\subseteq\left[n\right]}\left(4d\right)^{\left|T\right|}\sum_{S\supseteq T}\left(100d\right)^{d-\left|T\right|}\mathbb{E}_{x\sim\mathbb{F}_{p}^{S}}\left[(I_{S,x}[f])^2\right]\\
& \leq 2\cdot9^{d}\|f\|_{2}^{2} + \sum_{\emptyset \neq S\subseteq\left[n\right]}\mathbb{E}_{x\sim\mathbb{F}_{p}^{S}}\left[(I_{S,x}[f])^{2}\right]\sum_{\varnothing\ne T\subseteq S}\left(\left(100d\right)^{d-\left|T\right|}\cdot 2\cdot\left(4d\right)^{\left|T\right|}\right)\\
 & \leq \sum_{S\subseteq\left[n\right]}\left(100d\right)^{d}\mathbb{E}_{x\sim\mathbb{F}_{p}^{S}}\left[(I_{S,x}[f])^2\right],
\end{align*}
as required.
\end{proof}

\subsection{Our hypercontractive inequality for global functions}
\label{sec:glob}
To obtain our hypercontractive inequality for global functions, we need to show that the smallness of the influences $I_{S,x}[f]$ (for sets $S$ of bounded size) is equivalent to the smallness of the 2-norms of the restrictions of $f$. This is accomplished by the following two lemmas, which relate the influences of a function to the 2-norms of the restrictions of the function. (We note that this is part of our proof that does not generalise very easily to the setting of functions on $\l(V,W)$.)

\begin{lem}
\label{lem:restriction global implies global} Let $f\colon\mathbb{F}_{p}^{n}\to\mathbb{C}$,
let $S\subseteq[n]$, and let $x\in\mathbb{F}_{p}^{S}$. Suppose that
$\|f_{T\to x_{T}}\|_{2}^{2}\le\epsilon$ for all subsets $T\subseteq S$.
Then $I_{S,x}[f]\le2^{2|S|}\epsilon.$
\end{lem}
\begin{proof}
Using the
triangle inequality and the Cauchy--Schwarz inequality, we have 
\begin{align*}
 I_{S,x}[f]& = \|D_{S,x}[f]\|_2^2 = \|(L_S[f])_{S \to x}\|_2^2 = \left\|\left(\sum_{T \subseteq S}(-1)^{|T|} \mathrm{E}_T[f]\right)_{S \to x}\right\|_2^2 = \left(\sum_{T \subseteq S} \|(\mathrm{E}_T[f])_{S \to x}\|_2\right)^2\\
 & \leq 2^{|S|} \sum_{T \subseteq S}\|(\mathrm{E}_T[f])_{S \to x}\|_2^2 = 2^{|S|}\sum_{T \subseteq S}\|(\mathrm{E}_{S \setminus T}[f])_{S \to x}\|_2^2 = 2^{|S|}\sum_{T \subseteq S} \|\mathrm{E}_{S \setminus T}[f_{T \to x_T}]\|_2^2 \leq 2^{|S|}\sum_{T \subseteq S}\|f_{T \to x_T}\|_2^2\\ & \leq 2^{2|S|}\epsilon,
\end{align*}
where the penultimate inequality follows from the fact that the averaging operators $\mathrm{E}_{S\setminus T}$ are contractions
with respect to 2-norms.
\end{proof}
We also have the following converse implication, which we note for interest and to motivate our proof-strategy (though it is not used in the sequel).
\begin{lem}
\label{lem:global implies restriction global} Let $f\colon\mathbb{F}_{p}^{n}\to\mathbb{C}$,
let $S\subseteq[n]$, and let $x\in\mathbb{F}_{p}^{S}$. Suppose that
$I_{T,x\left(T\right)}[f]\le\epsilon$ for all sets $T\subseteq S$.
Then $\|f_{S\to x}\|_{2}^{2}\le2^{2|S|}\epsilon.$
\end{lem}
\begin{proof}
We have $f=\sum_{T\subseteq S}\mathrm{E}_{S\setminus T}L_{T}[f]$ for all sets $S \subset [n]$,
as can be observed most easily by using the formula $f = \mathrm{E}_i[f]+L_i[f]$, together with induction on $|S|$. As in the proof of the previous lemma, by the triangle
inequality, Cauchy-Schwarz, and the fact that the averaging operators $\mathrm{E}_{S \setminus T}$ are contractions
with respect to 2-norms, we obtain 
\[
\|f_{S\to x}\|_{2}^{2}\le\left(\sum_{T\subseteq S}\|L_{T}[f]_{T\to x(T)}\|_{2}\right)^{2}\le2^{2|S|}\epsilon.
\]
\end{proof}

We can now obtain our hypercontractive inequality for
global functions. The following corollary says that if $f$ is a $\left(d,\epsilon\right)$-global function, then the $4$-norm of its degree-$d$ truncation can be upper-bounded in
terms of its $2$-norm (in fact, by the 2-norm of its degree-$d$ truncation). 
\begin{cor}
\label{cor:global-hyp-trunc}
Let
$f:\mathbb{F}_p^n \to \mathbb{C}$ be a $\left(d,\epsilon\right)$-global function.
Then 
\[
\|f^{\le d}\|_{4}^{4}\le\left(800d\right)^{d}\epsilon\|f^{\leq d}\|_{2}^{2}.
\]
\end{cor}
\begin{proof}
Write $g = f^{\leq d}$ and note that, by the Lemma \ref{lem:decomposition of generalised influences}, $g$ is $(d,\epsilon)$-global. Using Theorem \ref{thm:hyp-prod}, Lemma \ref{lem:restriction global implies global} and the fact that $I_{S,x}[g] = 0$ for all $|S| > d$, we have 
\begin{align*}
\|g\|_{4}^{4} & \le\left(100d\right)^{d}\sum_{S\subseteq\left[n\right]}\mathbb{E}_{x\sim\mathbb{F}_{p}^{S}}(I_{S,x}[g])^{2}\\
& \leq \max_{S \subset [n]:\ |S| \leq d,\ x \in \mathbb{F}_p^S} (I_{S,x}[g]) \cdot \left(100d\right)^{d}\sum_{S\subseteq\left[n\right]}\mathbb{E}_{x\sim\mathbb{F}_{p}^{S}}I_{S,x}[g]\\
 & \le\epsilon\left(400d\right)^{d}\sum_{S\subseteq\left[n\right]}\mathbb{E}_{x\sim\mathbb{F}_{p}^{S}}I_{S,x}[g]\\
 & =\epsilon\left(400d\right)^{d}\sum_{S\subseteq\left[n\right]}\|L_{S}\left[g\right]\|_{2}^{2}\\
 &= \epsilon\left(400d\right)^{d}\sum_{S}\sum_{T\supseteq S}\|g^{=T}\|_{2}^{2}.\\
 & \le\epsilon\left(800d\right)^{d}\sum_{T \subseteq [n]}\|g^{=T}\|_{2}^{2}\\
 & =\epsilon\left(800d\right)^{d}\|g\|_{2}^{2}.
\end{align*}
 
\end{proof}
When substituting in a value of $\epsilon$ which is $\approx \|f\|_{2}^{2}$,
Corollary \ref{cor:global-hyp-trunc} tells us that the $4$-norm and the $2$-norm are within a constant
depending on $d$ of one another. 

\subsection{Overview of the proof of small-set expansion in $\mathbb{F}_p^n$.}

To obtain from Corollary \ref{cor:global-hyp-trunc} our small-set expansion theorem, we use an argument based on H\"{o}lder's inequality to take advantage of the following
two facts. On the one hand, if $f\colon\left\{ 0,1\right\} ^{n}\to\left\{ 0,1\right\}$
is a sparse Boolean function (here, `sparse' means that $\mathbb{E}\left[f\right]$
is small), then $\|f\|_{r}=\mathbb{E}[f]^{1/r}$
increases rapidly as $r$ increases. On the other hand, when considering
a low-degree function $g$ that has sufficiently small influences,
we obtain by our hypercontractive inequality that $\|g\|_{4}$ and
$\|g\|_{2}$ are within a constant of one another. This can be used
to show that the projection $f^{\le d}$ (onto the space of functions of degree at most $d$)
of a sparse Boolean function $f$, has very small $2$-norm. The latter
is accomplished by applying H\"{o}lder's inequality, to obtain 
\[
\|f^{\le d}\|_{2}^{2}=\left\langle f^{\le d},f\right\rangle \le\|f^{\le d}\|_{4}\|f\|_{\frac{4}{3}}.
\]
As we have seen, $f^{\le d}$ inherits the globalness properties
of $f$, and therefore our hypercontractive inequality for global functions implies that $\|f^{\le d}\|_{4}$
is small; and $\|f\|_{\frac{4}{3}}=\mathbb{E}\left[f\right]^{\frac{3}{4}}$
is much smaller than the factor of $\sqrt{\mathbb{E}\left[f\right]}$ which one
would get from applying Cauchy-Schwarz (instead of H\"{o}lder).

\subsection{Obtaining small-set expansion}

The following lemma shows that if $f$ is Boolean and global, and $\mathbb{E}[f^2]$ is small, then
$\|f^{\le d}\|_{2}^{2}$ is much smaller than $\|f\|_{2}^{2}.$
\begin{cor}
\label{cor:Holder argument} Let $f\colon\mathbb{F}_{p}^{n}\to\left\{ 0,1\right\} $
be $\left(d,\epsilon\right)$-global. Then 
\[
\|f^{\le d}\|_{2}^{2}\le\epsilon^{\frac{1}{4}}\left(800d\right)^{d/4}\|f\|_{2}^2.
\]
\end{cor}
\begin{proof}
By Lemma \ref{lem:decomposition of generalised influences} and Corollary
\ref{cor:global-hyp-trunc}, we have 
\[
\|f^{\le d}\|_{2}^{2}=\left\langle f^{\le d},f\right\rangle \le\|f^{\le d}\|_{4}\|f\|_{\frac{4}{3}}\le\epsilon^{\frac{1}{4}}\left(800d\right)^{d/4}\|f\|_{2}^{2}.
\]
\end{proof}
By direct analogy with the Boolean $(p=2)$ case, for $0\leq \rho \leq 1$ we define the noise operator $T_\rho:L^2(\mathbb{F}_p^n) \to L^2(\mathbb{F}_p^n)$ by
$$(T_{\rho}f)(x) = \mathbb{E}_{y \sim N_{\rho}(x)}[f(y)]\quad \forall x \in \mathbb{F}_p^n,\ \forall f \in L^2(\mathbb{F}_p^n),$$
where the distribution $y \sim N_\rho(x)$ is defined as follows: independently for each coordinate $i \in [n]$, we set $y_i = x_i$ with probability $\rho$, and with probability $1-\rho$ we take $y_i \in \mathbb{F}_p$ uniformly at random (independently of $x_i$). It is easy to see that the character $\chi_{\gamma}$ is an eigenvector of $T_{\rho}$ with eigenvalue $\rho^{|\textrm{supp}(\gamma)|}$, for all $\gamma \in \mathbb{F}_p^n$, and this yields the following Efron-Stein formula for $T_{\rho}$:
$$T_{\rho}(f) = \sum_{S \subseteq [n]}\rho^{|S|} f^{=S}\quad \forall f \in L^2(\mathbb{F}_p^n).$$

Corollary \ref{cor:Holder argument} quickly yields the following.
\begin{cor}
 Suppose that $f:\mathbb{F}_p^n \to \mathbb{C}$ is $\left(d,\epsilon\right)$-global, and $0 \leq \rho \leq 1$. Then 
\[
\left\langle \mathrm{T}_{\rho}f,f\right\rangle \le\left(\rho^{d+1}+\epsilon^{\frac{1}{4}}\left(800d\right)^{d/4}\right)\|f\|_{2}^{2}.
\]

In particular if $\delta>0$ we may set $d=\lfloor \log_{\rho}\delta/2 \rfloor$ and
$\epsilon=\frac{\delta^{4}}{16}\left(800d\right)^{-d}$ to obtain
that if $f$ is $\left(\lfloor \log_{\rho}(\delta/2) \rfloor,\frac{\delta^{4}}{16}\left(800d\right)^{-d}\right)$-global, then 
\[
\left\langle \mathrm{T}_{\rho}f,f\right\rangle \le\delta\|f\|_{2}^{2}.
\]
 
\end{cor}
\begin{proof}
We have 
\[
\left\langle \mathrm{T}_{\rho}f,f\right\rangle =\sum_{S\subseteq\left[n\right]}\rho^{\left|S\right|}\|f^{=S}\|_{2}^{2}\le\|f^{\le d}\|_{2}^{2}+\rho^{d+1}\|f\|_{2}^{2}\le\left(\rho^{d+1}+\epsilon^{\frac{1}{4}}\left(800d\right)^{d/4}\right)\|f\|_{2}^{2},
\]
using Corollary \ref{cor:Holder argument}. 
\end{proof}

In the case that $f=1_{A}$ where $A \subset \mathbb{F}_p^n$, we note that
$$\Pr_{x \in A,\ y\sim N_{\rho}(x)}[y \in A] = \frac{\langle T_{\rho}1_{A},1_{A}\rangle}{\langle 1_{A},1_{A}\rangle},$$
where $x \sim A$ means that $x$ is chosen uniformly at random from $A$, and $N_{\rho}(x)$ is the `noised' distribution defined above, so we obtain the following small-set expansion theorem as a corollary. (For $A \subset \mathbb{F}_p^n$ and for $S \subseteq [n]$, $x \in \mathbb{F}_p^S$, we write $A_{S \to x}: = \{y \in \mathbb{F}_p^{\overline{S}}:\ (x,y) \in A\}$, and we equip $A_{S \to x}$ with the uniform measure on $\mathbb{F}_p^{\overline{S}}$.)

\begin{cor}[Small-set expansion]
\label{cor:ssetfp}
 Suppose that $A \subset \mathbb{F}_p^n$ is such that $1_A$ is $\left(d,\epsilon\right)$-global, or equivalently $\mu(A_{S \to x}) \leq \epsilon$ for all $|S| \leq d$ and all $x \in \mathbb{F}_p^S$. Then we
have 
\[
\Pr_{x \in A,\ y\sim N_{\rho}(x)}[y \in A] \le \rho^{d+1}+\epsilon^{\frac{1}{4}}\left(800d\right)^{d/4}.
\]
In particular, if $1_{A}$ is $\left(\lfloor \log_{\rho}(\delta/2) \rfloor,\frac{\delta^{4}}{16}\left(800d\right)^{-d}\right)$-global, then 
\[
\Pr_{x \in A,\ y\sim N_{\rho}(x)}[y \in A] \le \delta.
\]
\end{cor}

\section{Fourier analysis on $\mathcal{L}(V,W)$}
\subsection{The Fourier expansion of functions on $\l(V,W)$}
The material in this subsection is fairly standard, but we include it for completeness. Let $q$ be a prime power, and let $V$ and $W$ be finite-dimensional vector spaces over $\mathbb{F}_{q}$. We write $\mathcal{L}\left(V,W\right)$
for the (vector) space of linear maps from $V$ to $W$. As usual, we write $V^{*}$ for
the space of linear functionals from $V$ to $\mathbb{F}_{q}$. Every
map $A\in\mathcal{L}\left(V,W\right)$ then has a `dual' map $A^{*}\in\mathcal{L}\left(W^{*},V^{*}\right)$,
defined by $A^{*}(\varphi) = (v\mapsto\varphi\left(Av\right))$, for all $\varphi \in W^*$.
We view $\mathcal{L}\left(V,W\right)$ as an Abelian group under addition; $L^2(\l(V,W))$ therefore has an orthonormal basis of characters of the group, which we now derive.

In order
to do this, we first recall the {\em trace map} $\tau\colon\mathbb{F}_{q}\to\mathbb{F}_{p}$. Write $q=p^{s}$, where $p$ is prime; we define
$$\tau: \mathbb{F}_{q} \to \mathbb{F}_{p}; \quad
\tau(x)=x+x^{p}+\ldots+x^{p^{s-1}}\quad \forall x \in \mathbb{F}_q.$$
We note that $\tau(x+y)=\tau(x)+\tau(y)$ for all $x,y \in \mathbb{F}_q$, and that $\tau$ is surjective.

For each $X\in\l(W,V)$, we define 
\[
u_{X}:\l(V,W)\to\mathbb{C};\quad u_{X}(A)=\omega^{\tau(\Tr(XA))},
\]
where $\omega:=\exp(2\pi i/p)$. We equip $\l(V,W)$ with the uniform
measure. As usual, $L^{2}(\l(V,W))$ denotes the Hilbert space of all complex-valued
functions on $\l(V,W)$, equipped with the inner product 
\[
\langle f,g\rangle=\mathbb{E}[f\bar{g}].
\]
We recall that a character on a finite abelian group $A$ is a homomorphism
from $A$ to the unit circle $\{x\in\mathbb{C}:\,|x|=1\}$.
\begin{prop}
The functions $\{u_{X}:\ X\in\l(W,V)\}$ are the characters of the
Abelian group $\l(V,W)$.
\end{prop}
\begin{proof}
First we assert that each character $u_{X}$ is a homomorphism from
$\l(V,W)$ to the unit circle. Indeed, we have
\[
u_{X}(A+B)=\omega^{\tau(\Tr(X(A+B)))}=\omega^{\tau(\Tr(XA))}\omega^{\tau(\Tr(XB))}=u_{X}(A)u_{X}(B).
\]

Secondly, we show that the characters $\mu_{X}$ are all distinct.
Indeed, if $u_{X}=u_{Y}$, then $u_{X-Y}=u_{X}u_{Y}^{-1}=1$. Write
$Z=X-Y.$ Our goal is to show that $Z=0$.

Since $u_{Z}=1$, we have $\tau(\Tr(ZA))=0$ for all $A$. This implies
that $\Tr(ZA)=0$ for all $A$. Indeed, suppose that $\Tr(ZA)=\beta\ne0$
for some $A$. Then for all $\alpha\in\mathbb{F}_{q}$, we have $0=u_{Z}(\alpha A)=\tau(\alpha\beta)$,
which would imply that the trace map $\tau$ is identically 0, which
it is not. Now it is easy to see that if $\Tr(ZA)=0$ for all $A$, then $Z=0$. This completes the proof that that the
characters $u_{X}$ are all distinct.

As the number of characters $u_{X}$ we have found is equal to $|\l(V,W)|$, it follows that we have found all of the characters.
\end{proof}
\begin{cor}
The characters $\{u_{X}: X \in \l(W,V)\}$ form an orthonormal basis for $L^{2}(\l(V,W))$.
\end{cor}
\begin{proof}
We have 
\[
\langle u_{X},u_{Y}\rangle=\mathbb{E}[u_{X-Y}].
\]
Now if $X=Y$, then $u_{X-Y}$ is the constant 1 function, and so
its expectation is 1. If $X$ and $Y$ are distinct, then $Z:=X-Y\ne0$.
Let $B \in \l(V,W)$ be such that $u_{Z}(B)\ne1$. Choose a map $\mathbf{A}$ uniformly at random from $\l(V,W)$. Since the map $\mathbf{A}+B$ has the same distribution
as $\mathbf{A}$, we have 
\[
\mathbb{E}_{\mathbf{A}}u_{Z}(\mathbf{A})=\mathbb{E}_{\mathbf{A}}u_{Z}(\mathbf{A}+B)=u_{Z}(B)\cdot\mathbb{E}_{\mathbf{A}}u_{Z}(\mathbf{A}).
\]
Since $u_{Z}(B)\ne 1$, the expectation of the character $u_{Z}$ must
be zero, completing the proof.
\end{proof}
For $f:\l(V,W)\to\mathbb{C}$, defining 
\[
\hat{f}(X)=\langle f,u_{X}\rangle=q^{-\dim(V)\dim(W)}\sum_{A\in\l(V,W)}f(A)\overline{u_{X}(A}),
\]
we have the \emph{Fourier expansion} 
\[
f=\sum_{X\in\l(W,V)}\hat{f}(X)u_{X}.
\]

\subsection{The dictators in $L^2(\l(V,W))$.}

In $L^2(\l(V,W))$ there are two types of dictators. One type corresponds
to the value of a linear map $A \in \l(V,W)$ on a fixed vector in $V$ and the other type corresponds
to value of $A^{*}$ on a fixed vector in $W^*$. More precisely, we say a function $f$ on $\l(V,W)$ is a {\em dictator} if it is of the form $f(A) = 1_{\{Av=w\}}$ for $v \in V$ and $w \in W$, or of the form $f(A) = 1_{\{A^{*}\phi=\psi\}}$ for some $\phi \in W^*$ and $\psi \in V^*$. For $v\in V$,
we write $\mathcal{S}_{v}$ for the linear space of all linear maps $X\in\l(W,V)$ such that $\Image(X)\subseteq\mathrm{Span}(v)$. Similarly, for
$\phi\in W^{*}$ we write $\mathcal{S}_{\phi}$ for the space of all linear maps $X\in\mathcal{L}\left(W,V\right)$
such that $\Image(X^{*})\subseteq\mathrm{Span}(\phi)$. 
\begin{prop}
\label{prop:dictators} The following functions have the following
Fourier expansions.
\begin{enumerate}
\item For all $v\in V$ we have 
\[
1_{\{Av=0\}}=\frac{\sum_{X\in\mathcal{S}_{v}}u_{X}}{|\mathcal{S}_{v}|}.
\]
\item For all $\phi\in W^{*}$ we have 
\[
1_{\{A^{*}\phi=0\}}=\frac{\sum_{X\in\mathcal{S}_{\phi}}u_{X}}{|\mathcal{S}_{\phi}|}.
\]
\item Let $v\in V$, let $w\in W$, and let $B\in\l(V,W)$ be any linear map such that $B(v)=w$; then
\[
1_{\{Av=w\}}=\frac{\sum_{X\in\mathcal{S}_{v}}\overline{u_{X}(B)}u_{X}}{|\mathcal{S}_{v}|}.
\]
\item Let $\phi\in W^{*}$, let $\psi\in V^{*}$, and let $B\in\l(V,W)$ be any linear map such
that $B^{*}(\phi)=\psi$; then
\[
1_{\{A^{*}\phi=\psi\}}=\frac{\sum_{X\in\mathcal{S}_{\phi}}\overline{u_{X}(B)}u_{X}}{|\mathcal{S}_{\phi}|}.
\]
\end{enumerate}
\end{prop}
\begin{proof}
We only prove (1) and (3), as (2) and (4) are similar. First we prove
(1). Write $D=\mathbb{E}_{\mathbf{X}\sim\mathcal{S}_{v}}u_{\mathbf{X}}$.
Our goal is to show that $D(A)=1_{\{Av=0\}}$. We first show that the
the function $D$ is Boolean valued, i.e.\ it only takes the values $0$
and $1$. We then show that $D$ is 1 whenever $Av=0$. We then finish
the proof by showing that the expectation of $D$ on a random $\mathbf{A}$
is the probability that $\mathbf{A}v=0$. 

Choose $\mathbf{X},\mathbf{Y}$ uniformly and independently at random from
$\mathcal{S}_{v}$, and note that $\mathbf{X}+\mathbf{Y}$ is also distributed
uniformly on $S_{v}$. We have $D=\mathbb{E}_{\mathbf{X}\sim\mathcal{S}_{v}}\left[u_{\mathbf{X}}\right].$
We therefore have 
\[
D^{2}=\mathbb{E}_{\mathbf{X},\mathbf{Y}}[u_{\mathbf{X}}u_{\mathbf{Y}}]=\mathbb{E}_{\mathbf{X},\mathbf{Y}}[u_{\mathbf{X}+\mathbf{Y}}]=D,
\]
 as $\boldsymbol{X}+\boldsymbol{Y}$ is uniformly distributed on
$\mathcal{S}_{v}$. This proves that $D$ is Boolean. We also have
$D(A)=1$ whenever $Av=0$, as in this case $AX=0$ for every $X\in\mathcal{S}_{v}$.
Consequently, to show that $D=1_{\{Av=0\}}$ it suffices to show that
they have the same expectation. We have
\[
\mathbb{E}[D]=\hat{D}(0)=\frac{1}{|\mathcal{S}_{v}|}=\frac{1}{|W|}=\mathbb{E}[1_{\{Av=0\}}],
\]
as required. We now show (3). Let $D=1_{\{Av=0\}}$, as in (1). We have $1_{\{Av=w\}}=D(A-B)$.
Moreover, every character $u_{X}$ satisfies $u_{X}(A-B)=u_{X}(A)\overline{u_{X}(B)}$.
The Fourier formula for $1_{\{Av=w\}}$ follows. 
\end{proof}

\subsection{Degree of functions on $\l(V,W)$.}

We recall that a {\em $d$-junta} is a function $f:\l(V,W) \to \mathbb{C}$ whose value on $A$ depends
only upon $Au_{1},\cdots,Au_{i},A^{*}u_{i+1},\cdots,A^{*}u_{d}$, for
some fixed $u_{1},\ldots,u_{i}\in V,u_{i+1},\ldots,u_{d}\in W^{*}$.
We defined the \emph{degree} of a function $f:\l(V,W) \to \mathbb{C}$ to be the minimal integer $d$
such that $f$ is a sum of $d$-juntas. Let us temporarily call this
notion \emph{junta-degree}. The \emph{Fourier-degree} of $f$ is $\max\{\rank(X):\ \hat{f}(X)\neq0\}$.
The dictators $1_{\{Av=w\}}$ and $1_{\{A^{*}\phi=\psi}$ both have junta-degree
equal to one and Fourier-degree equal to one. In fact, this is a general phenomenon, as the following easy lemma shows.
\begin{lem}
\label{lem:degree} Let $f\in L^{2}(\l(V,W))$. Then the junta-degree
of $f$ is equal to the Fourier-degree of $f$.
\end{lem}
\begin{proof}
We first show that the junta-degree is at most the Fourier-degree.
For any $X \in \l(W,V)$ of rank $d$, the value of $AX$, and therefore also of
$u_{X}(A)$, depends only upon the values of $A$ on any basis for the
image of $X$. Hence, the character $u_{X}$ is a $d$-junta. Now
if $f$ is a function of Fourier-degree $d$, then it is a sum of
the $d$-juntas of the form $\widehat{f}(X)u_{X}$. So it has junta
degree at most $d$.

We now show that the Fourier-degree is at most the junta-degree. Suppose
that $f$ has junta-degree $d$. Then it is a linear combination of
$d$-juntas. Now each $d$-junta is a linear combination of products of $d$ dictators,
as if $f$ depends only on the values of $Au_{1},\ldots,Au_{i},A^{*}u_{i+1},\ldots,A^{*}u_{d}$,
then by definition it is a linear combination of products of $d$
dictators of the form 
\[
\prod_{j=1}^{i}1_{\{Au_{j}=s_{j}\}}\prod_{j=i+1}^{d}1_{\{A^{*}u_{j}=s_{j}\}}.
\]
Each dictator has Fourier-degree one, by Proposition \ref{prop:dictators}, and since $u_{X}u_Y = u_{X+Y}$ for all $X,Y$, a product of $d$ dictators has Fourier degree at most $d$. It follows that $f$ has Fourier-degree at most $d$, as required.
\end{proof}
From now on, we use the term \emph{degree} for
the junta-degree, as we know the Fourier-degree is the same. If $f:\l(V,W)\to\mathbb{C}$, we write 
\[
f^{=d}:=\sum_{X\in\l(W,V):\ \rank(X)=d}\hat{f}(X)u_{X}
\]
for its pure-degree-$d$ part,
and 
\[
f^{\leq d}:=\sum_{X\in\l(W,V):\ \rank(X)\leq d}\hat{f}(X)u_{X}.
\]
for its degree-$d$ truncation. We say that a function $f$ is of \emph{pure degree $d$} if $f=f^{=d}$.
It turns out that that some aspects of the theory of hypercontractivity for functions on $\l(V,W)$
works better for functions of pure degree $d$, than for general
functions of degree $d$.

\subsection{Restrictions}

Let $V_{1}\le V$ and let $W_{1}\le W$. There is a natural one-to-one correspondence between $\l(V/V_1,W)$ and the space of linear maps in $\l(V,W)$ whose kernel contains $V_1$ (this correspondence is given by composing $A \in \l(V/V_1,W)$ on the right with the natural quotient map $\mathcal{Q}_{V_1}:V \to V/V_1;\ v \mapsto v+V_1$). Similarly, there is a natural one-to-one correspondence between $\l(V/V_1,W_1)$ and the space of linear maps in $\l(V,W)$ whose kernel contains $V_1$ and whose image is contained in $W_1$. If $A\in\l(V/V_{1},W_{1})$ then
we write $A(V,W)$ for the corresponding linear map from $V$ to $W$ whose
image is contained in $W_{1}$ and whose kernel contains $V_{1}$. Similarly, if $A \in \l(V,W)$ such that $V_{1} \subseteq \ker(A)$
and $\Image(A) \subseteq W_{1}$, then we write $A(V/V_{1},W_{1})$
for the corresponding linear map from $V/V_{1}$ to $W_{1}$. If $V_1 \leq V$, $W_1 \leq W$ and $A \in \l(V,W)$, we write $A(V_1,W/W_1)$ for the linear map in $\l(V_1,W/W_1)$ obtained by restricting $A$ to $V_1$ and then composing it (on the left) with the natural quotient map $\mathcal{Q}_{W_1}$ from $W$ to $W/W_1$; in symbols, $A(V_1,W/W_1) := \mathcal{Q}_{W_1} \circ (A|_{V_1})$. We will also `compose' these notations whenever it makes sense. For example, let $V_{1}\le V_{2}\le V$ and let $W_{1}\le W_{2}\le W$. If $A\in\l(V,W)$ satisfies $A(V_{1}) \subseteq W_{1}$ and $A(V_{2}) \subseteq W_2$, then we let $A(V_{2}/V_{1},W_{2}/W_{1})$ denote the map
sending $v+V_{1}$ to $Av+W_{1}$. 

For $W_{1}\le W$ we let $W_1^{\circ}$ be the annhilator of $W_1$, i.e.\ $W_1^{\circ} := \{\phi \in W^*: \phi(w)=0\ \forall w \in W_1\}$.

For $T \in \l(V,W)$, we write $\Delta_T$ for the {\em shift operator}
$$\Delta_T:L^2(\l(V,W)) \to L^2(\l(V,W));\quad (\Delta_T(f))(A) = f(A+T)\quad \forall A \in \l(V,W),\quad \forall f \in L^2(\l(V,W)).$$

\subsection*{Restrictions of functions}

Let $f\colon\l(V,W)\to\mathbb{C}$. Given a subspace $V_{1}\le V$,
a subspace $W_{1}\le W$, and a map $T\in\l(V,W)$, we would like
to find a convenient way of restricting $f$ to the (affine linear) space of all maps
such that $A$ that agrees with $T$ on $V_{1}$ and $A^{*}$ agrees
with $T^{*}$ on $W_1^{\circ}$. (The reader should think of $V_1$ having small dimension and $W_1$ having small codimension, though the following definitions make sense however large these are.)
\begin{defn}
Let $V_{1}\le V$, let $W_{1}\le W$ and let $f:\l(V,W) \to \mathbb{C}$. The \emph{restriction of $f$
to $\l(V/V_{1},W_{1})$} is the function $f_{(V_{1},W_{1})}\colon\l(V/V_{1},W_{1})\to\mathbb{C}$
defined by by $f_{(V_{1},W_{1})}(A)=f(A(V,W))$ for all $A \in \l(V/V_1,W_1)$. The \emph{restriction of $f$ corresponding to $(V_1,W_1,T)$} is the function $f_{(V_{1},W_{1})\to T}:=(\Delta_{T}f)_{(V_{1},W_{1})} \in \l(V/V_1,W_1)$.
Equivalently, $f_{(V_{1},W_{1})\to T} \in \l(V/V_{1},W_{1})$ is defined by $f(A)=f(A(V,W)+T)$ for all $A \in \l(V,W)$.
\end{defn}
Note that the linear maps of the form $A(V,W)+T$ are exactly the linear maps $B \in \l(V,W)$ such that
$B$ agrees with $T$ on $V_{1}$ and $B^{*}$ agrees with $T^{*}$ on $W_1^{\circ}$. Note also that by definition, $f_{(V_1,W_1)} = f_{(V_1,W_1) \to 0}$.

\subsection{Fourier formulas for the restrictions}

We now discuss how restrictions look in terms of their Fourier expansions.
In order to do that, we need to find the restriction of each
character $u_{X}$.
\begin{lem}
\label{lem:restriction of characters} Let $V_{1}\le V$, let $W_{1}\le W$,
let $X\in\l(W,V)$, and let $Y=X\left(W_{1},V/V_{1}\right)$, i.e.\ $Y$ is the linear map obtained by restricting the domain of $X$ to $W_1$, and then composing on the right with the quotient map $V \to V/V_1$. Then 
\[
(u_{X})_{(V_{1},W_{1})\to T}=u_{X}(T)u_{Y}.
\]
\end{lem}
\begin{proof}
We have $\Delta_{T}(u_{X})=u_{X}(T)u_{X}$, so it suffices to show
that $(u_{X})_{(V_{1},W_{1})}=u_{Y}$. Let $A\in\l(V/V_{1},W_{1})$.
It is easy to check that 
\[
\mathrm{Tr}\left(A\cdot Y\right)=\mathrm{Tr}\left(A\left(V,W\right)\cdot X\right)
\]
 for all $A \in \l(V/V_1,W_1)$ (see Lemma \ref{Lem:Traces} in the Appendix for the details, if necessary), and therefore $(u_{X})_{(V_{1},W_{1})}=u_{Y}$, as required.
\end{proof}
As a consequence, we obtain the following Fourier expansion of the
restriction of a function.
\begin{lem}
\label{lem:Fourier for restriction} Let $f \in L^2(\l(V,W))$, let $V_1 \leq V$ and let $W_1 \leq W$. Let $g:=f_{(V_{1},W_{1})\to T}$. Then 
\[
\hat{g}(Y)=\sum_{X\in\l(W,V):\,X(W_{1},V/V_{1})=Y}\hat{f}(X)u_{X}(T)
\]
for all $Y\in\l(W_{1},V/V_{1})$.
\end{lem}
\begin{proof}
Both restrictions, and taking Fourier coefficients, are linear operators.
Therefore, it suffices to consider the case where $f$ is a character
$u_{X}$. The proof for characters $u_{X}$ follows from Lemma \ref{lem:restriction of characters}.
\end{proof}

\subsection{Laplacians and derivatives of functions on $\l(V,W)$.}

The proof of our conditional hypercontractive inequality for functions on $\l(V,W)$  will rely (as in the product-space setting) on induction on the
degree of a function. It is therefore crucial to define the right
notion of the `derivative' of a function on $\l(V,W)$. We
will define derivatives with the following desirable properties.
\begin{enumerate}
\item The derivatives are linear operators.
\item The derivatives of order $i$ reduce the degree by at least $i$.
\item The composition of a derivative of order $i$ with a derivative of
order $j$ is a derivative of order $i+j$.
\item The 2-norms of the derivatives of a function are a measure of its
globalness.
\end{enumerate}

\subsection*{Laplacians}

\begin{defn}
For a subspace $V_{1}\leq V$, we define the {\em Laplacian corresponding to $V_1$} by
\[
L_{V_{1}}[f]=\sum_{X\in\l(W,V):\,\Image(X)\supseteq V_{1}}\hat{f}(X)u_{X}
\]
for all $f \in L^2(\l(V,W))$. In a dual way, for a subspace $W_1 \leq W$, we define the {\em Laplacian corresponding to $W_1$} by 
\[
L_{W_{1}}[f]=\sum_{X\in\l(W,V):\,\mathrm{Ker}(X)\subseteq W_{1}}\hat{f}(X)u_{X}
\]
for all $f \in L^2(\l(V,W))$.
\end{defn} 
Note that the Laplacians $L_{V_1}$ and $L_{W_1}$ are perhaps the most straightforward analogues of the Laplacians $L_S:L^2(\mathbb{F}_p^n) \to L^2(\mathbb{F}_p^n)$ (for $S \subset [n]$) which we used in the product-space setting. Unfortunately however, they do not satisfy the desirable property (2) above.

So, we also need to define `hybrid' Laplacians, of the form $L_{V_{1},W_{1}}$ for $V_1 \leq V$ and $W_1 \leq W$. The following definition might not seem terribly natural at first sight, but it is precisely what we need for the corresponding derivatives to satisfy the desirable property (2) above. (These
corresponding derivatives will soon be defined as restrictions of the `hybrid' Laplacians.)
\begin{defn}
For a character $u_X$ (for $X \in \l(W,V)$), we define $L_{V_{1},W_{1}}\left[u_{X}\right]: = u_{X}$ if $\Image(X)\supseteq V_{1}$ and $X^{-1}\left(V_{1}\right)\subseteq W_{1}$; otherwise we set $L_{V_{1},W_{1}}\left[u_{X}\right]:=0$.
We then extend linearly, defining
\end{defn}
\[
L_{V_{1},W_{1}}\left(f\right):=\sum_{{X\in\l(W,V):\atop \Image(X)\supseteq V_{1},\ X^{-1}(V_{1})\subseteq W_{1}}}\hat{f}\left(X\right)u_{X}
\]
for all $f \in L^2(\l(V,W))$. We call $L_{V_{1},W_{1}}$ a \emph{Laplacian of order
$i$} if $\dim(V_{1})+\codim(W_{1}) = i$.

\subsection*{Derivatives and influences}

\begin{defn} The {\em derivative} $D_{V_{1},W_{1},T}$ is the linear operator from $L^{2}(\l(V,W))$
to $L^{2}(\l(V/V_{1},W_{1}))$ defined by $D_{V_{1},W_{1},T}(f) = (L_{V_{1},W_{1}}[f])_{(V_{1},W_{1})\to T}$, for all $f \in L^2(\l(V,W))$.
We call the operator $D_{V_{1},W_{1},T}$ a \emph{derivative of order
$i$} if $i=\dim(V_{1})+\mathrm{codim}(W_{1})$. For brevity, we write
$D_{V_{1},W_{1}}: = D_{V_1,W_1,0}$. We define the {\em influence
of $(V_{1},W_{1})$ at $T$} by
\[
I_{(V_{1},W_{1},T)}[f]:=\|D_{V_{1},W_{1},T}[f]\|_{2}^{2}.
\]
\end{defn}

In Lemma \ref{lem:derivatives}, we will show that the $i$-order derivatives
reduce the degree by at least $i$. Note that the derivatives are indeed linear operators, since the Laplacians and the restriction operators are linear.

In Proposition \ref{prop:composition of derivatives}, we will show that
composition of a derivative of order $i$ with a derivative of order
$j$ is a derivative of order $i+j$. In Proposition \ref{prop:influences measure globalness},
we will show that the influences are indeed a good measure of the globalness
of $f$. Hence, the derivatives $D_{V_1,W_1,T}$ will satisfy all four of our desirable properties above.

\subsection{$X$-Laplacians and $X$-derivatives}
When adapting our strategy in the product-space setting to that of functions on $\l(V,W)$, compositions of the form $L_{V_1} \circ L_{W_1}$ (for $V_1 \leq V$ and $W_1 \leq W$) will arise naturally in the first step, when we expand $\reallywidehat{f^2}(X)$ for $X \in \l(W,V)$. In order to utilise degree-reduction (and hence implement induction on the degree), we would like to express such compositions in terms of the hybrid Laplacians $L_{V_2,W_2}$. This is not completely straightforward. Indeed, let $X\in\mathcal{L}\left(W,V\right)$ with $V_{1} \subseteq \Image(X)$ and $\ker(X) \subseteq W_{1}$, or equivalently $(L_{V_1} \circ L_{W_1})(u_X) \neq 0$. Then we
may still have $\mathcal{L}_{V_{1},W_{1}}\left(u_{X}\right)=0$ (if $X^{-1}(V_1) \nsubseteq W_1$): the condition $X^{-1}(V_1) \subseteq W_1$ is strictly stronger than the condition $(V_1 \subseteq \Image(X)) \wedge \ker(X) \subseteq W_1)$, as indeed can be seen by considering the case $q=2$, $V=W=\mathbb{F}_2^2$, $V_1 = \Span\{e_1\}$, $W_1 = \Span\{e_2\}$ and $X$ being the identity operator. In order to express the composition $L_{V_{1}}\circ L_{W_{1}}$
in terms of the hybrid Laplacians $L_{V_{2},W_{2}}$, we need to define yet more Laplacian operators $L_{Y}$, and corresponding derivative
operators $D_{Y}$.

\subsection*{A poset on linear maps}

\begin{defn}
We define a poset on $\l(W,V)$ by
setting $X\le Y$ if $\mathrm{rank}(Y)=\mathrm{rank}(X)+\mathrm{rank}(Y-X)$.\end{defn}
Note that $X \leq Y$ if and only if $\Image(X)$ and $\Image(Y-X)$ form a direct sum which is equal to $\Image(Y)$. It is also easy to see that $X \leq Y$ if and only if we have both $\ker(X)+\ker(Y-X)=W$ and $\Image(X) \cap \Image(Y) = \{0\}$. (To see this, first observe that if $\Image(X) \cap \Image(Y-X) = \{0\}$, then $\ker(X) \cap \ker(Y-X) = \ker(Y)$: clearly we have $\ker(X) \cap \ker(Y) \subseteq \ker(Y)$, whereas if $w \in \ker(Y)$ then $Xw + (Y-X)w=0$, so $Xw=  -(Y-X)w \in \ker(X) \cap \ker(Y-X)$. Hence, if $\Image(X) \cap \Image(Y-X)=\{0\}$ then $\dim(\ker(X)+\ker(Y-X)) = \dim(\ker(X))+\dim(\ker(Y-X))-\dim(\ker(X) \cap \ker(Y-X)) = \dim(\ker(X))+\dim(\ker(Y-X))-\dim(\ker(Y)) = \dim(W) - (\rank(Y)-\rank(X)-\rank(Y-X))$, so under the condition $\Image(X) \cap \Image(Y-X) = \{0\}$ we have $\ker(X)+\ker(Y-X)=W$ if and only if $X \leq Y$.)

If $X\le Y$, then (with only a very slight abuse of notation) we write $Y=X\oplus\left(Y-X\right)$. We note that the use of this notation
is consistent with its (usual) use for an internal direct sum of linear maps, since if $X \leq Y$ then we have $\ker(X),\ker(Y-X) \supseteq \ker(Y)$ and $\ker(X)+\ker(Y-X)=W$, so writing $W = \ker(X)\oplus W_3$ where $W_3 \subseteq \ker(Y-X)$, we may express $Y$ as the (usual) internal direct sum of $(Y-X)|_{\ker X}$ and $X|_{W_3}$.

We also have the following useful (alternative) characterisation of when $X \leq Y$.

\begin{prop}
\label{prop:poset} Let $X,Y \in \l(W,V)$. Then $X\le Y$ if and only if $\Image(X)\le\Image(Y)$
and when setting $P=Y^{-1}(\Image(X))$, we have $X(w)=Y(w)$ for all $w \in P$.
\end{prop}
\begin{proof}
First suppose that $X\le Y$. Since $\Image(Y)\le\Image(X)+\Image(Y-X)$,
and since we have the equality 
\[
\dim(\Image(Y))=\dim(\Image(X))+\dim(\Image(Y-X)),
\]
we clearly have $\Image(X) \cap \Image(Y-X) = \{0\}$. Now let $w \in W$ such that $Yw\in\Image(X)$. Then 
\[
(Y-X)w\in\Image(X)\cap\Image(Y-X)=\{0\},
\]
 and therefore $Yv=Xv$. So indeed $X$ and $Y$ agree on $P$.

Now suppose that $\Image(X)\le\Image(Y)$ and that $X$ and $Y$ agree on $P$.
The desired property of the ranks of $Y,X$ and $Y-X$, will follow
once we have shown that $\Image(Y)=\Image(X)\oplus\Image(Y-X)$. In order
to do this, we must show that the image of $Y$ contains
the image of $Y-X$ and to show that the intersection of $\Image(X)$ and $\Image(Y-X)$ is trivial.

Let $v \in \Image(Y-X)$; then there exists $w \in W$ such that $(Y-X)w=v$. Since the image of
$Y$ contains the image of $X$, there exists $u \in W$ such that $Yu=Xw$. Since
$Y$ agrees with $X$ on $Y^{-1}(\Image(X))$, we have $Xu=Yu$. Hence
$(Y-X)(w-u)$ is equal to both $Y(w-u)$, as $Xw=Yu=Xu$, and to $(Y-X)w$,
as $Yu=Xu$. Hence, $v=(Y-X)w$ is in the image of $Y$, completing the
proof that $\Image(Y-X)\le\Image(Y)$.

Suppose now that $v=Xu=(Y-X)w$ is an element in the intersection
of the images of the maps $X$ and $Y-X$. Rearranging, we obtain
that $X(u+w)=Yw$ lies in the image of $X$, and so $w\in Y^{-1}(\Image(X))$,
which implies that $Yw=Xw$. This in turn yields $v=(Y-X)w=0$,
as desired.
\end{proof}

We now demonstrate that the relation $\le$
is indeed a partial order on $\l(V,W)$.
\begin{prop}
The relation $\le$ is a partial order on $\l(V,W)$. 
\end{prop}
\begin{proof}
The only non-obvious property is transitivity. Suppose that $X\le Y$
and that $Y\le Z$. Then we obviously have $\Image(X)\leq\Image(Z)$,
and (by the preceding lemma) it remains only to show that $X$ and $Z$ agree on $Z^{-1}(\Image(X))$.
Since $Y$ and $Z$ agree on $Z^{-1}(\Image(Y))$ they also agree
on $Z^{-1}(\Image(X))$. Now $Z(Z^{-1}\Image(X),\Image(X))$ is surjective,
and therefore so is $Y(Z^{-1}(\Image(X)),\Image(X))$, implying that
$Z^{-1}(\Image(X))\subseteq Y^{-1}(\Image(X))$. Thus, $X$ agrees
with $Y$ and therefore with $Z$ on $Z^{-1}(\Image(X))$, as required.
\end{proof}

We are now ready to define the $X$-Laplacians and the $X$-derivatives.
\begin{defn}
For $X\in\l(W,V)$, we define $L_{X}:L^2(\l(V,W))\to L^2(\l(V,W))$ by $L[f]=\sum_{Y\ge X}\hat{f}(Y)u_{Y}$, for each $f \in L^2(\l(V,W))$. We call $L_{X}$
a \emph{Laplacian of order $\rank(X)$}. Letting $V_1: = \Image(X)$ and $W_1: = \ker(X)$, we define corresponding
derivatives $D_{X,T}:L^2(\l(V,W)) \to L^2(\l(V/V_1,W_1))$ by
\[
D_{X,T}[f]=(L_{X}[f])_{(V_{1},W_{1})\to T}\quad \forall f \in L^2(\l(V,W)),
\]
 and for brevity we write
\[
D_{X}[f]:=D_{X,0}[f]=(L_{X}[f])_{\left(V_{1},W_{1}\right)},
\]
 where we recall that $g_{\left(V_{1},W_{1}\right)}: = g_{(V_1,W_1)\to0}$ for each $g$.
\end{defn}

\subsection{How derivatives behave with respect to degree-decompositions and
compositions}
\label{subsec:Laplacians and restrictios relations}

We are now ready to show that the $i$-order derivatives decrease
the degree by at least $i$. We need the following linear-algebraic lemma.

\begin{lem}
\label{Lem:Linear algebraic} Suppose that $\Image(X)\supseteq V_{1}$
and that $X^{-1}(V_{1})\subseteq W_{1}$. Then $X(W_{1},V/V_{1})$
is of rank 
\[
\mathrm{rank}(X)-\dim(V_{1})-\mathrm{codim}(W_{1}).
\]
\end{lem}
\begin{proof}
First of all, since $\Image(X)\supseteq V_{1}$, we have 
\[
\mathrm{rank}(X)=\mathrm{rank}(X(W,V/V_{1}))+\dim(V_{1}).
\]
 Now since the kernel of the map $X(W,V/V_{1})$ is contained in $W_{1}$, we have $\dim(\ker(X(W,V/V_1)) = \dim(\ker(X(W_1,V/V_1))$ and therefore 
\[
\mathrm{rank}(X(W_{1},V/V_{1}))=\mathrm{rank}(X(W,V/V_{1}))-\mathrm{codim}(W_{1}).
\]
Substituting the first equality into the second proves the lemma.
\end{proof}

\begin{lem}
\label{lem:derivatives} Let $V_{1}\le V$, let $W_{1}\le W$ and let $i=\dim(V_{1})+\mathrm{codim}(W_{1})$.
For each $f:\l(V,W) \to \mathbb{C}$ and each $d \in \mathbb{N} \cup \{0\}$, we have 
\[
D_{V_{1},W_{1},T}[f^{=d}]=(D_{V_{1},W_{1},T}[f])^{=d-i}.
\]
\end{lem}
\begin{proof}
Since both the left-hand side and the right-hand side are linear in
$f$, it is enough to prove the lemma in the case where $f$ is some
character $u_{X}$. If $\Image(X) \supseteq V_1$ and $X^{-}(V_1) \subset W_1$, then $$D_{V_{1},W_{1},T}[u_{X}] = u_{X}(T)u_{X(W_{1},V/V_{1})},$$
in which case (by Lemma \ref{Lem:Linear algebraic}) we have $\rank(X(W_1,V/V_1)) = \rank(X)-i$; otherwise, $D_{V_{1},W_{1},T}[u_{X}] = 0$. The lemma follows.
\end{proof}

Similarly, we have the following.

\begin{lem}
\label{lem:linear-algebraic-2}
Suppose that $X,Y \in \l(W,V)$ with $Y \geq X$. Let $V_1 = \Image(X)$ and let $W_1=\ker(X)$. Then
$$\rank(Y(W_1,V/V_1)) = \rank(Y)-\rank(X).$$
\end{lem}
\begin{proof}
Since $Y \geq X$, we have $\ker(Y) \subset \ker(X) = W_1$. If $w \in \ker(Y(W_1,V/V_1))$ then $Y(w) \in V_1 = \Image(X)$ so $w \in Y^{-1}(\Image(X))$, and therefore, since $Y$ and $X$ agree on $Y^{-1}(\Image(X))$, and $w \in W_1 = \ker(X)$, we have $Yw = Xw = 0$. It follows that
$$\ker(Y(W_1,V/V_1)) = \ker(Y),$$
so
\begin{align*} \rank(Y(W_1,V/V_1)) & = \dim(W_1) - \dim(\ker(Y(W_1,V/V_1)))\\
&= \dim(W_1) - \dim(\ker(Y))\\
&= \dim(W) - \codim(W_1) - \dim(\ker(Y))\\
&= \rank(Y) - \codim(W_1)\\
&= \rank(Y)-\rank(X),
\end{align*}
as required.
\end{proof}
\begin{lem}
\label{lem:derivatives-2}
Let $T\in \l(V,W)$, let $X$ in $\l(W,V)$ with $\rank(X)=i$ and let $f:\l(V,W) \to \mathbb{C}$. Then
$$D_{X,T}[f^{=d}] = (D_{X,T}[f])^{=d-i}.$$
\end{lem}
\begin{proof}
It is enough to prove the lemma in the case where $f = u_Y$ for some $Y \in \l(W,V)$. Let $V_1 = \Image(X)$ and let $W_1 = \ker(X)$. If $Y \geq X$, then
$$D_{X,T}[u_Y] = u_Y(T) u_{Y(W_1,V/V_1)},$$
and we have $\rank(Y(W_1,V/V_1)) = \rank(Y)-i$ by the preceding lemma; otherwise, $D_{X,T}[u_Y]=0$. The lemma follows.
\end{proof}

We now show the third desired property of the derivatives, viz., that the composition
of an order-$i$ derivative with an order-$j$ derivative is an order-$(i+j)$
derivative.
\begin{prop}
\label{prop:composition of derivatives} Let $V_2 \le V_{1}\le V$, let $W_{1}\le W_{2}\le W$, let
$T \in \l(V,W)$ and let $S\in \l(V/V_{2},W_{2})$. Then
\[
D_{V_{1}/V_{2},W_{1},S}\circ D_{V_{2},W_{2},T}=D_{V_{1},W_{1},T+S(V,W)}.
\]
\end{prop}
\begin{proof}
Since the derivatives are all linear operators, it is enough to prove
the proposition for characters $u_{X}$. Also, we have $D_{V_2,W_2,T}[u_X] = u_X(T) D_{V_2,W_2}[u_X]$ and $D_{V_1/V_2,W_1,S}[u_Y] = u_Y(S) D_{V_1/V_2,W_1}[u_Y]$ for $Y = X(W_1,V/V_1)$, so
\begin{align*}
D_{V_{1}/V_{2},W_{1},S}\circ D_{V_{2},W_{2},T}[u_X] &= u_{X(W_1,V/V_1)}(S)u_X(T)D_{V_{1}/V_{2},W_{1}}\circ D_{V_{2},W_{2}}[u_X]\\
&= u_{X}(S(V,W))u_X(T) D_{V_{1}/V_{2},W_{1}}\circ D_{V_{2},W_{2}}[u_X]\\
& = u_{X}(S(V,W)+T)D_{V_{1}/V_{2},W_{1}}\circ D_{V_{2},W_{2}}[u_X],
\end{align*}
whereas
$$D_{V_{1},W_{1},T+S(W,V)}[u_X] = u_X(S(V,W)+T)D_{V_1,W_1}[u_X],$$
so we may reduce to the case $T=S=0$. Now, the operator $D_{V_{1},W_{1}}$ on the right-hand side
sends the character $u_{X}$ to $u_{X(W_{1},V/V_{1})}$ if
\begin{equation}\label{eq:basic}\Image(X) \supseteq V_1,\quad X^{-1}\left(V_{1}\right)\subseteq W_{1},\end{equation}
and to zero otherwise.
The operator $D_{V_{1}/V_{2},W_{1}}\circ D_{V_{2},W_{2}}$ on the
left-hand side sends $u_{X}$ to zero unless it satisfies the following properties:
\begin{equation}\label{eq:basic2}\Image(X) \supseteq V_2,\ X^{-1}\left(V_{2}\right)\subseteq W_{2},\quad \text{and}\quad \Image(Y) \supseteq V_1/V_2,\ Y^{-1}(V_1/V_2) \subseteq W_1,\ \text{where}\ Y: = X(W_2,V/V_2),\end{equation}
in which case it sends $u_X$ to
$$(u_{X(W_2,V/V_2)})_{(V_1/V_2,W_1)} = u_{X(W_1,(V/V_2)/(V_1/V_2))} = u_{X(W_1,V/V_1)}.$$
Therefore, it suffices to show that (\ref{eq:basic}) is equivalent to (\ref{eq:basic2}). This easy linear-algebraic fact is established
in Lemma \ref{lem:Establishing that the compositition of derivatives behaves nicely}.
\end{proof}

Things get more complicated when we restrict the composition of two
Laplacians. We now show that when restricting a composition of the form $L_{V_{1}}\circ L_{W_{1}}$,
we get a sum of compositions of derivatives, i.e.\ a sum of summands of the form $D_{X}\circ D_{V_{2},W_{2}}$,
where $V_{2}$ is contained in $V_{1}$, $W_{2}$ contains
$W_{1}$, and $X\in\mathcal{L}\left(W_{2},V/V_{2}\right)$ has $W_{1}$
as its kernel and $V_{1}/V_2$ as its image. We note that there is a natural one-to-one correspondence between linear maps $X\in\mathcal{L}\left(W_{2},V/V_{2}\right)$ with $\ker(X)=W_{1}$ and $\Image(X) = V_{1}/V_2$, and linear isomorphisms $\tilde{X}\colon W_{2}/W_{1}\overset{\sim}{\to}V_{2}/V_{1}$; we use this correspondence $X \leftrightarrow \tilde{X}$ in the statement below, as a notational convenience.
\begin{prop}
\label{prop:Taking restrictions of Laplacians} Let $V_{1}\le V$
and let $W_{1}\le W$. Then
\[
\left(L_{V_{1}}\circ L_{W_{1}}[f]\right)_{(V_{1},W_{1})\to T}=\sum_{V_{2}\subseteq V_{1}}\sum_{W_{2}\supseteq W_{1}}\sum_{\tilde{X}\colon W_{2}/W_{1}\overset{\sim}{\to}V_{1}/V_{2}}D_{X}\circ D_{V_{2},W_{2},T}.
\]
\end{prop}
\begin{proof}
Again it is enough to show that the equality holds in the special case where
$f=u_{Y}$ and $T=0$. This case of the proposition then follows immediately from Lemma \ref{lem:restrictiond of compositions of Laplacians}
in the Appendix. We restate it here, for the reader's convenience.
\begin{lem*}
Let $Y\in \l(W,V)$ be a linear map whose kernel is contained in $W_{1}$
and whose image contains $V_{1}$. Then there exist a unique triple $(W_2,V_2,X)$ such that
$W \geq W_{2}\geq W_{1}$, $V_{2}\leq V_{1} \leq V$, $X\in \l(W_{2}, V_{1}/V_{2})$ is a linear surjection with kernel $W_1$, $Y^{-1}(V_{2})\subseteq W_{2}$
and $X\le Y(W_{2},V/V_{2})$. 
\end{lem*}
\end{proof}
It will also be convenient for us to interchange the order of $D_{X}$ derivatives and $D_{V_{1},W_{1}}$ derivatives. This will be accomplished using the following relation.
\begin{prop}
\label{prop:switching the roles between the derivatives} Let $V_{1}\le V_{2}\le V$
and $W_{2}\le W_{1}\le W$. Suppose that $X\in\l(W,V)$ is a linear map whose
kernel is $W_{1}$ and whose image is $V_{1}$.
\begin{enumerate}
\item Let $\mathcal{F}_{1}$ be the set of subspaces $V_{3}$ of $V$ such that $V_{3}\le V_{2}$
and $V_{3}\oplus V_{1}=V_{2}$. Let $\mathcal{F}_{2}$ be the
set of subspaces $W_3$ of $W$ such that $W_{3}\ge W_{2}$ and 
\[
(W_{3}/W_{2})\oplus (W_{1}/W_{2})=W/W_{2}.
\]
 Then 
\[
D_{V_{2}/V_{1},W_{2}}\circ D_{X}=\sum_{V_{3}\in\mathcal{F}_{1},W_{3}\in\mathcal{F}_{2}}D_{X(W_{3},V/V_{3})}\circ D_{V_{3},W_{3}}.
\]
\item More generally, let $T\in\l(V/V_{1},W_{1})$ and $S\in\l(V,W)$ be linear maps.
Then 
\[
D_{V_{2}/V_{1},W_{2},T}\circ D_{X,S}=\sum_{V_{3}\in\mathcal{F}_{1},W_{3}\in\mathcal{F}_{2}}D_{X(W_{3},V/V_{3})}\circ D_{V_{3},W_{3},S+T(W,V)}.
\]
\end{enumerate}
\end{prop}
\begin{proof}
As in the previous propositions it is enough to prove the lemma for
characters $u_{Y}$ and to only prove (1) (where $T=S=0$). The proposition
now follows from Lemmas
\ref{Lem: swapping places 2} and \ref{lem:switching roles between derivatives lin-alg part} in the Appendix. 
\end{proof}

\subsection{The $\mathcal{E}_{v}$ operator.}

Recall that in the product-space setting, we defined the Laplacian via the
formula $f-E_{i}[f]$. There is no completely straightforward way to generalise
this notion. In fact, there is no completely obvious way of generalising the averaging
operator $E_{i}[f]$. Indeed, given a linear map $A$, one cannot simply
change its value on a vector $v$ without affecting its values on
other vectors. A possible attempt to generalise the Laplacian is to
complete $v$ to a basis $v=v_{1},v_{2},\ldots,v_{n}$ of $V$, to
leave the value of $v_{i}$ as it is for all $i\ge2$, while resampling
the value of $v$. The problem with this approach is that different
choices of the vectors $v_{2},\ldots,v_{n}$ yield different operators.
Our combinatorial version of the Laplacian is the average of all
such operators.
\begin{defn}
Given a subspace $V'\leq V$, we define a linear operator $\mathfrak{e}_{V/V'}:L^2(\l(V,W)) \to L^2(\l(V,W))$ by
\[
(\mathfrak{e}_{V/V'}[f])(A):=\underset{\mathbf{B}\sim\l(V/V',W)}{\mathbb{E}}f(A+\mathbf{B}(V,W))\quad \forall A \in \l(V,W),
\]
where the expectation is (as the notation suggests) over a uniform random element of $\l(V/V',W)$. 
We also define, for $v \in V$ with $v\neq 0$, a linear operator $\mathcal{E}_v:L^2(\l(V,W)) \to L^2(\l(V,W))$ by
$$\mathcal{E}_{v}[f]:=\underset{\mathbf{V'} \notni v}{\mathbb{E}}[\mathfrak{e}_{V/\mathbf{V}'}[f]],$$
where the expectation is over a uniformly random subspace $\mathbf{V}'\subseteq V$ of codimension one, that does not contain $v$. 
If $U$ is the one-dimensional subspace spanned by $v$, then we may write $\mathcal{E}_{U}$
in place of $\mathcal{E}_{v}$. (The operator $\mathcal{E}_{U}$ is easily seen to be independent
of the choice of the generator $v$.) For $v \in V \setminus \{0\}$, we define the {\em combinatorial Laplacian} $\frak{L}_v$ by
$$\mathfrak{L}_v[f] =f-\mathcal{E}_v[f]\quad \forall f \in L^2(\l(V,W)).$$
\end{defn}
We note that the combinatorial Laplacian $\mathfrak{L}_v$ is the Laplacian of the Markov chain on $\l(V,W)$ where at each step, we replace a matrix $A$ with $A+\mathbf{B}(V,W)$, where $\mathbf{B}$ is a uniform random element of $\l(V/\mathbf{V'},W)$ and $\mathbf{V'}$ is a uniform random codimension-one subspace of $V$ that does not contain $v$ (the random choices being independent of all previous steps).
\begin{lem}
\label{lem:EVV'}
For any $X \in \l(W,V)$, we have
$$\mathfrak{e}_{V/V'}[f]=\sum_{X:\,\Image(X)\subseteq V'}\hat{f}(X)u_{X}.$$
\end{lem}
\begin{proof}
Let $X\in\l(W,V)$. We have 
\[
\mathfrak{e}_{V/V'}\left[u_{X}\right]=\underset{\mathbf{B}\sim\l(V/V',W)}{\mathbb{E}}\Delta_{B}\left(u_{X}\right)=\left(\underset{\mathbf{B}\sim\l(V/V',W)}{\mathbb{E}}[u_{X}(\mathbf{B}(V,W))]\right)u_{X}.
\]
 Now if $\Image(X)\subseteq V'$, then $u_{X}(B)=\omega^{\tau(\Tr\left(XB\right))} = \omega^{\tau(\Tr(BX))}$
is identically 1 on $\{B(V,W):\ B \in \l(V/V',W)\}$, as $B(V,W)X=0$ for all $B \in \l(V/V',W)$. On the other hand, if
$\Image(X)$ is not contained in $V'$, then the map $B\mapsto\tau(\Tr(BX))$
is a nontrival (and therefore surjective) homomorphism from $\l(V/U,W)$ to $\mathbb{F}_{p}$.
Therefore, its image is uniformly distributed on $\mathbb{F}_{p}$,
and so
\[
\underset{\mathbf{B}\sim\l(V/V',W)}{\mathbb{E}}[u_{X}(\mathbf{B}(V,W))]=\underset{a \in \mathbb{F}_p} {\mathbb{E}}[\omega^{a}]=0.
\]
\end{proof}
By averaging the above, we obtain the following.
\begin{lem}
\label{lem:Fourier formula for combinatorial Laplacian}
For any $f \in L^2(\l(W,V))$ and $v \in V \setminus \{0\}$, we have
\[
\mathcal{E}_{v}[f]=\sum_{X \in \l(W,V):\ v \notin \Image(X)} q^{-\rank(X)} \hat{f}(X)u_X.
\]
\end{lem}
\begin{proof}
By Lemma \ref{lem:EVV'}, we only need to prove that if $v \notin \Image(X)$, then the probability
that $\Image(X)$ is contained in a random codimension-one
subspace $\mathbf{V'}$ not containing $v$, is equal to $q^{-\rank(X)}$.

\subsection*{Settling the case where $\protect\rank\left(X\right)=1$}

Choose a uniformly random subspace $\mathbf{V'}$ of $V$ with codimension one, that does not contain
$v$. First, we assert that for any $v'\notin\mathrm{Span}(v)$ we
have $\Pr[v'\in\mathbf{V'}]=\frac{1}{q}$. This will follow once we
show that $\mathbf{V'}$ contains exactly one element of the set $\left\{v'+\alpha v:\ \alpha \in \mathbb{F}_q\right\}$. To see this, simply choose a linear functional $\varphi\colon V\to\mathbb{F}_{q}$ with $\ker(\varphi) =\mathbf{V}'$ and $\varphi\left(v\right)=1$, and note that $\varphi\left(v'+\alpha v\right)$ is equal to zero for exactly one value of $\alpha \in \mathbb{F}_q$.

\subsection*{Settling the case of higher rank}

For $X$ of rank greater than one, we choose a basis $u_{1},\ldots,u_{d}$ of $\mathrm{rank(X)}$,
and we finish the proof by noting that that the probability that $u_{i}\in\mathbf{V}'$
given that the vectors $u_{1},\ldots,u_{i-1}$ are already in $\mathbf{V'}$
is still $1/q$. Indeed, this follows by appealing to case where the
rank is $1$ when inside the space $\mathcal{L}\left(V/U,W\right)$,
where $U=\mathrm{Span}\left(u_{1},\ldots,u_{d-1}\right).$
\end{proof}

\subsection{The dual operators $\mathcal{E}_{W'}$}

For $f\colon\mathcal{L}\left(V,W\right)\to\mathbb{C}$, we define $f^{*}\colon\mathcal{L}\left(W^{*},V^{*}\right) \to \mathbb{C}$, as usual,
by $f^{*}\left(A\right)=f\left(A^{*}\right)$ for each $A \in \l(W^*,V^*)$.

Given a subspace $W'\leq W$ of codimension $1$, we define the linear operator $\mathcal{E}_{W'}:L^2(\l(V,W)) \to L^2(\l(V,W))$
as follows. We let $\varphi \in W^*$ with $\varphi\ne0$ and $\varphi\left(W'\right)=0$,
and set 
\[
\mathcal{E}_{W'}\left[f\right]=\left(\mathcal{E}_{\varphi}\left[f^{*}\right]\right)^{*}\quad \forall f \in L^2(\l(V,W)).
\]
Dually to Lemma \ref{lem:Fourier formula for combinatorial Laplacian},
we obtain
\begin{lem}
\label{lem:lw formula}
For any $f \in L^2(\l(V,W))$ and any codimension-one subspace $W'$ of $W$, we have
\[
\mathcal{E}_{W'}[f]=\sum_{X\in \l(W,V):\,\mathrm{Ker}(X)+W'=W}q^{-\mathrm{rank}(X)}\hat{f}(X)u_{X}.
\]
\end{lem}

\section{Our conditional hypercontractive inequality for global functions on $\l(V,W)$:
statement, and overview of the proof.}

We are now ready to state our `conditional hypercontractive inequality' for functions on $\l(V,W)$, which is
a clear analogue of Theorem \ref{thm:hyp-prod}.
\begin{thm}
\label{thm:hyp in BS} There exists an absolute constant $C>0$ such
that the following holds. Let $f:\l(V,W) \to \mathbb{C}$ be a function of degree at most $d$. Then
\[
\|f\|_{4}^{4}\le q^{Cd^{2}}\sum_{V_1 \leq V,\ W_1 \leq W:\atop \dim(V_{1})+\mathrm{codim}(W_{1})\le d}\mathbb{E}_{T\sim\l(V,W)}[I_{V_{1},W_{1},T}[f]^{2}].
\]
\end{thm}
A crucial property of the statement of this theorem is that there
is no dependency on the dimensions of $V$ and $W$, anywhere in the
statement. Another nice feature of the statement is that it involves
only the generalized influences, which are defined in terms of the
restrictions of the `hybrid' Laplacians $L_{V_{1},W_{1}}$. (We note that, even though the
$X$-Laplacians $L_{X}$ play a crucial role in the proof, they do not appear in the theorem's statement.)

\subsection{Overview of proof of Theorem \ref{thm:hyp in BS} }

Proposition \ref{prop:Inductive lemma }, our `degree-reduction lemma' regarding functions on $\mathbb{F}_{p}^{n}$,
took the form 
\[
\|f\|_{4}^{4}\le C^{d}\|f\|_{2}^{2}+\sum_{S\subseteq\left[n\right]:\,S\ne\varnothing}\left(Cd\right)^{\left|S\right|}\|L_{S}\left[f\right]\|_{4}^{4},
\]
 for any function $f:\mathbb{F}_p^n \to \mathbb{C}$ of degree at most $d$. Our first step is to prove an $\l(V,W)$-analogue of this degree-reduction lemma, of the form
\[
\|f\|_{4}^{4}\le q^{Cd^{2}}\|f\|_{2}^{2}+\sum_{V_{1}\subseteq V,W_{1}\subseteq W:\,\left(V_{1},W_{1}\right)\ne\left(0,W\right)}q^{Cd\left(\dim\left(V_{1}\right)+\codim\left(W_{1}\right)\right)}\|L_{V_{1}}\circ L_{W_{1}}\left[f\right]\|_{4}^{4}.
\]
 The proof of this degree-reduction lemma follows closely the pattern of the proof in the $\mathbb{F}_{p}^{n}$ setting. Clearly, we have 
\[
\reallywidehat{f^{2}}\left(X\right)=\sum_{(Y,Z)\in (\l(W,V))^2:\atop Y+Z=X}\hat{f}\left(Y\right)\hat{f}\left(Z\right)
\]
for each $X \in \l(W,V)$. Now for each $X \in \l(W,V)$, we define $\mathcal{F}_{1} = \mathcal{F}_1(X)$ to be the set of all pairs
$(Y,Z) \in (\l(W,V))^2$ such that $Y+Z=X$ and such that there exist $Y'\le Y$ and $Z'\le Z$ with $Y'\oplus Z'=X.$ 

We then let $\mathcal{F}_{2}=\mathcal{F}_2(X)$ consist of all pairs $(Y,Z) \in (\l(W,V))^2$ 
with $Y+Z=X$ and with
\[
V_{1}:=\mathrm{image}\left(Y\right)\cap\mathrm{image}\left(Z\right)\cap\mathrm{image}\left(X\right)\ne\left\{ 0\right\} \quad \text{or} \quad
W_{1}:=\mathrm{ker}\left(Y\right)+\mathrm{ker}\left(Z\right)+\mathrm{ker}\left(X\right)\ne W.
\]
 It turns out that any pair $(Y,Z)$ with $Y+Z = X$ must belong either to $\mathcal{F}_1(X)$ or to $\mathcal{F}_2(X)$. Further, the summands in the above sum arising from pairs in $\mathcal{F}_{2}(X)$ (for some $X$) correspond to terms
appearing in $\|(L_{V_1}\circ L_{W_1})\left[f\right]\|_{4}^{4}$, for some $(V_1,W_1) \neq (\{0\},W)$. And finally, the
contribution from summands arising from pairs in $\mathcal{F}_{1}(X)$ (for some $X$) can be upper-bounded
via an encoding of $f$ inside the Boolean hypercube with the uniform measure (together with an application of the classical Bonami-Beckner-Gross hypercontractive inequality for the Boolean cube, as in the product-space setting). 

To deduce our conditional hypercontractive inequality from our degree-reduction lemma, we use induction on the degree. Assuming the inequality holds for all functions of degree less than $d$, we let $f:\l(V,W) \to \mathbb{C}$ be of degree $d$. We first take the terms of the form $\|L_{V_{1}}\circ L_{W_{1}}\left[f\right]\|_{4}^{4}$ that we obtain from applying the degree-reduction lemma,
and upper-bound their $4$-norms by upper-bounding the 4-norm of each
restriction $(L_{V_1} \circ L_{W_1}[f])_{\left(V_{1},W_{1}\right)\to T}$, and then taking expectations over a random $T$. We then use Proposition \ref{prop:Taking restrictions of Laplacians} to write the restriction
as a sum of compositions of derivatives of the form $D_{X}D_{V_{1},W_{1}, T}[f]$.
We then use the triangle inequality to obtain an upper-bound in terms of the individual 4-norms $\|D_{X}D_{V_1,W_1,T}[f]\|_{4}$. We then apply the inductive hypothesis to each of the functions $D_{X} D_{V_1,W_1,T}[f]$, to obtain an upper bound involving sums of terms of the form
$\|D_{V_{2},W_{2}}D_{X}D_{V_{1},W_{1},T}[f]\|_{2}^{4}$.
We then use the relations in Section \ref{subsec:Laplacians and restrictios relations} and the triangle inequality again to upper-bound this sum in terms of a sum of compositions of derivatives of the form $\|D_{X}D_{V_{3},W_{3},T}[f]\|_{2}^{4}$.
Finally, we complete the proof of the inductive step by getting rid of the $D_{X}$-derivatives, using a simple bound.

\subsection{Small-set expansion}

The proof of the small-set expansion theorem follows a similar pattern to the proof
in $\mathbb{F}_{p}^{n}$. The only difference is that it is less clear than in the product-space setting,
how one should relate the (combinatorial) globalness of a Boolean-valued function $f \in L^2(\l(V,W))$, to (analytic) globalness in terms of small influences of the function $f^{=d}$.
This is accomplished in Section \ref{sec:Equivalence-between-globalness notions}, by means of the combinatorial Laplacians defined in the previous section.

\section{The degree-reduction lemma.}

Our goal in this section is to prove the following lemma. We write $\mathcal{V}_{i}$ for the set ${V \brack i}$ of $i$-dimensional subspaces
of $V$; we write $\mathcal{W}_{j}$ for the set ${W \brack \dim(W)-j}$
of codimension-$j$ subspaces of $W$. 
\begin{lem}
\label{lem:hyp inductive step} Let $f\colon\l(V,W)\to\mathbb{C}$
be a function of degree at most $d$. Then 
\[
\frac{1}{162}\|f\|_{4}^{4}\le q{}^{3d^{2}}\|f\|_{2}^{4}+\sum_{\left(i,j\right)\ne\left(0,0\right)}\sum_{V_{1}\in\mathcal{V}_{i},\atop W\in\mathcal{W}_{j}}q{}^{7d(i+j)}\|L_{V_{1}}\circ L_{W_{1}}f\|_{4}^{4}.
\]
\end{lem}
We use the (Parseval) equality $\mathbb{E}[f^{4}]=\sum_{X\in\l(W,V)}|\widehat{f^{2}}(X)|^{2}$.
Our task is reduced to upper-bounding $|\widehat{f^{2}}(X)|^{2}$ for each $X$.
Note that
$$\reallywidehat{f^{2}}(X)=\sum_{Y+Z=X}\hat{f}(Y)\hat{f}(Z).$$
As outlined above, we now classify the pairs $(Y,Z) \in (\l(W,V))^2$ with $Y+Z=X$, into two types.
We define $\mathcal{F}_{1}(X)$ to be the set of all pairs $(Y,Z) \in \l(W,V)^2$ with $Y+Z=X$,
such that there exists $Y'\le Y$ and $Z'\le Z$ with $Y'\oplus Z'=X.$
The pairs in $\mathcal{F}_{1}\left(X\right)$ behave somewhat like Fourier coefficients in the Boolean cube, and we will be able to upper-bound their contribution in terms of $\|f\|_{2}^{4}$, using the classical Bonami-Beckner hypercontractive inequality on the discrete cube --- in fact, as in the product-space setting, Bonami's lemma (Lemma \ref{lem:bonami}, above) suffices.

We now define $\mathcal{F}_{2}\left(X\right)$ to be the set of all pairs $(Y,Z) \in \l(W,V)^2$ with $Y+Z=X$, such that either $\mathrm{image}\left(Y\right)\cap\text{image}\left(Z\right)\cap\mathrm{image}\left(X\right)\ne\varnothing$
or $\mathrm{ker}\left(Y\right)+\text{ker}\left(Z\right)+\text{ker}\left(X\right)\ne W.$
These are the terms that appear when expanding 
\[
\|L_{V_{1}}\circ L_{W_{1}}f\|_{4}^{4}=\sum_{X}\reallywidehat{L_{V_{1}}\circ L_{W_{1}}\left[f\right]^{2}}\left(X\right)
\]
 for some $V_{1}$ contained in the intersection of the images of
$X,Y,Z$ and some $W_{1}$ containing all of the kernels of the maps
$X,Y,Z$ (with at least one of $V_1$ and $W_1$ being nontrivial, i.e., with at least one of $V_1 \neq \{0\}$ and $W_1 \neq W$ holding). It turns out that every pair $(Y,Z) \in \l(W,V)^2$ with $Y+Z=X$ belongs to either $\f_1(X)$ or to $\f_2(X)$, as the following lemma shows.

\begin{lem}
\label{lem:F1}Let $X,Y,Z \in \l(W,V)$ with $Y+Z=X$. Then at least one of the following
holds. 
\begin{enumerate}
\item $\mathrm{image}\left(Y\right)\cap\text{image}\left(Z\right)\cap\mathrm{image}\left(X\right)\ne\varnothing$
\item $\mathrm{ker}\left(Y\right)+\text{ker}\left(Z\right)+\text{ker}\left(X\right)\ne W$
\item $\left(Y,Z\right)\in\mathcal{F}_{1}\left(X\right).$ 
\end{enumerate}
\end{lem}
\begin{proof}
We suppose that (1) and (2) do not hold, and show that (3) holds. Suppose that (1) and (2) do not hold. We claim there exists
a (unique) linear map $Z' \in \l(W,V)$ such that $Z'w=Zw$ for each $w\in\text{ker}\left(Y\right)$
and $Z'w=0$ for each $w\in\text{ker}\left(Z\right)+\text{ker}\left(X\right)$.
This follows from the fact that 
\[
\ensuremath{\left(\text{ker}\left(Z\right)+\ensuremath{\text{ker}\left(X\right)}\right)\cap\text{ker}\left(Y\right)=\text{ker}\left(Z\right)\cap\text{ker}\left(Y\right).}
\]
To see the latter, observe that the left-hand side trivially contains the right; on the other hand, if $w \in (\ker(X) + \ker(Z))\cap Y$, then we may write $w=w_1+w_2$ where $w_1 \in \ker(X)$, $w_2 \in \ker(Z)$ and $w_1+w_2=w \in \ker(Y)$; it follows that $Xw_1=Yw_1+Zw_1=0$, $Zw_2=0$ and $Yw_1+Yw_2 = 0$, so $X(w_1+w_2) = Zw_1$ and $Zw_1 \in \Image(X) \cap \Image(Y) \cap \Image(Z)=\{0\}$ and therefore $w \in \ker(Z) \cap \ker(Y)$. Note that the linear map $Z'$ is defined on all of $W$, since $\text{ker}\left(Z\right)+\text{ker}\left(X\right)+\text{ker}\left(Y\right)=W$.

By symmetry (between $Y$ and $Z$), there also exists a linear map $Y' \in \l(W,V)$ such that $Y'$ is equal to $Y$ on $\text{ker}\left(Z\right)$
and equal to zero on $\text{ker}\left(Y\right)+\text{ker}\left(X\right)$.

Let $X'=Y-Y'=Z'-Z$. Then it is easy to see that we have 
\[
\text{image}\left(Z\right)=\text{image}\left(Z'\right)+\text{image}\left(X'\right),
\]
\[
\text{image}\left(Y\right)=\text{image}\left(Y'\right)+\text{image}\left(X'\right)
\]
 and 
\[
\text{image}\left(X\right)=\text{image}\left(Z'\right)+\text{image}\left(Y'\right).
\]
 The fact that (1) does not hold, implies that all of these sums are actually direct sums.
Hence $Y'\oplus Z'=X$, $Y'\le Y$ and $Z'\le Z$. This shows that $\left(Y,Z\right)\in\mathcal{F}_{1}(X)$, so (3) holds, as required.
\end{proof}

\subsection{Upper bounding the $\mathcal{F}_{1}$ terms}

We now show that, when restricting to the pairs $\left(Y,Z\right)\in\mathcal{F}_{1}\left(X\right)$
for each $X$, we obtain (by Bonami's lemma) an inequality
of the form 
\[
\sum_{X}\left(\sum_{Y,Z\in\mathcal{F}_{1}\left(X\right)}\left|\hat{f}\left(Y\right)\hat{f}\left(Z\right)\right|\right)^{2}\le C_{d,q}\|f\|_{2}^{4},
\]
 where $C_{d,q}$ depends upon $d$ and $q$ alone.
\begin{lem}
\label{lem:F1helper}
Let $f\colon\mathcal{L}\left(V,W\right)\to\mathbb{C}$ be a function
of degree at most $d$. Then 
\[
\sum_{X}\left(\sum_{Y,Z\in\mathcal{F}_{1}\left(X\right)}\left|\hat{f}\left(Y\right)\hat{f}\left(Z\right)\right|\right)^{2}\le81q^{6d^{2}}\|f\|_{2}^{4}.
\]
 
\end{lem}
\begin{proof}
We start by encoding $f$ as a pair of functions on the discrete cube (with the uniform measure). Let $\mathcal{B} = \mathcal{B}(d)$
be the set of all linear maps in $\l(W,V)$ with rank at most $d$.
We encode our function $f$ by means of two functions on the $(2|\mathcal{B}|)$-dimensional discrete cube,
each of degree $2$. We define $f_{1},f_{2}\colon\{-1,1\}^{\mathcal{B}\times\left\{1,2\right\} }\to\mathbb{R}$
by 
\[
f_{1}:=\sum_{Y',X'\in \l(W,V):\atop Y'+X'=Y'\oplus X'}\left|\hat{f}\left(Y'\oplus X'\right)\right|x_{\left(Y',1\right)}x_{\left(X',2\right)}
\]
 and 
\[
f_{2}:=\sum_{Z',X'\in \l(W,V):\atop Z'+X'=Z'\oplus X'}\left|\hat{f}\left(Z'\oplus-X'\right)\right|x_{\left(Z',1\right)}x_{\left(X',2\right)},
\]
where the $x_I \in \{-1,1\}$ (for $I \in \mathcal{B} \times \{1,2\}$) are Boolean variables. Note that $\|f_{1}\|_{2}^{2}\le q^{d^{2}}\|f\|_{2}^{2}$ and similarly $\|f_{2}\|_2^2 \leq q^{d^2}\|f\|_2^2$. Indeed, each term $\hat{f}\left(Y\right)^{2}$ appears
in (the Parseval expansion of) $\|f_{1}\|_{2}^{2}$ the same number of times as the number of
ways of writing $Y=Y'\oplus X'$. Such (a way of writing $Y=Y'\oplus X'$) is uniquely determined by choosing a linear map
$Y'$ in $\l(V/\ker(Y),\Image(Y))$, since any $Y'\leq Y$ has $\Image(Y') \leq \Image(Y)$ and $\ker(Y') \geq \ker(Y)$,
and we have $|\l(V/\ker(Y),\Image(Y))|= q^{d^{2}}$.

Note that the summand on the left-hand side of the statement of the lemma is zero unless $\rank(X) \leq 2d$, since $f$ has degree at most $d$. So let $X \in \l(W,V)$ be of rank at most $2d$. For each $(Y,Z) \in \mathcal{F}_1(X)$, we may write $Y=Y'\oplus X'$ and $Z=Z'\oplus (-X')$,
where $Y'\oplus Z'=X$. Write $\mathcal{S}\left(X\right)$ for the
set of triples $\left(X',Y',Z'\right) \in \mathcal{B}^3$ with $Y'\oplus Z'=X$ and with
$X'$ having the property that $Y'+X'$ and $Z'+(-X')$ are direct sums.
Using Cauchy-Schwarz, and noting that the number of pairs $(Y',Z')$ such that $Y'\oplus Z' = X$ is at most $q^{(2d)^2}$ (by the same argument as above, but with $X$ in place of $Y$), we have
\begin{align*}
\left(\sum_{(Y,Z)\in\mathcal{F}_{1}\left(X\right)}\left|\hat{f}\left(Y\right)\hat{f}\left(Z\right)\right|\right)^2 & \leq\left(\sum_{(X',Y',Z')\in\mathcal{S}\left(X\right)}\left|\hat{f}\left(Y'\oplus X'\right)\hat{f}\left(Z'\oplus-X'\right)\right|\right)^{2}\\
 & \le q^{4d^{2}}\sum_{(Y',Z'):\atop Y'\oplus Z'=X}\left(\sum_{X':\,\left(X',Y',Z'\right)\in\mathcal{S}\left(X\right)}\left|\hat{f}\left(Y'\oplus X'\right)\hat{f}\left(Z'\oplus-X'\right)\right|\right)^{2}\\
 & = q^{4d^{2}}\sum_{(Y',Z'):\atop Y'\oplus Z'=X}\left(\sum_{X'\in\mathcal{L}\left(W,V\right):\atop (X',Y',Z') \in \mathcal{S}(X)}\hat{f_{1}}\left(\{(Y',1),(X',2)\}\right)\hat{f_{2}}\left(\{(Z',1),(X',2)\}\right)\right)^{2}\\
 & \leq q^{4d^{2}}\sum_{(Y',Z'):\atop Y'\oplus Z'=X}\left(\sum_{X'\in\mathcal{L}\left(W,V\right):\atop \rank(X') \leq d}\hat{f_{1}}\left(\{(Y',1),(X',2)\}\right)\hat{f_{2}}\left(\{(Z',1),(X',2)\}\right)\right)^{2}\\
 & =q^{4d^{2}}\sum_{(Y',Z'):\atop Y'\oplus Z'=X}\widehat{f_{1}f_{2}}\left(\{(Y',1),(Z',1)\})\right)^{2}.
\end{align*}
Therefore, by Cauchy--Schwarz and Lemma \ref{lem:bonami} (applied to the functions $f_{1}$ and $f_{2}$, both of which have degree two), we have
\[
\sum_{X}\left(\sum_{(Y,Z)\in\mathcal{F}_{1}\left(X\right)}\left|\hat{f}\left(Y\right)\hat{f}\left(Z\right)\right|\right)^{2}\le q^{4d^{2}}\|f_{1}f_{2}\|_{2}^{2}\le q^{4d^{2}}\|f_{1}\|_{4}^{2}\|f_{2}\|_{4}^{2}\le 81q^{4d^{2}}\|f_{1}\|_{2}^{2}\|f_{2}\|_{2}^{2}\le81q^{6d^{2}} \|f\|_2^4.
\]
 \end{proof}
\subsection{Upper bounding the $\mathcal{F}_{2}$ terms}

Let $\mathcal{V}_{i}\left(X\right)$ be the set of $i$-dimensional subspaces of $V$
contained in the image of $X$, and let $\mathcal{W}_{j}\left(X\right)$ be the set of codimension-$j$ subspaces of $W$ containing the kernel of $X$. We now
upper-bound the terms corresponding to pairs $\left(Y,Z\right)\in\mathcal{F}_{2}\left(X\right).$
These are terms that appear when expanding $\reallywidehat{L_{V_{1}}\left[f\right]^{2}}\left(X\right)$
for $V_{1}\in\mathcal{V}_{1}\left(X\right)$ and $\reallywidehat{L_{W_{1}}\left[f\right]^{2}}\left(X\right)$ for $W_{1}\in\mathcal{W}_{1}\left(X\right)$.
This suggests that an upper bound of the form 
\[
\left|\sum_{(Y,Z)\in\mathcal{F}_{2}(X)}\hat{f}(Y)\hat{f}(Z)\right|\le\sum_{V_{1}\in\mathcal{V}_{1}\left(X\right)}\reallywidehat{L_{V_{1}}\left[f\right]^{2}}\left(X\right)+\sum_{W_{1}\in\mathcal{W}_{1}\left(X\right)}\reallywidehat{L_{W_{1}}\left[f\right]^{2}}\left(X\right)
\]
might hold. This is not quite right, as there may be additional (negative) terms on the right-hand side. But the latter problem can easily be fixed, with an inclusion-exclusion type argument. 
\begin{lem}
\label{lem:F2 helper} 
\[
\left|\sum_{(Y,Z)\in\mathcal{F}_{2}(X)}\hat{f}(Y)\hat{f}(Z)\right|\le\sum_{\left(i,j\right)\ne\left(0,0\right)}\sum_{V_{1}\in\mathcal{V}_{i}\left(X\right),\atop W_{1}\in\mathcal{W}_{j}\left(X\right)}q^{2i^{2}+2j^{2}}|\reallywidehat{(L_{V_{1}}\circ L_{W_{1}}[f])^{2}}(X)|.
\]
 
\end{lem}
\begin{proof}
We define integers $(a_{i,j})$ such
that
\begin{equation}
\sum_{(Y,Z)\in\mathcal{F}_{2}(X)}\hat{f}(Y)\hat{f}(Z)=\sum_{(i,j)\ne\left(0,0\right)}\sum_{V_{1}\in\mathcal{V}_{i}\left(X\right),\atop W_{1}\in\mathcal{W}_{j}\left(X\right)}a_{i,j}\reallywidehat{(L_{V_{1}}\circ L_{W_{1}}[f])^{2}}(X),\label{eq:inclusion exclusion}
\end{equation}
 and such that $a_{ij}\le q^{2i^{2}+2j^{2}}$ for all $i,j$. Let $V'$ be the
intersection of the images of $X,Y$ and $Z$, and let $W'$ be the sum of their
kernels. The pair $\left(Y,Z\right)$ is in $\mathcal{F}_{2}$ if
and only if either $V'\ne\left\{ 0\right\} $ or $W'\ne W.$ We can
therefore write 
\begin{equation}
1_{\left(Y,Z\right)\in\mathcal{F}_{2}}=1_{V'\ne\left\{ 0\right\} }+1_{W'\ne W}-1_{V'\ne\left\{ 0\right\} }\cdot1_{W'\ne W}.\label{eq: first simple inclusion exclusion}
\end{equation}

Define a sequence of integers $a_{1},\ldots,a_{d}$ by $a_{1}=1$ and
$a_{k}=1-\sum_{i=1}^{k-1}a_{i}{k\brack i}_q$ for $2\leq k \leq d$. Observe that
\[
1_{V'\ne\left\{ 0\right\} }=\sum_{i=1}^{\dim\left(V'\right)}a_{i}{\dim\left(V'\right) \brack i}_q=\sum_{i=1}^{d}\sum_{V_{1}\in\mathcal{V}_{i}\left(X\right)}a_{i}1_{V'\supseteq V_{1}}.
\]
 Dually, we have 
\[
1_{W'\ne W}=\sum_{i=1}^{d}\sum_{W_{1}\in\mathcal{W}_{i}\left(X\right)}a_{i}1_{W'\subseteq W_{1}}.
\]

Therefore, defining $a_{i,j}:=-a_{i}a_{j}$ for $i,j\geq 1$, defining
$a_{0,j}:=a_{j}$ for all $j \geq 1$, defining $a_{i,0}:=a_{i}$ for all $i \geq 1$, and defining $a_{0,0}:=0$, we obtain by (\ref{eq: first simple inclusion exclusion})
that 
\[
1_{\left(Y,Z\right)\in\mathcal{F}_{2}}=\sum_{i,j}\sum_{V_{1}\in\mathcal{V}_{i}\left(X\right),\,W_{1}\in\mathcal{W}_{j}\left(X\right)}a_{i,j}1_{V'\supseteq V_{1}}1_{W'\subseteq W_{1}}.
\]
We therefore have
\begin{align*}
\sum_{(Y,Z)\in\mathcal{F}_{2}(X)}\hat{f}(Y)\hat{f}(Z) & = \sum_{(Y,Z)\in (\l(W,V))^2:\atop Y+Z=X}\hat{f}(Y)\hat{f}(Z)1_{(Y,Z) \in \mathcal{F}_2(X)}\\
& = \sum_{(Y,Z)\in (\l(W,V))^2:\atop Y+Z=X}\hat{f}(Y)\hat{f}(Z) \sum_{i,j}\sum_{V_{1}\in\mathcal{V}_{i}\left(X\right),\atop W_{1}\in\mathcal{W}_{j}\left(X\right)}a_{i,j}1_{V'\supseteq V_{1}}1_{W'\subseteq W_{1}}\\
& =  \sum_{i,j}\sum_{V_{1}\in\mathcal{V}_{i}\left(X\right),\atop W_{1}\in\mathcal{W}_{j}\left(X\right)}a_{i,j} \sum_{(Y,Z)\in (\l(W,V))^2:\atop Y+Z=X}\hat{f}(Y)\hat{f}(Z) 1_{V'\supseteq V_{1}}1_{W'\subseteq W_{1}}\\
& = \sum_{i,j}\sum_{V_{1}\in\mathcal{V}_{i}\left(X\right),\atop W_{1}\in\mathcal{W}_{j}\left(X\right)}a_{i,j} \sum_{(Y,Z)\in (\l(W,V))^2:\atop Y+Z=X}\hat{f}(Y)\hat{f}(Z) 1_{\Image(Y') \cap \Image(Z') \supseteq V_{1}}1_{\ker(Y')+\ker(Z') \subseteq W_{1}}\\
& = \sum_{i,j}\sum_{V_{1}\in\mathcal{V}_{i}\left(X\right),\atop W_{1}\in\mathcal{W}_{j}\left(X\right)}a_{i,j}\reallywidehat{(L_{V_{1}}\circ L_{W_{1}}[f])^{2}}(X)\\
& = \sum_{(i,j) \neq (0,0)}\sum_{V_{1}\in\mathcal{V}_{i}\left(X\right),\atop W_{1}\in\mathcal{W}_{j}\left(X\right)}a_{i,j}\reallywidehat{(L_{V_{1}}\circ L_{W_{1}}[f])^{2}}(X),
\end{align*}
using the facts that $\hat{f}\left(Y\right)\hat{f}\left(Z\right)$
appears in the expansion of $\reallywidehat{(L_{V_{1}}\circ L_{W_{1}}[f])^{2}}(X)$
precisely when $\Image(Y) \cap \Image(Z) \supseteq V_{1}$ and $\ker(Y),\ker(Z) \subseteq W_{1}$, and that $a_{0,0} = 0$.

To complete the proof of the lemma we must show that $|a_{i,j}|\le q^{2i^{2}+2j^{2}}$ for all $i$ and $j$.
To prove this, it suffices to show that $\left|a_{k}\right|\le q^{2k^{2}}$ for all $k$. We prove the latter by induction on $k$. It certainly holds when $k=0$. Suppose by induction that $|a_{i}| \leq q^{2i^2}$ for all $i < k$; then
\[
\left|a_{k}\right|\le1+\sum_{i=1}^{k-1}\left|a_{i}\right| {k \brack i}_q\le1+\sum_{i=1}^{k-1}q^{2i^{2}}q^{k\left(k-i\right)}\le q^{2k^{2}},
\]
 as required, completing the proof of the lemma.
\end{proof}
We quickly obtain the following consequence. 
\begin{lem}
\label{lem:F2} We have
\[
\left|\sum_{(Y,Z)\in\mathcal{F}_{2}(X)}\hat{f}(Y)\hat{f}(Z)\right|^{2}\le\sum_{\left(i,j\right)\ne\left(0,0\right)}\sum_{V_{1}\in\mathcal{V}_{i}\left(X\right),\atop W_{1}\in\mathcal{W}_{j}\left(X\right)}q^{7d(i+j)}|\reallywidehat{(L_{V_{1}}\circ L_{W_{1}}[f])^{2}}(X)|^{2},
\]
 and therefore 
\[
\sum_{X}\left|\sum_{(Y,Z)\in\mathcal{F}_{2}(X)}\hat{f}(Y)\hat{f}(Z)\right|^{2}\le\sum_{\left(i,j\right)\ne\left(0,0\right)}\sum_{V_{1}\in\mathcal{V}_{i},\atop W_{1}\in\mathcal{W}_{j}}q^{7d(i+j)}\|L_{V_{1}}\circ L_{W_{1}}[f]\|_{4}^{4}.
\]
 
\end{lem}
\begin{proof}
This follows from Lemma \ref{lem:F2 helper} and Cauchy--Schwarz, which
yield 
\begin{align*}
\left|\sum_{(Y,Z)\in\mathcal{F}_{2}(X)}\hat{f}(Y)\hat{f}(Z)\right|^{2} & \le\left(\sum_{\left(i,j\right)\ne\left(0,0\right)}\sum_{V_{1}\in\mathcal{V}_{i}\left(X\right),\atop W_{1}\in\mathcal{W}_{j}\left(X\right)}q^{2i^{2}+2j^{2}}|\reallywidehat{(L_{V_{1}}\circ L_{W_{1}}[f])^{2}}(X)|\right)^{2}\le\\
 & \le\sum_{\left(i,j\right)\ne\left(0,0\right)}\sum_{V_{1}\in\mathcal{V}_{i}\left(X\right),\atop W_{1}\in\mathcal{W}_{j}\left(X\right)}q{}^{7d(i+j)}|\reallywidehat{(L_{V_{1}}\circ L_{W_{1}}[f])^{2}}(X)|^{2}\\
 & \cdot\sum_{\left(i,j\right)\ne\left(0,0\right)}\sum_{V_{1}\in\mathcal{V}_{i}\left(X\right),\atop W_{1}\in\mathcal{W}_{j}\left(X\right)}q^{-7d(i+j)}q^{4i^{2}+4j^{2}},
\end{align*}
 which completes the proof of the first part of the lemma, as 
\begin{align*}
\sum_{\left(i,j\right)\ne\left(0,0\right)}\sum_{V_{1}\in\mathcal{V}_{i}\left(X\right),\atop W_{1}\in\mathcal{W}_{j}\left(X\right)}q^{-7d(i+j)}q^{4i^{2}+4j^{2}} & \leq \sum_{\left(i,j\right) \neq (0,0)} {d \brack i}_q {d \brack j}_q q^{-7d(i+j)}q^{4i^2+4j^2}\\
& \leq \sum_{\left(i,j\right) \neq (0,0)} q^{di} q^{dj} q^{-7d(i+j)}q^{4di+4dj}\\
& = \sum_{\left(i,j\right) \neq (0,0)} q^{-2d(i+j)}\\
& \leq \sum_{\left(i,j\right) \neq (0,0)} 2^{-2d(i+j)}\\
& < 1.
\end{align*}
 The last part of the lemma follows by using Parseval,
\[
\|g\|_{4}^{4}=\sum_{X}\left|\widehat{g^{2}}\left(X\right)\right|^{2}
\]
 for the functions $g$ of the form $L_{V_{1}}\circ L_{W_{1}}\left[f\right]$. 
\end{proof}

\subsection{Proof of Lemma \ref{lem:hyp inductive step}}

We are now ready to deduce Lemma \ref{lem:hyp inductive step} from
our upper bounds on the $\mathcal{F}_{1}$ terms and the $\mathcal{F}_{2}$
terms.
\begin{proof}[Proof of Lemma \ref{lem:hyp inductive step}]
Write $\f(X) = \{(Y,Z)\in \l(W,V)^2:\ Y+Z = X\}$. By Cauchy--Schwarz and Lemma \ref{lem:F1}, we obtain 
\[
\widehat{f^{2}}(X)^{2}\le2\left|\sum_{(Y,Z)\in\mathcal{F}\left(X\right)\setminus \mathcal{F}_2(X)}\hat{f}\left(Y\right)\hat{f}\left(Z\right)\right|^{2}+2\left|\sum_{(Y,Z)\in\mathcal{F}_{2}\left(X\right)}\hat{f}\left(Y\right)\hat{f}\left(Z\right)\right|^{2} \leq 2\left(\sum_{(Y,Z)\in\mathcal{F}_1\left(X\right)}\left|\hat{f}\left(Y\right)\hat{f}\left(Z\right)\right|\right)^{2}+2\left|\sum_{(Y,Z)\in\mathcal{F}_{2}\left(X\right)}\hat{f}\left(Y\right)\hat{f}\left(Z\right)\right|^{2}.
\]
 Summing over all values of $X$, and using Lemmas \ref{lem:F1helper} and
\ref{lem:F2} we obtain
\[
\|f\|_{4}^{4}\leq 162 q^{6d^{2}}\|f\|_{2}^{4}+2\sum_{\left(i,j\right)\ne\left(0,0\right)}\sum_{V_{1}\in\mathcal{V}_{i},\atop W_{1}\in\mathcal{W}_{j}}q{}^{7d(i+j)}\|L_{V_{1}}\circ L_{W_{1}}[f]\|_{4}^{4},
\]
 as required.
\end{proof}

\section{Proof of our conditional hypercontractive inequality for global functions}

In this section we prove Theorem \ref{thm:hyp in BS}. It takes the
form 
\[
\|f\|_{4}^{4}\le q^{O\left(d^{2}\right)}\left(\|f\|_{2}^{4}+\sum_{\left(V_{1},W_{1}\right)\ne\left(0,W\right)}\mathbb{E}_{T}\|D_{V_{1},W_{1},T}\left[f\right]\|_{2}^{4}\right)
\]
We can now be a bit more specific about the proof-strategy. The idea is to use the inequality 
\[
\|f\|_{4}^{4}\le 162q^{6d^{2}}\|f\|_{2}^{4}+\sum_{\left(i,j\right)\ne\left(0,0\right)}\sum_{V_{1}\in\mathcal{V}_{i},\atop W\in\mathcal{W}_{j}}q^{7d(i+j)}\|(L_{V_{1}}\circ L_{W_{1}})[f]\|_{4}^{4}
\]
 which we established in Lemma \ref{lem:hyp inductive step}.
We then take a restriction $\left(V_{1},W_{1},T\right)$ (w.r.t.\ to a random $T$) and
apply Proposition \ref{prop:Taking restrictions of Laplacians}, yielding a sum of compositions of derivatives of the form $D_{X}D_{V_{1},W_{1},T}[f].$
We then apply the induction hypothesis to each of the functions
$D_{X}D_{V_{1},W_{1},T}[f]$, which has a lower degree than $f$,
to show that 
\[
\|D_{X}D_{V_{1},W_{1},T}[f]\|_{4}^{4}
\]
 can be upper bounded by sums of terms of the form $\|D_{V_{2},W_{2},S}D_{X}D_{V_{1},W_{1},T}[f]\|_{2}^{4}$.
We then apply Lemma \ref{lem:switching roles between derivatives lin-alg part}
to show that $D_{V_{2},W_{2},S}D_{X}D_{V_{1},W_{1},T}[f]$
is a sum of (not too many) terms of the form 
\[
D_{Y}D_{V_{3},W_{3},T+S(V,W)}[f].
\]
 It remains to upper-bound sums of terms of
the form $\|D_{Y}D_{V_{3},W_{3}, T+S(V,W)}[f]\|_{2}^{4}$,
i.e.\ we need to get rid of the $D_{Y}$-derivatives, which we indeed manage to do.

\subsection{Replacing the Laplacians by derivatives}

For $V_{1}\le V$, $W_{1}\le W$, and $X\in \l(W_{1},V/V_{1})$
we write $|(V_{1},W_{1})|:=\dim(V_{1})+\codim(W_{1})$. We
also write $|(V_{1},W_{1},X)|:=|(V_{1},W_{1})|+\mathrm{rank}(X)$.
\begin{lem}
\label{lem:bounds on composition} Let $f:\l(V,W) \to \mathbb{C}$, let $V_{1}\in\mathcal{V}_{i}:={V \brack i}$
and let $W_{1}\in\mathcal{W}_{j}: = {W \brack \dim(W)-j}$. Then 
\[
\|L_{V_{1}}\circ L_{W_{1}}\left[f\right]\|_{4}^{4}\le q^{5\left(i^{2}+j^{2}\right)}\sum_{V_{2}\le V_{1},W_{2}\ge W_{1}}\sum_{X\colon W_{2}/W_{1}\overset{\sim}{\to}V_{1}/V_{2}}\mathbb{E}_{T}\|D_{X}\circ D_{V_{2},W_{2},T}[f]\|_{4}^{4}
\]
\end{lem}
\begin{proof}
By Proposition \ref{prop:Taking restrictions of Laplacians}, the triangle
inequality and H\"{o}lder's inequality, we have 
\begin{align*}
\|L_{V_{1}}\circ L_{W_{1}}\left[f\right]\|_{4}^{4} & =\mathbb{E}_{T}\left\Vert \left(L_{V_{1}}\circ L_{W_{1}}\left[f\right]\right)_{\left(V_{1},W_{1}\right)\to T}\right\Vert _{4}^{4}\\
 & =\mathbb{E}_{T}\left\Vert \sum_{V_{2}\le V_{1},W_{2}\ge W_{1}}\sum_{\tilde{X}\colon W_{2}/W_{1}\overset{\sim}{\to}V_{1}/V_{2}}D_{X}\circ D_{V_{2},W_{2},T}\left[f\right]\right\Vert _{4}^{4}\\
 & \leq\mathbb{E}_{T}\left(\sum_{V_{2}\le V_{1},W_{2}\ge W_{1}}\sum_{\tilde{X}\colon W_{2}/W_{1}\overset{\sim}{\to}V_{1}/V_{2}}\left\Vert D_{X}\circ D_{V_{2},W_{2},T}\left[f\right]\right\Vert _{4}\right)^{4}\\
 & \le\left(\sum_{V_{2}\le V_{1},W_{2}\ge W_{1}}\sum_{\tilde{X}\colon W_{2}/W_{1}\overset{\sim}{\to}V_{1}/V_{2}}1\right)^{3}\sum_{V_{2}\le V_{1},W_{2}\ge W_{1}}\sum_{\tilde{X}\colon W_{2}/W_{1}\overset{\sim}{\to}V_{1}/V_{2}}\mathbb{E}_{T}\|D_{X}\circ D_{V_{2},W_{2},T}[f]\|_{4}^{4}\\
 & \le q^{9\left(i^{2}+j^{2}\right)/2}\sum_{V_{2}\le V_{1},W_{2}\ge W_{1}}\sum_{\tilde{X}\colon W_{2}/W_{1}\overset{\sim}{\to}V_{1}/V_{2}}\mathbb{E}_{T}\|D_{X}\circ D_{V_{2},W_{2},T}[f]\|_{4}^{4},
\end{align*}
using the fact that for any $0 \leq k \leq \min\{i,j\}$, there are at most ${i \brack k}_q{j \brack k}_q q^{k^2} \leq q^{(i+j+k)k}$ triples $(V_2,W_2,\tilde{X})$ such that $V_2 \leq V_1$, $W_2 \geq W_1$, $\dim(V_1/V_2) = \dim(W_2/W_1)=k$ and $\tilde{X}:W_2/W_1 \to V_1/V_2$ is a linear isomorphism.
\end{proof}
By combining Lemma \ref{lem:bounds on composition} with Lemma \ref{lem:hyp inductive step}, we obtain the following.
\begin{lem}
\label{lem:after restricting}Let $f:\l(V,W) \to \mathbb{C}$ be a function of degree at most $d$.
Then 
\begin{align*}
\frac{1}{162}\|f\|_{4}^{4} & \le q^{6d^{2}}\|f\|_{2}^{4}+\sum_{V_{2}\le V,\ W_{2}\le W,\atop X\in \l(W_{2},V/V_{2}),\, |(V_{2},W_{2},X)|>0}q^{24d|(V_{2},W_{2},X)|}\mathbb{E}_{T}\|D_{X}\circ D_{V_{2},W_{2},T}[f]\|_{4}^{4}\\
\end{align*}
\end{lem}
\begin{proof}
The lemma follows immediately from Lemmas \ref{lem:hyp inductive step}
and \ref{lem:bounds on composition}. Indeed, they imply that
\[
\frac{1}{162}\|f\|_{4}^{4}\le q^{6d^{2}}\|f\|_{2}^{4}+\sum_{(V_{1},W_{1})\ne\left(0,W\right)}q^{7d|(V_1,W_{1})|}\sum_{V_{2}\le V_{1},W_{2}\ge W_{1}}\sum_{\tilde{X}\colon W_{2}/W_{1}\overset{\sim}{\to}V_{1}/V_{2}}q^{5d\left|(V_{1},W_{1})\right|}\mathbb{E}_{T}\|D_{X}\circ D_{V_{2},W_{2},T}[f]\|_{4}^{4}.
\]
 Now observe that each term in the sum on the right-hand side corresponds to a quintuple $(V_1,W_1,V_2,W_2,X)$ with $V_2 \leq V_1$, $W_2 \geq W_1$ and $X \in \l(W_2 \to V_1/V_2)$ being a linear surjection with kernel $W_1$, but in fact the triple $\left(V_{2},W_{2},X\right)$
uniquely determines $(V_{1},W_{1})$, since $V_{1}$ is the
preimage of $\Image(X)$ under the projection map
$V\to V/V_{2}$, and $W_{1}=\ker(X)$.
Moreover, we have $\ensuremath{\left|\left(V_{2},W_{2},X\right)\right|}\ensuremath{\ge\frac{\left|(V_{1},W_{1})\right|}{2}}$, and $|(V_2,W_2,X)| = 0$ only if $V_1 = 0$ and $W_1 = W$.
Hence, 
\[
\frac{1}{162}\|f\|_{4}^{4}\le q^{6d^{2}}\|f\|_{2}^{4}+\sum_{V_{2}\le V,W_{2}\le W,\atop X \in \l(W_2,V/V_2):\ |(V_2,W_2,X)| > 0} q^{24d\left|\left(V_{2},W_{2},X\right)\right|}\mathbb{E}_{T}\|D_{X}\circ D_{V_{2},W_{2},T}[f]\|_{4}^{4}.
\]
as required.
\end{proof}

As mentioned above, we will prove our conditional hypercontractive inequality by induction on the degree $d$. Our induction hypothesis will be as follows.

\textbf{Induction hypothesis.} Let $f:\l(V,W) \to \mathbb{C}$ be a function of degree at most $d'$, where $d' \leq d-1$.
Then 
\begin{equation}
    \label{eq:ind-hyp}
\|f\|_{4}^{4}\le q^{100(d')^{2}}\sum_{V_{1}\le V,W_{1}\le W}\mathbb{E}_{\mathbf{T}\sim\mathcal{L}\left(V,W\right)}\|D_{V_{1},W_{1},\mathbf{T}}[f]\|_{2}^{4}.
\end{equation}

Taking a function of degree $d$ and applying our inductive hypothesis to the lower-degree functions appearing on the right-hand side of the inequality in Lemma \ref{lem:after restricting}, will at first yield a version of the induction hypothesis for $f$ {\em with extra $D_X$-deriatives appearing}; the following sequence of lemmas will enable us to remove these extra derivatives, recoving the (genuine) induction hypothesis.

\begin{lem}
\label{lem:getting rid of D_x 2-norms}Let $f\colon\l(V,W)\to\mathbb{C}$
be a function of degree at most $d$. Then $\sum_{X\in\l(W,V)}q^{-4d\rank(X)}\|D_{X}[f]\|_{2}^{2}\le 2\|f\|_{2}^{2}.$ 
\end{lem}
\begin{proof}
We have
$$\sum_{X\in\l(W,V)}q^{-4d\rank(X)}\|D_{X}[f]\|_{2}^{2}=\sum_{X\in\l(W,V)}q^{-4d\rank(X)}\sum_{Y\ge X}|\hat{f}(Y)|^{2}=\sum_{Y\in\l(W,V)}|\hat{f}(Y)|^{2}\sum_{X \leq Y}q^{-4d\rank(X)},$$
so it suffices to prove the following.
\begin{claim}
Let $Y\in \l(W,V)$ be of rank at most $d$. Then the number of $X\in\l(W,V)$
with $X\le Y$ and $\rank(X)=k$ is at most $q^{3kd}$. 
\end{claim}
\begin{proof}[Proof of claim.]
Let $W_{1} = \ker(Y)$ and let $V_{1}=\Image(Y)$. If $X \leq Y$, then $\ker(X) \supset W_{1}$ and $\Image(X) \subset V_1$. There are therefore at most $q^{kd}$ choices for $\Image(X)$ and at most $q^{kd}$ choices for $\ker(X)$, and given these, there are at most $q^{k^2} \leq q^{kd}$ choices for $X$. This concludes the proof of the claim.
\end{proof}
We now have
$$\sum_{X\in\l(W,V)}q^{-4d\rank(X)}\|D_{X}[f]\|_{2}^{2} \leq \sum_{Y \in \l(W,V)} |\hat{f}(Y)|^2 \sum_{k=0}^{d} q^{-kd} \leq 2\|f\|_2^2,$$
completing the proof of the lemma.

\end{proof}
\begin{cor}
\label{cor:getting rid of D_X 4-norms}Let $f\colon\l(V,W)\to\mathbb{C}$
be a function of degree $d$. Then $\sum_{X\in\l(W,V)}q^{-4d\rank(X)}\|D_{X}[f]\|_{2}^{4}\le 2\|f\|_{2}^{4}.$ 
\end{cor}
\begin{proof}
We have $\|D_{X}[f]\|_{2}^{2}\le\|f\|_{2}^{2}$ and therefore $\|D_{X}[f]\|_{2}^{4}\le\|D_{X}[f]\|_{2}^{2}\cdot\|f\|_{2}^{2}$, so we are done by the preceding lemma. 
\end{proof}

\subsection{Applying the induction hypothesis to the functions $D_{X}D_{V_{2},W_{2},T}[f]$.}

\begin{lem}
\label{lem:upper bound on 4-norms of the derivatives given induction hypothesis}
Suppose that the inductive
hypothesis (\ref{eq:ind-hyp}) holds. Let $f:\l(V,W) \to \mathbb{C}$ be a function of degree at most $d$. Let $V_{2}\le V$, let $W_{2}\le W$ and let $X\in\l(W_{2},V/V_{2})$ such that $|(V_2,W_2,X)|: = \dim(V_2)+\codim(W_2)+\rank(X) \geq 1$. Let $V_{1}$ be the preimage of $\Image(X)$ under the natural projection map from $V$ to $V/V_2$, and let $W_{1} = \ker(X)$. Then

\begin{equation}
\underset{T \sim \l(V,W)}{\mathbb{E}}\|D_{X}\circ D_{V_{2},W_{2},T}f\|_{4}^{4}\le q^{100\left(d-\left|\left(V_{2},W_{2},X\right)\right|\right)^{2}+6d\rank(X)}\sum_{V_{3}\ge V_{2},W_{3}\le W_{2}:\atop V_{3}\cap V_{1}=V_{2},W_{3}+W_{1}=W_{2}}\underset{T \sim \l(V,W)}{\mathbb{E}}\|D_{X(W_{3},V/V_{3})}D_{V_{3},W_{3},T}f\|_{2}^{4}.\label{eq:technical}
\end{equation}
 
\end{lem}
\begin{proof}
By Lemmas \ref{lem:derivatives} and \ref{lem:derivatives-2}, the function $D_X \circ D_{V_2,W_2,T}f \in \l(V/V_1,W_1)$ has degree at most $d-|(V_2,W_2,X)| \leq d-1$, so by the induction hypothesis we have 
\begin{equation}\label{eq:halfway-house}
\|D_{X}\circ D_{V_{2},W_{2},T}f\|_{4}^{4}\le\sum_{V_4: V_1 \leq V_4 \leq V,\atop W_{4}\le W_{1}}q^{100(d-|(V_{2},W_{2},X)|)^2}\mathbb{E}_{S\sim\mathcal{L}\left(V/V_{1},W_{1}\right)}\|D_{V_{4}/V_1,W_{4},S} \circ D_{X} \circ D_{V_{2},W_{2},T}f\|_{2}^{4}.
\end{equation}
 for each $T \in \l(V,W)$. Note for later that the summand in (\ref{eq:halfway-house}) is zero unless 
 $$\dim(V_4)+\codim(W_4) \leq d.$$
 We now apply Proposition \ref{prop:switching the roles between the derivatives}.
For a subspace $V_4$ of $V$ with $V_1 \subseteq V_4$, let $\mathcal{V}\left(V_{4}\right)$ be the set of subspaces $V_{3}$ of $V$
such that $V_3 \supseteq V_2$ and $(V_{3}/V_2)\oplus (V_{1}/V_2)=V_{4}/V_2$, and for a subspace $W_4$ of $W_2$ with $W_4 \subseteq W_1$, let $\mathcal{W}\left(W_{4}\right)$
be the set of subspaces $W_3$ of $W_2$ such that $W_3 \supseteq W_4$ and $(W_{3}/W_{4})\oplus (W_{1}/W_{4})=W_2/W_{4}$. Using Proposition \ref{prop:switching the roles between the derivatives}, Proposition \ref{prop:composition of derivatives}, the triangle inequality and H\"{o}lder's inequality, we obtain 
\begin{align*}
\frac{\underset{T \sim \l(V,W)}{\mathbb{E}}\|D_{X}\circ D_{V_{2},W_{2},T}f\|_{4}^{4}}{q^{100\left(d-\left|\left(V_{2},W_{2},X\right)\right|\right)^{2}}} & \le\sum_{V_4: V_1 \leq V_4 \leq V,\atop W_{4}\le W_{1}}\underset{T \sim \l(V,W),\atop S \sim \l(V/V_1,W_1)}{\mathbb{E}}\left\Vert \sum_{V_{3}\in\mathcal{V}\left(V_{4}\right),\atop W_{3}\in\mathcal{W}\left(W_{4}\right)}D_{X\left(W_{3},V/V_{3}\right)}D_{V_{3}/V_2,W_{3},S(V/V_2,W_2)} D_{V_{2},W_{2},T}f\right\Vert _{2}^{4}\\
& = \sum_{V_4: V_1 \leq V_4 \leq V,\atop W_{4}\le W_{1}}\underset{T \sim \l(V,W),\atop S \sim \l(V/V_1,W_1)}{\mathbb{E}}\left\Vert \sum_{V_{3}\in\mathcal{V}\left(V_{4}\right),\atop W_{3}\in\mathcal{W}\left(W_{4}\right)}D_{X\left(W_{3},V/V_{3}\right)}D_{V_{3},W_{3},S(V,W)+T} f\right\Vert _{2}^{4}\\
& = \sum_{V_4: V_1 \leq V_4 \leq V,\atop W_{4}\le W_{1}}\underset{T \sim \l(V,W)}{\mathbb{E}}\left\Vert \sum_{V_{3}\in\mathcal{V}\left(V_{4}\right),\atop W_{3}\in\mathcal{W}\left(W_{4}\right)}D_{X\left(W_{3},V/V_{3}\right)}D_{V_{3},W_{3},T} f\right\Vert _{2}^{4}\\
 & \le\sum_{V_4: V_1 \leq V_4 \leq V,\atop W_{4}\le W_{1}}\underset{T \sim \l(V,W)}{\mathbb{E}}\left(\sum_{V_{3}\in\mathcal{V}\left(V_{4}\right),\atop W_{3}\in\mathcal{W}\left(W_{4}\right)}\left\Vert D_{X\left(W_{3},V/V_{3}\right)}D_{V_{3},W_{3},T}f\right\Vert _{2}\right)^{4}\\
 & \le\sum_{V_4: V_1 \leq V_4 \leq V,\atop W_{4}\le W_{1}}\left|\mathcal{V}\left(V_{4}\right)\right|^{3}\left|\mathcal{W}\left(W_{4}\right)\right|^{3}\underset{T \sim \l(V,W)}{\mathbb{E}}\sum_{V_{3}\in\mathcal{V}\left(V_{4}\right),\atop W_{3}\in\mathcal{W}\left(W_{4}\right)}\left\Vert D_{X\left(W_{3},V/V_{3}\right)}D_{
 V_{3},W_{3},T}f\right\Vert _{2}^{4}\\
 & \le \sum_{V_{3}\ge V_{2},W_{3}\le W_{2}:\atop V_{3}\cap V_{1}=V_{2},\,W_{3}+W_{1}=W_{2}} q^{6d\rank(X)} \underset{T \sim \l(V,W)}{\mathbb{E}}\|D_{X(W_{3},V/V_{3})}D_{V_{3},W_{3},T}f\|_{2}^{4},
\end{align*}
 where for the last inequality we use the observations
\[
|\mathcal{V}(V_{4})|\le\sqbinom{\dim(V_{4}/V_2)}{\dim(V_{1}/V_2)}_q = {\dim(V_4/V_2) \brack \rank(X)}_q \le q^{d\cdot\rank(X)},
\]
and (dually) $|\mathcal{W}(W_{4})|\leq q^{d\cdot\text{rank}\left(X\right)}$. This completes the proof of the lemma.
\end{proof}

Plugging Lemma \ref{lem:upper bound on 4-norms of the derivatives given induction hypothesis}
into Lemma \ref{lem:after restricting} we obtain: 
\begin{lem}
\label{lem:after restricting and switching roles} Let $f:\l(V,W) \to \mathbb{C}$ be a function
of degree at most $d$. Then 
\[
\frac{1}{162}\|f\|_{4}^{4}\le q^{6d^2} \|f\|_2^4+ q^{100d^2}\sum_{k=0}^{d} q^{-31d(k+1)}\sum_{V_3 \leq V,\ W_3 \leq W,\ Y \in \l(W_3,V/V_3):\atop \rank(Y)=k,\ |(V_3,W_3,Y)|>0}q^{-4kd} \underset{T \sim \l(V,W)}{\mathbb{E}}\| D_Y D_{V_3,W_3,T}[f]\|_2^4.
\]
\end{lem}
\begin{proof}
This follows fairly easily from Lemmas \ref{lem:after restricting} and \ref{lem:upper bound on 4-norms of the derivatives given induction hypothesis}. Indeed, these two lemmas together imply that
\begin{align*}
\frac{1}{162}\|f\|_{4}^{4} & \le q^{6d^{2}}\|f\|_{2}^{4}+\sum_{V_{2}\le V,\ W_{2}\le W,\atop X\in \l(W_{2},V/V_{2}),\, |(V_{2},W_{2},X)|>0}q^{24d|(V_2,W_2,X)|}\underset{T \sim \l(V,W)}{\mathbb{E}}\|D_{X}\circ D_{V_{2},W_{2},T}f\|_{4}^{4}\\
& \leq q^{6d^{2}}\|f\|_{2}^{4} +\\
& \sum_{V_{2}\le V,\ W_{2}\le W,\atop X\in \l(W_{2},V/V_{2}),\, |(V_{2},W_{2},X)|>0} q^{100\left(d-\left|\left(V_{2},W_{2},X\right)\right|\right)^{2}+24d|(V_2,W_2,X)|+6d\rank(X)}\sum_{V_{3}\ge V_{2},W_{3}\le W_{2}:\atop V_{3}\cap V_{1}=V_{2},W_{3}+W_{1}=W_{2}}\underset{T \sim \l(V,W)}{\mathbb{E}}\|D_{X(W_{3},V/V_{3})}D_{V_{3},W_{3},T}f\|_{2}^{4},
\end{align*}
where in the final sum, $W_1$ denotes the kernel of $X$ and $V_1$ the preimage (under the natural projection $V \to V/V_2$) of the image of $X$.

In order to complete the proof (by interchanging the order of summation), we only need show that for any fixed subspaces $V_3 \leq V$, $W_3 \leq W$, any fixed $Y \in \l(W_3,V/V_3)$ with $|(V_3,W_3,Y)| \leq d$, and any fixed $(i,j,k) \in \mathbb{N}^3$ with $0 < i+j+k \leq d$, the number of triples $(V_2,W_2,X)$ with
\begin{align*} & V_2 \leq V_3,\ W \geq W_2 \geq W_3,\ X \in \l(W_2,V/V_2),\ \dim(V_2)=i,\ \codim(W_2)=j,\ \rank(X)=k,\\
& V_3 \cap \mathcal{Q}_{V_2}^{-1}(\Image X) = V_2,\ W_3 + \ker(X) = W_2,\ X(W_3,V/V_3)=Y
\end{align*}
is at most $q^{3(i+j+k)d}$, where $\mathcal{Q}_{V_2}:V \to V/V_2$ is the natural quotient map, and that the above conditions force $\rank(X) = \rank(Y)$. This follows because firstly, there are at most $q^{i\dim(V_3)} \leq q^{id}$ choices for $V_2 \leq V_3$ with $\dim(V_2)=i$, and at most $q^{j\codim(W_3)} \leq q^{jd}$ choices for $W_2 \geq W_3$ with $\codim(W_2)=j$. Secondly, observe that $X \in \l(W_2,V/V_2)$ satisfying the above conditions must have $\ker(X) \supset \ker(Y)$; indeed, if $Yw = 0$ for some $w \in W_3$, then $Xw \in V_3/V_2$, but $Xw \in \Image(X)$ and $\mathcal{Q}_{V_2}^{-1}(\Image X) \cap V_3 = V_2$, so $Xw=0$. Moreover, since $W_3+\ker(X)=W_2$, we have $\Image(X(W_3,V/V_2)) = \Image(X)$, and since $V_3 \cap \mathcal{Q}_{V_2}^{-1}(\Image X) = V_2$, $X$ and $Y$ have the same rank; furthermore we have $\mathcal{Q}_{V_3/V_2}(\Image(X))=\Image(Y)$, so there are at most $q^{k(\dim(V_3)-\dim(V_2))} \leq q^{kd}$ choices for $
\Image(X)$. A dual argument shows that there are at most $q^{k(\codim(W_3)-\codim(W_2))} \leq q^{kd}$ choices for $\ker(X)$. Given $\Image(X)$ and $\ker(X)$, there are at most $q^{k^2} \leq q^{kd}$ choices for $X$, yielding the above bound.

Interchanging the order of summation, we therefore obtain
\begin{align*}
\frac{1}{162}\|f\|_{4}^{4} & \leq q^{6d^{2}}\|f\|_{2}^{4} + \sum_{i,j,k:\ 1 \leq i+j+k \leq d} q^{100(d-i-j-k)^2+24d(i+j+k)+6dk+3d(i+j+k)} \sum_{V_3 \leq V,\ W_3 \leq W,\ Y \in \l(W_3,V/V_3):\atop \rank(Y)=k,\ |(V_3,W_3,Y)|>0} \underset{T \sim \l(V,W)}{\mathbb{E}}\| D_Y D_{V_3,W_3,T}[f]\|_2^4\\
& \le q^{6d^{2}}\|f\|_{2}^{4} + q^{100d^2} \sum_{i,j,k:\ 1 \leq i+j+k \leq d} q^{-63d(i+j+k)}\sum_{V_3 \leq V,\ W_3 \leq W,\ Y \in \l(W_3,V/V_3):\atop \rank(Y)=k,\ |(V_3,W_3,Y)|>0}q^{-4kd} \underset{T \sim \l(V,W)}{\mathbb{E}}\| D_Y D_{V_3,W_3,T}[f]\|_2^4\\
& \le q^{6d^{2}}\|f\|_{2}^{4} + q^{100d^2}\sum_{k=0}^{d} q^{-31d(k+1)}\sum_{V_3 \leq V,\ W_3 \leq W,\ Y \in \l(W_3,V/V_3):\atop \rank(Y)=k,\ |(V_3,W_3,Y)|>0}q^{-4kd} \underset{T \sim \l(V,W)}{\mathbb{E}}\| D_Y D_{V_3,W_3,T}[f]\|_2^4\\
& \le q^{6d^{2}}\|f\|_{2}^{4} + q^{100d^2-31d} \sum_{V_3 \leq V,\ W_3 \leq W} \underset{T \sim \l(V,W)}{\mathbb{E}}\| D_{V_3,W_3,T}[f]\|_2^4,
\end{align*}
where for the last inequality we have applied Corollary \ref{cor:getting rid of D_X 4-norms} to remove the $D_Y$-derivatives.
\end{proof}

\subsection{Finishing the proof of our conditional hypercontractive inequality}
\begin{thm}[Conditional hypercontractive inequality for functions on $\l(V,W)$]
\label{thm:Hypercontractivity in the Bilinear scheme} If $V$ and $W$ are vector spaces over $\mathbb{F}_q$, and $f:\l(V,W)\to \mathbb{C}$ is
a function of degree at most $d$, then
$$\|f\|_{4}^{4}\le q^{100d^{2}}\sum_{V_{1}\le V,W_{1}\le W}\underset{\mathbf{T} \sim \l(V,W)}{\mathbb{E}}\left[\|D_{V_{1},W_{1},\mathbf{T}}[f]\|_{2}^{4}\right] = q^{100d^{2}}\sum_{V_{1}\le V,W_{1}\le W:\atop \dim(V)+\codim(W) \leq d}\underset{\mathbf{T} \sim \l(V,W)}{\mathbb{E}}\left[\|D_{V_{1},W_{1},\mathbf{T}}[f]\|_{2}^{4}\right].$$
\end{thm}
\begin{proof}
 It suffices to prove the inequality, since the equality is trivial. We prove the theorem by induction on $d$. Clearly, it holds when $d=0$, so let $d \geq 1$ and assume the inequality holds (with $d'$ in place of $d$) for all functions of degree at most $d'$, for all $d' < d$. Now let $f:\l(V,W)\to \mathbb{C}$ be of degree $d$. By Lemma \ref{lem:after restricting and switching roles}
we have 
\[
\frac{1}{162}\|f\|_{4}^{4}\le q^{6d^{2}}\|f\|_{2}^{4} + q^{100d^2-31d} \sum_{V_3 \leq V,\ W_3 \leq W}\underset{T \sim \l(V,W)}{\mathbb{E}}\| D_{V_3,W_3,T}[f]\|_2^4;
\]
multiplying through by 162, we obtain that the induction hypothesis holds for $f$, completing the proof of the induction step, and proving the theorem.
\end{proof}

\section{Equivalence between globalness and having small generalized influences. \label{sec:Equivalence-between-globalness notions}}

Our goal in this section is to show that if a Boolean function $f$
has a small density inside restrictions of dictators, then the pure-degree-$i$ part
$f^{=i}$ (for small $i$) has small generalized influences. 

\subsection{Combinatorial interpretation of the Laplacian}

We say that $f$ is \emph{homogeneous of degree $i$ if $f=f^{=i}.$
}While we do not have a nice combinatorial interpretation for the
Laplacians of a general function, we do have one in the case where
$f$ is homogeneous or nearly homogeneous.

The following lemmas give a combinatorial interpretation of the
Laplacian for homogeneous functions.
\begin{lem}
\label{prop:Combinatorial interpretation of the Laplacian} Let $U$
be either a 1-dimensional subspace of $V$ or a subspace of $W$ of
codimension 1, and let $i \in \mathbb{N} \cup \{0\}$. Then we have 
\[
L_{U}[f^{=i}]=f^{=i}-q^{i}\mathcal{E}_{U}[f^{=i}].
\]
\end{lem}
\begin{proof}
This is immediate from the Fourier formulae for $L_{U}$ and $\mathcal{E}_{U}$.
\end{proof}
We will also need the following.
\begin{lem}
Let $U$ be either a 1-dimensional subspace of $V$ or a subspace
of $W$ of codimension 1, and let $i \in \mathbb{N}$. Write
$\mathcal{T} = \mathcal{T}_{i,U}:L^2(\l(V,W)) \to L^2(\l(V,W))$ for the linear operator defined by 
\[
\mathcal{T}f:=f-(q^{i}+q^{i-1})\mathcal{E}_{U}[f]+q^{2i-1}\mathcal{E}_U^2[f]\quad \forall f \in L^2(\l(V,W)).
\]
 Then for all $f \in L^2(\l(V,W))$ we have
\[
L_{U}[f^{=i}]=(\mathcal{T}[f])^{=i}
\]
and
\[
L_{U}[f^{=i-1}]=(\mathcal{T}[f])^{=i-1}.
\]
\end{lem}
\begin{proof}
The lemma follows from   the Fourier formulas for $\mathcal{E}_{U}$ given in
Lemmas \ref{lem:Fourier formula for combinatorial Laplacian} and \ref{lem:lw formula}, together with Lemma \ref{prop:Combinatorial interpretation of the Laplacian}. Indeed, we have
\begin{align*}(\mathcal{T}[f])^{=i} & = f^{=i}-q^{i}\mathcal{E}_U[f^{=i}] - q^{i-1}\mathcal{E}_{U}(f^{=i}-q^i\mathcal{E}_U [f^{=i}])\\
& = L_U[f^{=i}]-q^{i-1}\mathcal{E}_U(L_U[f^{=i}])\\
& = L_U[f^{=i}]
\end{align*}
and
\begin{align*}(\mathcal{T}[f])^{=i-1} & = f^{=i-1}-q^{i-1}\mathcal{E}_U[f^{=i-1}] - q^{i}\mathcal{E}_{U}(f^{=i-1}-q^{i-1}\mathcal{E}_U [f^{=i-1}])\\
& = L_U[f^{=i-1}]-q^{i}\mathcal{E}_U(L_U[f^{=i-1}])\\
& = L_U[f^{=i-1}].
\end{align*}
\end{proof}
The following lemma shows how to relate the order-one derivatives of $f^{=i}$ to
the homogeneous parts of the restrictions of $\mathcal{T}[f]$. In what follows, if $U$ is a subspace of $V$ of dimension one or a subspace of $W_1$ of codimension one, and $T \in \l(V,W)$, by a slight abuse of notation we write $D_{U,T}$ for $D_{U,W,T}$ (if $U \leq V$) or for $D_{\{0\},U,T}$ (if $U \leq W$).
\begin{lem}
\label{lem:lv} Let $U$ be either a subspace of $V$ of dimension one, or a
subspace of $W$ of codimension one, and let $i \in \mathbb{N}$. Then for any $f \in L^2(\l(V,W))$ we have
\[
D_{U\to T}\left[f^{=i}\right]=\left(\left(\mathcal{T}_{i,U}[f]\right)_{U\to T}\right)^{=i-1}.
\]
\end{lem}
\begin{proof}
We consider just the case where $U\subseteq V$ is of dimension 1
as the other case is dual. The space $H_{=i}$ of homogeneous functions
of degree $i$ is $\mathcal{T}_{i,U}$-invariant and the restriction $g \mapsto g_{(U,W)\to T}$
sends $H_{=j}$ to $H_j+H_{=j-1}$ for all $j\in \mathbb{N}$. Hence, we have 
\begin{align*}
(\mathcal{T}_{i,U}[f]_{U\to T})^{=i-1} & =((\mathcal{T}_{i,U}[f^{=i}+f^{=i-1}])_{U\to T})^{=i-1}\\
 & =((L_{U}[f^{=i}+f^{=i-1}])_{U\to T})^{=i-1}\\
 & =\left(D_{U\to T}[f^{=i}]\right)^{=i-1}+\left(D_{U,T}\left[f^{=i-1}\right]\right)^{=i-1}\\
 & =D_{U\to T}[f^{=i}],
\end{align*}
where the last equality uses Lemma \ref{lem:derivatives}. 
\end{proof}

\subsection{Laplacians and restrictions}

\begin{defn}Recall that we say a function $f:\l(V,W)\to \mathbb{C}$ is {\em $(d,\epsilon)$-restriction global} if $\|f_{(V_{1},W_{1})\to T}\|_{2}^{2}\le\epsilon$
for each $V_{1} \leq V$ and $W_{1} \leq W$ with $\dim\left(V_{1}\right)+\text{codim}\left(W_{1}\right)\le d$, and each $T \in \l(V,W)$.
We say that $f$ has {\em $\left(d,\epsilon\right)$-small generalized influences}
if $I_{V_{1},W_{1},T}[f]\le\epsilon$ for each $V_{1} \leq V$ and $W_{1} \leq W$
with $\dim\left(V_{1}\right)+\codim\left(W_{1}\right)\le d$,
and each $T\in\mathcal{L}\left(V,W\right)$.
\end{defn}

Our next goal is to obtain an analogue of Lemma \ref{lem:restriction global implies global},
saying that restriction-globalness implies that $f^{=d}$ has small
generalized influences.
\begin{prop}
\label{prop:influences measure globalness} Let $f:\l(V,W) \to \mathbb{C}$ and let $d\in \mathbb{N}$, $\epsilon>0$. Suppose that $f$ is $(d,\epsilon)$-restriction global, i.e.\ that $\|f_{(V_{1},W_{1})\to T}\|_{2}^{2}\le\epsilon$ for
all $V_{1}\subseteq V$ and $W_{1}\subseteq W$ with $\dim(V_1)+\codim(W_1) \leq d$. Then $I_{V_{1},W_{1},T}[f^{=d}]\le q^{10d^{2}}\epsilon$ for all $V_1 \leq V$ and $W_1 \leq W$ with $\dim(V_1)+\codim(W_1) \leq d$, and all $T \in \l(V,W)$.
\end{prop}
This proposition will follow from applying the following lemma repeatedly
and decomposing each derivative $D_{V_{1},W_{1},T}$
as a composition of order-one derivatives. 
\begin{lem}
\label{lem:inductive step}
Let $d \in \mathbb{N}$, let $\epsilon>0$ and let $f\colon\mathcal{L}\left(V,W\right)\to\mathbb{C}$
be $(d,\epsilon)$-restriction global. Let $D_{U,T}$ be an order-one
derivative for $f$. Then there exists a $\left(d-1,4q^{4d}\epsilon\right)$-restriction
global function $f'\colon\mathcal{L}\left(V',W'\right)\to\mathbb{C}$ (where $(V',W') = (V/U,W)$ if $U \subseteq V$ and $(V',W') = (V,U)$ if $U \subseteq W$)
such that $(f')^{=d-1}=D_{U,T}\left[f^{=d}\right]$. 
\end{lem}
\begin{proof}
Assume without loss of generality that $U\subseteq V$ with $\dim(U)=1$. We set $f'=(\mathcal{T}_{d,U}\left[f\right])_{(U,W) \to 0}$.
Let $V \geq V_{2}\geq U$ and $W_{2} \leq W$ such that $\dim\left(V_{2}\right)+\codim\left(W_{2}\right)\le d$. We have
\begin{align}
\left(\mathcal{T}_{d,U}\left[f\right]_{(U,W)\to 0}\right)_{\left(V_{2}/U,W_{2}\right)\to T}&=(\mathcal{T}_{d,U}[f])_{(V_2,W_2) \to T}\nonumber \\
&= f_{\left(V_{2},W_{2}\right)\to T}-\left(q^{d}+q^{d-1}\right)\left(\mathcal{E}_{U}\left[f\right]\right)_{\left(V_{2},W_{2}\right)\to T}+q^{2d-1}\left(\mathcal{E}_{U}^{2}\left[f\right]\right)_{\left(V_{2},W_{2}\right)\to T}.
\label{eq:expanding T_i,U}
\end{align}
 Now $\|f_{\left(V_{2},W_{2}\right)\to T}\|_{2}^{2}\le\epsilon.$
Taking $V'$ to be a uniformly random subspace of $V$ of codimension $1$ with $V'+U=V$, and $B$ to be a uniformly random element of $\l(V/V',W)$, we have by convexity that
\begin{align*}
\|\left(\mathcal{E}_{U}\left[f\right]\right)_{\left(V_{2},W_{2}\right)\to T}\|_{2} & =\|\mathbb{E}_{B}\left[\Delta_{B(V,W)}f\right]_{\left(V_{2},W_{2}\right)\to T}\|_{2}\\
 & \le\mathbb{E}_{B}\|\left(\Delta_{B(V,W)}f\right)_{\left(V_{2},W_{2}\right)\to T}\|_{2}\\
 & =\mathbb{E}_{B}\|f_{\left(V_{2},W_{2}\right)\to T+B(V,W)}\|_{2}\\
 & \le \sqrt{\epsilon}.
\end{align*}
 Repeating the above argument for $\mathcal{E}_{U}\left[f\right]$ shows that
\[
\|(\mathcal{E}_{U} \circ \mathcal{E}_U [f])_{\left(V_{2},W_{2}\right)\to T}\|_{2}^{2}\le\epsilon.
\]
 Therefore applying the triangle inequality to (\ref{eq:expanding T_i,U}) yields that
 $$\left\|\left(\mathcal{T}_{d,U}\left[f\right]_{(U,W)\to 0}\right)_{\left(V_{2}/U,W_{2}\right)\to T}\right\|_2 \leq 4q^{4d}\epsilon.$$
Hence, $(\mathcal{T}_{d,U}[f])_{(U,W) \to 0}$ is $\left(d-1,4q^{4d}\epsilon\right)$-restriction
global, as required.
\end{proof}
\begin{proof}[Proof of Proposition \ref{prop:influences measure globalness}]
 The proof is by induction on $d$. Trivially, the statement of the proposition holds when $d=0$. Let $d \in \mathbb{N}$, and assume the statement of the proposition holds with $d-1$ in place of $d$; we will obtain it for $d$. Let $f:\l(V,W) \to \mathbb{C}$ be $(d,\epsilon)$-restriction global. Let $V_1 \leq V$ and $W_1 \leq W$ with $\dim(V_1)+\codim(W_1) = d$; without loss of generality we may assume that $\dim(V_1) \geq 1$. Let $U$ be a one-dimensional subspace of $V_1$; then we may write 
\[
\|D_{V_{1},W_{1},T}\left[f^{=d}\right]\|_{2}^{2}=\|D_{V_{1}/U,W_{1}}D_{U,W,T}\left[f^{=d}\right]\|_{2}^{2}.
\]
Now we may apply Lemma \ref{lem:inductive step} to conclude that 
\[
\|D_{V_{1}/U,W_{1}}D_{U,W,T}\left[f^{=d}\right]\|_{2}^{2}=\|D_{V_{1}/U,W_{1}}\left(f'\right)^{=d-1}\|_{2}^{2}
\]
 for a $\left(d-1,4q^{4d}\epsilon \right)$-restriction global function $f':\l(V/U,W) \to \mathbb{C}$. By the induction hypothesis, we obtain
\[
\|D_{V_{1}/U,W_{1}}\left(f'\right)^{=d-1}\|_{2}^{2}\le q^{10\left(d-1\right)^{2}}4q^{4d}\epsilon\le q^{10d^{2}}\epsilon,
\]
 completing the inductive step, and proving the proposition.
\end{proof}

\section{Applications of our conditional hypercontractive inequality}

Our first application is an upper bound on the $4$-norm of a low-degree function having small generalized influences.
\begin{cor}
\label{cor:Corollary of hypercontractivity} Suppose that $f:\l(V,W) \to \mathbb{C}$ is
a function of degree at most $d$ that has $\left(d,\epsilon\right)$-small
generalized influences. Then 
\[
\|f\|_{4}^{4}\le q^{103d^{2}}\epsilon\|f\|_{2}^{2}.
\]
 
\end{cor}
\begin{proof}
By Theorem \ref{thm:Hypercontractivity in the Bilinear scheme}, we
have 
\begin{align*}
\|f\|_{4}^{4} & \le q^{100d^{2}}\sum_{V_{1}\le V,W_{1}\le W:\atop \dim(V_1)+\codim(W_1) \leq d}\mathbb{E}_{T\sim\mathcal{L}\left(V,W\right)}\|D_{V_{1},W_{1},T}[f]\|_{2}^{4}\\
 & \le q^{100d^{2}}\sum_{V_{1}\le V,W_{1}\le W:\atop \dim(V_1)+\codim(W_1) \leq d}\epsilon \cdot \mathbb{E}_{T\sim\mathcal{L}\left(V,W\right)}\|D_{V_{1},W_{1},T}[f]\|_{2}^{2}\\
 & = \epsilon \cdot q^{100d^{2}}\sum_{V_{1}\le V,W_{1}\le W:\atop \dim(V_1)+\codim(W_1) \leq d}\|L_{V_{1},W_{1}}[f]\|_{2}^{2}\\
 & =q^{100d^{2}}\epsilon\sum_{X:\rank(X) \leq d}\hat{f}\left(X\right)^{2}\#\left\{ (V_{1},W_1):\ \dim(V_1)+\codim(W_1) \leq d,\ V_{1}\subseteq \Image\left(X\right),W_{1}\supseteq X^{-1}(V_1)\right\} \\
 & \leq q^{100d^{2}}\epsilon\sum_{X:\rank(X) \leq d}\hat{f}\left(X\right)^{2}\#\left\{ (V_{1},W_1):\ \dim(V_1)+\codim(W_1) \leq d,\ V_{1}\subseteq\text{image}\left(X\right),W_{1}\supseteq \ker(X)\right\}\\
 & = q^{100d^{2}}\epsilon\sum_{X:\rank(X) \leq d}\hat{f}\left(X\right)^{2} \sum_{0 \leq i+j \leq d}\#\left\{ (V_{1},W_1):\ \dim(V_1)=i,\ \codim(W_1) =j,\ V_{1}\subseteq\text{image}\left(X\right),W_{1}\supseteq \ker(X)\right\}\\
 & \leq q^{100d^{2}}\epsilon\sum_{X:\rank(X) \leq d}\hat{f}\left(X\right)^{2} \sum_{0 \leq i+j \leq d} q^{d(i+j)}\\ 
 & \le q^{103d^{2}}\cdot \epsilon \cdot \|f\|_{2}^{2},
\end{align*}
as required. (Here, we have used the fact that for $X$ of rank at most $d$ and for integers $i,j \geq 0$, there are at most $q^{d(i+j)}$ pairs $(V_1,W_1)$ with $\dim(V_1)=i,\ \codim(W_1)=j$, $V_1 \leq \Image(X)$ and $W_1 \geq \ker(X)$.)
\end{proof}
We may conclude from the above that restriction-global Boolean functions are
concentrated on the high degrees.
\begin{cor}[Level-$d$ inequality]
\label{cor:concentration on high degrees} Suppose that $f\colon\mathcal{L}\left(V,W\right)\to\{0,1\}$
is $\left(d,\epsilon\right)$-restriction global. Then 
\[
\|f^{=d}\|_{2}^{2}\le q^{30d^{2}}\epsilon^{\frac{1}{4}}\|f\|_{2}^{2}.
\]
 
\end{cor}
\begin{proof}
By Proposition \ref{prop:influences measure globalness}, $f^{=d}$
has $\left(d,q^{10d^{2}}\epsilon\right)$-small generalized influences.
By Corollary \ref{cor:Corollary of hypercontractivity} we then have 
\[
\|f^{=d}\|_{4}^{4}\le q^{113d^{2}}\epsilon\|f\|_{2}^{2}.
\]
 Hence, using H\"older's inequality and the fact that $f$ is Boolean, we obtain
\[
\|f^{=d}\|_{2}^{2}=\left\langle f^{=d},f\right\rangle \le\|f^{=d}\|_{4}\|f\|_{\frac{4}{3}}\le q^{30d^2}\epsilon^{\frac{1}{4}}\|f\|_{2}^{2}.
\]
\end{proof}

In a similar vein, we have the following, which is a Bonami-type lemma for global functions on $\l(V,W)$.
\begin{cor}
\label{cor:atmostd}
Suppose that $f:\l(V,W) \to \{0,1\}$ is $\left(d,\epsilon\right)$-restriction global. Then 
\[
\|f^{\leq d}\|_{4}^{4}\le q^{115d^{2}}\epsilon\|f^{\leq d}\|_{2}^{2}.
\]
\end{cor}
\begin{proof}
 By Proposition \ref{prop:influences measure globalness}, $f^{=d}$ has $(d,q^{10d^2}\epsilon)$-small generalised influences, and since $f$ is also $(i,\epsilon)$-restriction global for all $i <d$, the function $f^{=i}$ has $(d,q^{10i^2}\epsilon)$-small generalised influences for all $i < d$. It follows from the linearity of $D_{V_1,W_1,T}$, the triangle inequality and Cauchy-Schwarz that for any subspaces $V_1\leq V$ and $W_1 \leq W$ with $\dim(V_1)+\codim(W_1) \leq d$, and any $T \in \l(V,W)$, we have
 \begin{align*}I_{V_1,W_1,T}[f^{\leq d}] & = \|D_{V_1,W_1,T}[f^{\leq d}]\|_2^2\\
 &\leq \left(\sum_{i=0}^{d}\|D_{V_1,W_1,T}[f^{=i}]\|_2\right)^2\\
 & \leq (d+1) \sum_{i=0}^{d} \|D_{V_1,W_1,T}[f^{=i}]\|_2^2\\
 & \leq (d+1) \sum_{i=0}^{d}q^{10i^2}\epsilon\\
 & \leq q^{12d^2}\epsilon.
 \end{align*}
 Hence, $f^{\leq d}$ has $(d,q^{12d^2}\epsilon)$-small generalized influences. We are now done by Corollary \ref{cor:Corollary of hypercontractivity}.
\end{proof}
Corollary \ref{cor:atmostd} yields Theorem \ref{cor:KLLM for BS} with $C=115$. 

\subsection{Applications to small-set expansion}
\begin{lem}
Let $G$ be the shortcode graph on $\mathcal{L}\left(V,W\right)$, i.e.\ the Cayley graph of $\l(V,W)$ where
two matrices $A_1$ and $A_2$ are adjacent if $A_1-A_2$ has rank one. Let $T$ be the (normalised) adjacency operator corresponding to the graph $G$, i.e.\ $T$ is defined by
$$Tf(A) = \underset{B:\ \rank(B)=1}{\mathbb{E}}f(A+B)\quad \forall A \in \l(V,W),\ f \in L^2(\l(V,W)),$$
where the expectation is over a uniformly random $B \in \l(V,W)$ of rank one. Then for any $f : \l(V,W) \to \mathbb{C}$, we have
\[
Tf=\sum_{X \in \l(W,V)}\lambda_{\text{rank}\left(X\right)}\hat{f}\left(X\right)u_{X},
\]
 where 
\[
\lambda_{d}:=\frac{q^{-d}-\frac{1}{\left|W\right|}}{1-\frac{1}{\left|W\right|}}
\]
for each $0 \leq d \leq \dim(W)$.
 
\end{lem}
\begin{proof}
Note that we have 
\[
Tf=\mathbb{E}_{V'}\left(\frac{\mathfrak{e}_{V/V'}[f]-\frac{1}{\left|W\right|}f}{1-\frac{1}{|W|}}\right),
\]
where the expectation is taken over a uniformly random subspace $V' \leq V$ with $\codim(V')=1$ (one just needs to remove the identically-zero linear map from the average defining $\mathfrak{e}_{V/V'}$). Hence, by Lemma \ref{lem:EVV'}, we have 
\begin{align*}
Tf & = \frac{1}{1-\frac{1}{|W|}}\left( \mathbb{E}_{V'} \sum_{X \in \l(W,V)}1\{\Image(X) \subset V'\} \hat{f}(X)u_X - \frac{1}{|W|} \sum_{X \in \l(W,V)} \hat{f}(X)u_X\right)\\ 
& =\frac{1}{1-\frac{1}{|W|}}\sum_{X \in \l(W,V)}\left(\Pr_{V'}\left[\text{image}\left(X\right)\subseteq V'\right]-\frac{1}{|W|}\right)\hat{f}\left(X\right)u_{X}\\
 & =\sum_{X \in \l(W,V)}\left(\frac{q^{-\text{rank\ensuremath{\left(X\right)}}}-\frac{1}{\left|W\right|}}{1-\frac{1}{\left|W\right|}}\right) \hat{f}(X)u_X,
\end{align*}
as required.
\end{proof}
\begin{thm}[Small-set expansion in the shortcode graph]
\label{thm:sses}
There exists an absolute constant $C_0$ such that the following holds. Let $r\in \mathbb{N}$, and let $S\subseteq\mathcal{L}\left(V,W\right)$
be a family of linear maps with $1_{S}$ being $\left(r+1,q^{-C_0r^2}\right)$-restriction
global. Then 
\[
\Pr_{A\sim S,\ B\text{ of rank }1}\left[A+B\in S\right] < q^{-r},
\]
where the probability is over a uniform random member $A$ of $S$ and a uniform random linear map $B \in \l(V,W)$ of rank one.
\end{thm}
\begin{proof}
We have 
\[
\Pr_{A\sim S,\ B\text{ of rank }1}\left[A+B\in S\right]=\frac{\left\langle T1_{S},1_{S}\right\rangle }{\|1_{S}\|_{2}^{2}}.
\]
Let $n=\dim(W)$. By the preceding lemma, we have
\[
\left\langle T1_{S},1_{S}\right\rangle =\sum_{d=0}^{n}\frac{q^{-d}-\frac{1}{\left|W\right|}}{1-\frac{1}{\left|W\right|}}\|1_{S}^{=d}\|_{2}^{2}.
\]
Crudely, we have 
\[
\sum_{d=r+2}^{n}\frac{q^{-d}-\frac{1}{\left|W\right|}}{1-\frac{1}{\left|W\right|}}\|1_{S}^{=d}\|_{2}^{2}<\left(\sum_{d=r+2}^{n}q^{-d}\right)\|1_S\|_2^2 < \frac{q^{-r}}{2}\|1_S\|_{2}^{2}.
\]
For $d\leq r+1$, we have (by Corollary \ref{cor:concentration on high degrees}, applied to $f=1_S$)
\[
\|(1_S)^{=d}\|_{2}^{2}\le q^{-C_0 r^2/4}q^{30d^{2}}\|1_S\|_{2}^{2}\le\frac{q^{-r}}{2\left(r+2\right)}\|1_S\|_{2}^{2},
\]
 provided $C_0$ is a sufficiently large absolute constant. Combining the two prior inequalities yields 
\[
\left\langle T1_{S},1_{S}\right\rangle < q^{-r}\|1_S\|_{2}^{2},
\]
 completing the proof. 
\end{proof}
Theorem \ref{thm:sses} yields Theorem \ref{thm:ssesc}, with a constant $C_1$ slightly larger than one.
The Inverse Shortcode Hypothesis (with sharp quantitative parameter-dependence, in a certain sense) is an immediate corollary of the $q=2$ case of Theorem \ref{thm:ssesc}.
\begin{cor}[Inverse Shortcode Hypothesis]
For each $\eta>0$, there exist $C>0$ and $\delta >0$ such that
the following holds. Let $V$ and $W$ be vector spaces over $\mathbb{F}_2$. Let $S\subseteq\mathcal{L}\left(V,W\right)$
be a set of matrices with $1_{S}$ being $\left(C,\delta\right)$-restriction
global. Then 
\[
\Pr_{A\sim S,\ B\text{ of rank }1}\left[A+B\in S\right] < \eta.
\]
\end{cor}

As mentioned in the Introduction, Barak, Kothari and Steurer \cite{bks} showed that the Inverse Shortcode Hypothesis implies the Grassmann Soundness Hypothesis (Hypothesis 1.2 in \cite{bks}), which (in combination with the theorem of Dinur, Khot, Kindler, Minzer and Safra in \cite{dkkms}), implies the 2-to-2 Games conjecture (with imperfect completeness).

\appendix

\section{Some technical facts from linear algebra.}
\begin{lem}
\label{Lem:Traces} Let $V_{1}\subseteq V$ and $W_{1}\subseteq W.$
Let $A\in\mathcal{L}\left(V/V_{1},W_{1}\right)$ and $X\in\mathcal{L}\left(W,V\right).$
Then 
\[
\Tr\left(A\cdot X\left(W_{1},V/V_{1}\right)\right)=\Tr\left(A(V,W)\cdot X\right).
\]
\end{lem}
\begin{proof}
Write $W=W_{1}\oplus W_{2}$ and $V=V_1 \oplus V_2$, and let $\pi_{1}$ and $\pi_{2}$ be the projections
to $W_1$ and $W_2$ respectively. Let $Y = X(W_1,V/V_1)$. Firstly, we observe
that the map $B:=A(V,W)X$ agrees with $A\cdot Y$ on $W_1$, since $\ker(A(V,W)) \supseteq V_1$. Indeed, if $w \in W_1$, then write $Xw = v_1+v_2$ where $v_1 \in V_1$ and $v_2 \in V_2$; we have $Bw = A(V,W)Xw = A(V,W)(v_1+v_2)= A(V,W)v_2 = A(v_2+V_1) = AYw$.
Secondly, we observe that the trace of $B$ is the sum of the traces
of the restricted projected maps: $\pi_{1}B(W_{1},W)$ and $\pi_{2} B(W_{2},W)$.
The lemma now follows from the fact that $\pi_{1}B(W_{1},W)=\pi_1 A\cdot Y$,
and that $\pi_{2}B(W_{2},V)=0$ (as the image of $B$ lies entirely
within $W_{1}$).
\end{proof}

\begin{lem}
\label{lem:Establishing that the compositition of derivatives behaves nicely}
Let $X\in\l(W,V)$. Let $V_{2}\le V_{1}\le V$, and let $W_{1}\le W_{2}\le W$.
Then the following are equivalent.
\begin{enumerate}
\item $\Image(X)\supseteq V_{1}$ and $X^{-1}(V_{1})\subseteq W_{1}$.
\item We have (a) $\Image(X)\supseteq V_{2}$ and (b) $X^{-1}(V_{2})\subseteq W_{2}$.
Setting $Y=X\left(W_{2},V/V_{2}\right)$, we have (c) $\Image(Y)\supseteq V_{1}/V_{2}$
and (d) $Y^{-1}(V_{1}/V_{2})\subseteq W_{1}$.
\end{enumerate}
\end{lem}

\begin{proof}
Let us first prove that (1) implies (2). Suppose that (1) holds. Since $\Image(X)$ contains $V_{1}$,
it also contains $V_{2}$, proving item (a) of (2). Item (b) follows from the chain
of inclusions 
\[
X^{-1}(V_{2})\subseteq X^{-1}(V_{1})\subseteq W_{1}\subseteq W_{2}.
\]
We now prove (c). Since $X^{-1}(V_1) \subseteq W_1 \subseteq W_2$, and $\Image(X)\supseteq V_{1}$, we have $\Image(X(W_2,V)) \supseteq V_1$ and therefore $\Image(Y) \supseteq V_1/V_2$, as required. To prove (d), just note that 
\[
Y^{-1}(V_{1}/V_{2}) \subseteq X^{-1}(V_{1}) \subseteq W_{1}.
\]

We now show that (2) implies (1). Suppose that (2) holds. Then $\Image(X)\supseteq V_{2}$
and 
\[
X(W_{2})/V_{2}=\Image(Y)\supseteq V_{1}/V_{2},
\]
implying that $\Image(X)\supseteq V_{1}$.

It remains only to show that $X^{-1}(V_{1})\le W_{1}$. Let $w\in X^{-1}(V_{1})$, i.e.\ $Xw \in V_1$. We must show that $w\in W_{1}$. Since $W_1 \subseteq W_2$, we obtain that $Xw+V_2$ is in the image of $Y$, and it is also in $V_1/V_2$ by hypothesis. Since $Y^{-1}(V_1/V_2) \subset W_1$, there exists $w_1 \in W_1$ such that $Yw_1 = Xw+V_2$, so $Xw_1 \in Xw+V_2$ and therefore $X(w-w_1) \in V_2$. But $X^{-1}(V_2) \subset W_2$ and therefore $w-w_1 \in W_2$, so $w \in W_1$ as required.
\end{proof}

\begin{lem}
\label{lem:restrictiond of compositions of Laplacians}
Let $Y\in \l(W,V)$
be a linear map with $\ker(Y) \subset W_{1}$ and
$\Image(Y) \supset V_{1}$. Then there exists a unique triple $(W_2,V_2,X)$ such that $W \geq W_{2}\geq W_{1}$, $V_{2}\leq V_{1} \leq V$,
$X\in \l(W_{2},V_{1}/V_{2})$ is surjective with kernel $W_1$,
$Y^{-1}(V_{2})\subseteq W_{2}$ and $X\le Y(W_{2},V/V_{2})$.
\end{lem}
\begin{proof}
We start by proving the existence part of the claim. As usual, we will identify between a surjective linear map $X:W_2 \to U$ with kernel $W_1$, and the unique linear isomorphism $\tilde{X}$ from $W_2/W_1$ to $U$ satisfying $\tilde{X}(w+W_1) = X(w)$ for all $w \in W_2$.

\subsection*{The existence part of the claim}

Let $W_{2}=W_{1}+Y^{-1}(V_{1})$, and let $V_{2}=Y(W_{1})\cap V_{1}$. Note that
$$Y^{-1}(V_2) \subseteq Y^{-1}(Y(W_1)) = W_1 \subseteq W_2,$$
where the equality $Y^{-1}(Y(W_1))=W_1$ holds since if $w \in W$ with $Yw = Yw_1$ for some $w_1 \in W_1$, then $w-w_1 \in \ker(Y) \subseteq W_1$ so $w \in W_1$.

Let $Y_{2}=Y(W_{2},V/V_{2})$. Note that since $Y(W_1 \cap Y^{-1}(V_1)) \subset Y(W_1) \cap V_1 = V_2$, the restriction of the map $Y_{2}\colon W_{2}\to V/V_{2}$
to the subspace $Y^{-1}(V_{1})$ sends $W_{1}\cap Y^{-1}(V_{1})$ to
$0=V_2$, and therefore induces a linear map 
\[
\tilde{X}\colon Y^{-1}(V_{1})/(W_{1}\cap Y^{-1}(V_{1}))\to V_{1}/V_{2}.
\]
Since the image of $Y$ contains $V_1$, the linear map $\tilde{X}$ is clearly surjective. It is also injective, since if $w \in W$ with $Yw \in V_2 := Y(W_1)\cap V_1$, then $w \in Y^{-1}(Y(W_1)) \cap Y^{-1}(V_1) = W_1 \cap Y^{-1}(V_1)$. Hence, $\tilde{X}$ is a linear
isomorphism. Since $Y^{-1}(V_1)/(W_1 \cap Y^{-1}(V_1))$ is naturally isomorphic to $(Y^{-1}(V_1)+W_1)/W_1 = W_2/W_1$, we may equivalently view $\tilde{X}$ as a linear isomorphism from $W_2/W_1$ to $V_1/V_2$. We let $X \in \l(W_2,V_1/V_2)$ be the corresponding linear surjection with kernel $W_1$. Explicitly, for $w \in W_2$ with $w = w_1+z$ ($w_1 \in W_1$, $z \in Y^{-1}(V_1)$), we define $X(w) = \tilde{X}(z+W_1 \cap Y^{-1}(V_1))$.

To finish the proof of the existence part of the claim, we must show that
$X\le Y_{2}$. Since $\mathrm{Ker}(Y_2)=W_{1}\cap Y^{-1}(V_{1})$, and
since $W_{1}+Y^{-1}(V_{1})=W_{2}$, by the rank-nullity formula we have 
\begin{align*}
\mathrm{rank}(Y_{2}) & =\dim(W_2) - \dim(\ker(Y_2))\\
& = \dim(W_1 + Y^{-1}(V_1)) - \dim(\ker(Y_2))\\
&= \dim(W_1)+\dim(Y^{-1}(V_1)) - \dim(W_1 \cap Y^{-1}(V_1))\\
&= \dim(W_{1})+\dim(Y^{-1}(V_{1}))-2\dim(\mathrm{Ker}(Y_{2})).\end{align*}
On the other hand, we have
\[
\rank(X)=\dim(Y^{-1}(V_{1}))-\dim(\mathrm{Ker}(Y_{2})).
\]
By construction, $Y_2$ agrees with $X$ on $Y^{-1}(V_1)$, and therefore
$$\rank(Y_2-X) \leq \dim(W_2) - \dim(Y^{-1}(V_1)) = \dim(W_1) - \dim(W_1 \cap Y^{-1}(V_1)) = \dim(W_1) - \dim(\ker(Y_2)),$$
so
$$\rank(Y_2-X) \leq \dim(W_{1})-\dim(\mathrm{Ker}(Y_2)).$$
Hence, we have $\rank(Y_2) \geq \rank(X)+\rank(Y_2-X)$, so in fact $\rank(Y_2) = \rank(X)+\rank(Y_2-X)$, and therefore $X\le Y_2$, as required.

\subsection*{The uniqueness part of the claim}

Let $V_{2}\le V_{1}$ and let $W_{2}\ge W_{1}$ be such that $Y^{-1}(V_{2})\le W_{2}$, write $Y_2 = Y(W_2,V/V_2)$, and let $X$ be a linear surjection from $W_2$ to $V_{1}/V_{2}$ with kernel $W_1$, and with $X\le Y_2$. We must show that $X,V_{2}$ and $W_{2}$
are the same as obtained in the existence part of the claim. Viz., we must show that 
\[
V_{2}=Y(W_{1})\cap V_{1},\quad W_{2}=W_{1}+Y^{-1}(V_{1}),
\]
and that $X$ is the linear surjection obtained above. Since $X\le Y_{2}$,
we obtain 
\[
V_{1}/V_{2}=\Image(X)\subseteq \Image(Y_{2}),
\]
and by the previous proposition, $X$ agrees with $Y_{2}$ on 
\[
Y_{2}^{-1}(\Image(X))=Y_{2}^{-1}(V_{1}/V_{2}).
\]
Now since $\ker(X) = W_{1}$, we obtain that the subspace $W_{1}\cap Y_{2}^{-1}(V_{1}/V_2)$
is the kernel of the map $Y(Y_{2}^{-1}(V_{1}/V_2),V/V_{2})$. So $Y$
induces a linear isomorphism 
\[
\tilde{X'}\colon Y_{2}^{-1}(V_{1}/V_2)/(W_{1}\cap Y_{2}^{-1}(V_{1}/V_2))\to V_{1}/V_{2}.
\]
Since $Y_2^{-1}(V_1/V_2)/(W_1 \cap Y_2^{-1}(V_1/V_2))$ is naturally isomorphic to $(Y_2^{-1}(V_1/V_2)+W_1)/W_1$, we may alternatively view $\tilde{X'}$ as a linear isomorphism between $(W_{1}+Y_{2}^{-1}(V_{1}/V_2))/W_{1}$
and $V_{1}/V_{2}$. However, we have $W_{1}+Y_{2}^{-1}(V_{1}/V_2)\le W_{2}$ and
$W_{2}/W_{1}$ is also isomorphic to $V_{1}/V_{2}$ by hypothesis,
and therefore $W_{2}=Y_{2}^{-1}(V_{1}/V_2)+W_{1}$. So far, we have
$$Y^{-1}(V_{1})+W_{1}\supseteq Y_{2}^{-1}(V_{1}/V_2)+W_{1} = W_{2}.$$
We now note that in fact, $Y^{-1}(V_{1}) = Y_2^{-1}(V_1/V_2)$. It is clear that the first set contains the second; we claim the second also contains the first. Indeed, suppose $w \in Y^{-1}(V_1)$; then $Yw \in V_1$. Since $\Image(Y_2) \supseteq V_1/V_2$, there exists $w_2 \in W_2$ such that $Yw_2 \in Yw+V_2$. But then $Y(w-w_2) \in V_2$ and therefore $w-w_2 \in W_2$. It follows that $w \in W_2$ and therefore $w \in Y_2^{-1}(V_1/V_2)$, as required. Hence,
$$W_2 = Y^{-1}(V_{1})+W_{1}.$$
We now wish to show that
$$V_2 = Y(W_1) \cap V_1.$$
First we show that $V_2 \subset Y(W_1) \cap V_1$. Clearly, $V_2 \subset V_1$, so it suffices to show that $V_2 \subset Y(W_1)$. Let $v_2 \in V_2$. Since $\Image(Y) \supset V_1 \supset V_2$, there exists $w \in W$ such that $Yw = v_2$. Since $Y^{-1}(V_2) \subset W_2$, we have $w \in W_2$. Since $w \in \ker(Y(Y_2^{-1}(V_1/V_2),V/V_2))$, we have $w \in W_1$ as required. Now we show that $V_2 \supset Y(W_1) \cap V_1$. Indeed, suppose that $v_1 \in Y(W_1) \cap V_1$; then there exists $w_1 \in W_1$ such that $Yw_1 = v_1 \in V_1$. But then $w_1 \in \ker(Y(Y_2^{-1}(V_1/V_2),V/V_2))$, so $Yw_1 =v_1 \in V_2$, as required.

It follows that $\tilde{X}'$ is precisely the map $\tilde{X}$ in the proof of the existence part of the claim. This completes the proof of the uniqueness part of the claim.
\end{proof}

\begin{lem}
\label{Lem: swapping places 2} Let $V_{1}\le V_{2}\le V$, $W_{2}\le W_{1}\le W$,
let $X\in\l(W,V)$ be linear map whose kernel is $W_{1}$ and whose image
is $V_{1}$, and let $Y\in\l(W,V)$ be a linear map. Suppose that $W \geq W_{3}\ge W_{2}$
and $V_{3}\le V_{2} \leq V$, 
\[
W_{3}/W_{2}\oplus W_{1}/W_{2}=W/W_2,
\]
\[
V_{1}\oplus V_{3}=V_{2},
\]
\[
\Image(Y) \supseteq V_3,\quad Y^{-1}(V_{3})\subseteq W_{3}
\]
and 
\[
X(W_{3},V/V_{3})\le Y(W_{3},V/V_{3}).
\]
\
Then $X\le Y$, 
\[
\Image(Y(W_{1},V/V_{1}))\supseteq V_{2}/V_1
\]
and 
\[
Y(W_{1},V/V_{1})^{-1}(V_{2})\subseteq W_{2}.
\]
Moreover, 
\[
V_{3}=\Image(Y-X)\cap V_{2}
\]
and 
\[
W_{3}=\mathrm{Ker}(Y-X)+W_{2}.
\]
\end{lem}
\begin{proof}
\textbf{Showing that $\Image(X)\le\Image(Y)$.} We have 
\[
\Image(X(W_{3},V/V_{3}))\le \Image(Y(W_{3},V/V_{3}))
\]
and therefore
\[
X(W_{3})+V_{3}\le Y(W_{3})+V_{3}.
\]
Now 
\[
W_{3}+W_{1}=W,
\]
and $W_1 = \ker(X)$, so $\Image(X) = X(W_3)$. Moreover, $V_{3}\leq Y(W_3)$. Thus, $$V_1 = \Image(X) = X(W_3) \leq X(W_3)+V_3 \leq Y(W_3)+V_3 \le Y(W_3) \leq \Image(Y),$$
as required. Note that since $X(W_3)+V_3 = V_1+V_3=V_2$, we showed along the way that $Y(W_3) \supseteq V_2$.

\textbf{Showing that $X$ and $Y$ agree on $Y^{-1}(\Image(X))$.}
Let $w \in W$ with $Yw\in V_{1}$. We need to show that $Xw=Yw$. Since
\[
\Image(Y(W_{3},V))\supseteq V_{2},
\]
there exists $w_{3}\in W_{3}$ such that $Yw_3 = Yw$, and therefore $Y(w-w_3)=0$. Since $Y^{-1}(V_3) \subset W_3$ It follows that $w-w_3 \in W_3$ and therefore $w \in W_3$. Hence, $$Y(W_3,V/V_3)w = Yw+V_3 \in \Image(X(W_3,V/V_3));$$
since $X(W_3,V/V_3) \leq Y(W_3,V/V_3)$, these two maps agree on $w$ and therefore $Yw-Xw \in V_3$. But $Yw-Xw \in V_1$ as well, and $V_1 \cap V_3 = \{0\}$, so $Xw=Yw$, as required.

\textbf{Showing that $V_{3}=\Image(Y-X)\cap V_{2}$ and that $W_{3}=\mathrm{Ker}(Y-X)+W_{2}$.}
We prove it by showing that 
\[
\mathrm{rank}(X(W_{3},V/V_{3}))=\mathrm{rank}(X),
\]
and that 
\[
\mathrm{rank}(Y(W_{3},V/V_{3}))=\mathrm{rank}(Y)-\dim(V_{3})-\mathrm{codim}(W_{3}).
\]
This will allow us to deduce that 
\[
\mathrm{rank}(Y-X)(W_{3},V/V_{3})=\mathrm{rank}(Y-X)-\dim(V_{3})-\mathrm{codim}(W_{3}).
\]
Now the above equality holds if and only if $\Image(Y-X)\supseteq V_{3}$
and $(Y-X)^{-1}(V_{3})\subseteq W_{3}.$ Hence, $\Image(Y-X)\cap V_{2}\supseteq V_{3}$,
and $\mathrm{Ker}(Y-X)+W_{2}\le W_{3}$. It is easy to see that these inclusions are actually equalities. Indeed, since $X \leq Y$ we have $(\Image(Y-X) \cap V_2) \cap (\Image(X) \cap V_2) = \{0\}$ and therefore $\dim(\Image(Y-X) \cap V_2) \leq \dim(V_2)-\dim(\Image(X) \cap V_2) = \dim(V_2)-\dim(V_1)=\dim(V_3)$ and therefore $\Image(Y-X) \cap V_2 = V_3$. Similarly, we have $\ker(Y-X)+W_2 = W_3$. Indeed, since
$$X(W_3) \leq X(W_3)+V_3 \leq Y(W_3)+V_3 = Y(W_3),$$
and since $Y(W_3,V)$ and $X(W_3,V)$ agree on $Y(W_3,V)^{-1}(\Image(X))$ (using that $X \leq Y$), we have $X(W_3,V) \leq Y(W_3,V)$; it follows that 
$$\ker(X(W_3,V))+\ker((Y-X)(W_3,V)) = \ker(Y(W_3,V))$$
and therefore
$$\ker(Y-X) + W_2 = (\ker(Y-X))\cap W_3 + \ker(X) \cap W_3 = W_3.$$

\textbf{Showing that 
\[
\mathrm{rank}(X(W_{3},V/V_{3}))=\mathrm{rank}(X).
\]
} We have 
\[
\Image(X(W_{3},V/V_{3}))=(X(W_{3})+V_{3})/V_3.
\]
Now 
\[
X(W_{3})=X(W_{3}+\ker(X))=X(W_{3}+W_{1})=\Image(X)=V_{1}.
\]
Hence,
\[
\rank(X(W_3,V/V_3)) = \dim(V_{1}\oplus V_{3})/V_{3}=\dim(V_{1})=\mathrm{rank}(X).
\]
\textbf{Showing that 
\[
\mathrm{rank}(Y(W_{3},V/V_{3}))=\mathrm{rank}(Y)-\dim(V_{3})-\mathrm{codim}(W_{3}).
\]
}
Since $\Image(Y) \supseteq V_3$ and $Y^{-1}(V_{3})\subseteq W_{3}$, this follows from Lemma \ref{Lem:Linear algebraic}.

Since $X(W_{3},V/V_{3})\le Y(W_{3},V/V_{3})$, we obtain 
\begin{align*}
\mathrm{rank}((Y-X)(W_{3},V/V_{3}))&=\rank(Y)-\dim(V_{3})-\mathrm{codim}(W_{3}) - \rank(X)\\
&=\mathrm{rank}(Y-X)-\dim(V_{3})-\mathrm{codim}(W_{3}),\end{align*}
as required.
\textbf{Showing that $\Image(Y(W_{1},V/V_{1}))\ge V_{2}/V_{1}$.}
It suffices to show that $Y(W_{1})+V_{1}\ge V_{2}$. Let $v_{3}\in V_{3}$;
we show that $v_{3}\in Y(W_{1})$. Since $\Image(Y) \supseteq V_3$ and $Y^{-1}(V_{3})\subseteq W_{3}$,
there exists $w_{3}\in W_{3}$ such that $Yw_{3}=v_{3}$. Hence,
\[
w_{3}\in\mathrm{Ker}(Y(W_{3},V/V_{3}))\le\mathrm{Ker}(X(W_{3},V/V_{3}));
\]
since $\Image(X)\cap V_3 = \{0\}$ we obtain $Xw_{3}=0$, and therefore $w_{3}\in W_{1}$. Thus, $v_{3}\in Y(W_{1})$, as required.

\textbf{Showing that $Y(W_{1},V/V_{1})^{-1}(V_{2}/V_1)\le W_{2}$.} Let
$w\in W_{1}$ be with $Yw\in V_{2}$; we show that $w\in W_{2}$. It suffices to show that $w \in W_3$. Write $Yw = v_1+v_3$ where $v_1 \in V_1$ and $v_3 \in V_3$. Since $\Image(Y) \supseteq V_3$ and $Y^{-1}(V_3) \subset W_3$, there exists $w_3 \in W_3$ such that $Yw_3=v_3$. Hence, $Y(w-w_3)=v_1 \in V_1 = \Image(X)$, so $X$ and $Y$ agree on $w-w_3$, so $w-w_3 \in \ker(Y-X) \subset W_3$, and therefore $w \in W_3$, as required.

This completes the proof of the lemma. 
\end{proof}
\begin{lem}
\label{lem:switching roles between derivatives lin-alg part} Let
$V_{1}\le V_{2}\le V$, and let $W_{2}\le W_{1}\le W.$ Suppose
$X\in\l(W,V)$ is a map whose kernel is $W_{1}$ and whose image is
$V_{1}$, and let $Y\in\l(W,V)$. Suppose that $X\le Y$,
that $V_{2}/V_{1}\le\Image(Y(W_{1},V/V_{1}))$, and that $Y(W_{1},V/V_{1})^{-1}(V_{2}/V_{1})\le W_{2}$.

Then there exist $V_{3}\leq V_{2}$ and $W_{3}\geq W_{2}$,
such that $V_{3}\oplus V_{1}=V_{2}$, $(W_{3}/W_{2})\oplus (W_{1}/W_{2})=W/W_{2}$,
$X(W_{3},V/V_{3})\le Y(W_{3},V/V_{3})$, $\Image(Y) \supseteq V_3$ and $Y^{-1}(V_3) \subset W_3$.
\end{lem}
\begin{proof}
Set $V_{3}=\Image(Y-X)\cap V_{2}$, and $W_{3}=W_{2}+\mathrm{Ker}(Y-X)$. Clearly, we have $\Image(Y) \supseteq V_3$.

\textbf{Showing that $V_{3}\oplus V_{1}=V_{2}$.} Since $\Image(X)\cap\Image(Y-X)=\{0\}$,
we have $V_{3}\cap V_{1}=\{0\}.$ Now let $v\in V_{2}$. Since 
\[
\Image(Y(W_{1},V/V_{1}))\supseteq V_{2}/V_1,
\]
and since 
\[
Y(W_{1},V/V_{1})^{-1}(V_{2}/V_1)\subseteq W_{2},
\]
we may write $v=Yu + v_1 = (Y-X)u+Xu+v_1$ for some $u \in W_2$ and $v_1 \in V_1$. Now $v_1,Xu\in V_{1}\subseteq V_{2}$ and therefore $(Y-X)u$ is also in $V_{2}$. Hence, 
\[
v\in V_{1}+(V_{2}\cap\Image(Y-X))=V_{1}+V_{3},
\]
as required.

\textbf{Showing that $W_{3}/W_{2}\oplus W_{1}/W_{2}=W/W_{2}$.} Since
\[
\mathrm{Ker}(X)+\mathrm{Ker}(Y-X)=W,
\]
we have $W_{1}+W_{3}=W$. It remains to show that $W_{3}\cap W_{1}=W_{2}$. Clearly, $W_3 \cap W_1 \supseteq W_2$; we must show the reverse inclusion.
Let $w\in W_{3}\cap W_{1}$. Then $w$ is in the kernel of $X$ and
we may write $w=w_{2}+u$, where $w_{2}\in W_{2}$ and $u \in \ker(Y-X)$. Since $W_{2}\subseteq W_{1} = \ker(X)$,
this implies that $u\in W_{1}=\ker(X)$. Hence
\[
u\in \ker(X) \cap \ker(Y-X) = \ker(Y)\subseteq W_{1}\cap Y^{-1}(V_{2})= Y(W_{1},V/V_{1})^{-1}(V_{2}/V_1)\subseteq W_{2},
\]
so $v \in V_2$ as well. This completes the proof that 
\[
W/W_{2}=W_{1}/W_{2}\oplus W_{3}/W_{2}.
\]

\textbf{Showing that $Y^{-1}(V_{3})\le W_{3}$.} Let $w \in W$ with $Yw\in V_{3}$.
Then $Yw\in V_{3}\le V_{2}$ and therefore there exists $w_{2}\in W_{2}$
with $Y(W_{1},V/V_{1})w_{2}=Yw+V_{1}$. Hence, 
\[
Y(w-w_{2})\in V_{1}=\Image(X).
\]
Since $X\le Y$, we have $w-w_{2}\in\mathrm{Ker}(Y-X)$, proving that
$w\in W_{2}+\mathrm{Ker}(Y-X).$

\textbf{Showing that $X(W_{3},V/V_{3})\le Y(W_{3},V/V_{3})$.} Since
$V_{3}$ has trivial intersection with the image of $X$, and since 
\[
\mathrm{Ker}(X(W_{3},V/V_{3}))=\mathrm{Ker}(X)\cap W_{3}=W_{2},
\]
we have 
\[
\mathrm{rank}(X(W_{3},V/V_{3}))=\mathrm{rank}(X)+\dim(W_{3}/W_{2})-\dim(W/W_{1})=\mathrm{rank}(X).
\]
On the other hand, by Lemma \ref{Lem:Linear algebraic} and the facts that $\Image(Y) \supseteq V_3$ and $Y^{-1}(V_3) \subset W_3$, we have
\[
\mathrm{rank}(Y(W_{3},V/V_{3}))=\mathrm{rank}(Y)-\dim(V_{3})-\mathrm{codim}(W_{3}).
\]
Clearly, we have $\Image(Y-X) \supseteq V_3$; we claim that also $(Y-X)^{-1}(V_3) \subset W_3$. Indeed, suppose that $w \in W$ with $(Y-X)w = v_3 \in V_3$; since $V_3 \subset \Image(Y)$ and $Y^{-1}(V_3) \subset W_3$, there exists $w_3 \in W_3$ such that $Yw_3 = v_3$. Hence, $(Y-X)w=Yw_3$, so $Y(w-w_3) = Xw \in \Image(X)$, and therefore $w-w_3 \in Y^{-1}(\Image(X)) \subset \ker(Y-X) \subset W_3$; it follows that $w\in W_3$. Hence, we have
\[
\mathrm{rank}((Y-X)(W_{3},V/V_{3}))= \rank(Y-X) - \dim(V_{3})-\mathrm{codim}(W_{3}).
\]
Thus, 
\[
X(W_{3},V/V_{3})\leq Y(W_3,V/V_3),
\]
as required.
\end{proof}

\end{document}